\documentclass{amsart}
\usepackage{fullpage}
\usepackage[utf8]{inputenc}
\usepackage[english]{babel}
\usepackage{amsmath, amsthm, amsfonts, amssymb}
\usepackage{pifont}
\usepackage{enumitem}
\usepackage{tikz,tikz-cd}
\usetikzlibrary{arrows}
\usepackage{longtable}
\usepackage{pdflscape}
\usepackage{hyperref}
\usepackage{appendix}
\usepackage{leftindex}
\usepackage{cleveref}
\usepackage{listofitems}
\usepackage{csquotes}
\usepackage[backend=biber,style=alphabetic,sorting=nyt,doi=false,url=false,isbn=false]{biblatex}
\crefname{appsec}{appendix}{appendices}
\Crefname{appsec}{Appendix}{Appendices}
\crefname{subsection}{algorithm}{algorithms}
\Crefname{subsection}{Algorithm}{Algorithms}
\crefname{remark}{remark}{remarks}
\Crefname{remark}{Remark}{Remarks}

\usetikzlibrary{babel,decorations.pathreplacing}

\newcommand{\ds}{\displaystyle}

\newcommand{\ov}{\overline}
\newcommand{\ti}{\widetilde}

\newcommand{\M}{\mathcal{M}}
\renewcommand{\O}{\mathcal{O}}
\renewcommand{\P}{\mathbb{P}}
\newcommand{\A}{\mathbb{A}}
\newcommand{\Z}{\mathbb{Z}}
\newcommand{\Q}{\mathbb{Q}}
\newcommand{\C}{\mathbb{C}}

\newcommand{\G}{\mathbb{G}}
\newcommand{\F}{\mathbb{F}}

\newcommand{\Aut}{\operatorname{Aut}}
\newcommand{\Spec}{\operatorname{Spec}}

\newcommand{\p}[1]{{\left( #1 \right)}}
\newcommand{\floor}[1]{\left\lfloor #1 \right\rfloor}
\newcommand{\ceil}[1]{\left\lceil #1 \right\rceil}
\newcommand{\chr}{\operatorname{char}}
\newcommand{\e}{\varepsilon}
\newcommand{\GL}{\operatorname{GL}}

\newcommand{\SL}{\operatorname{SL}}
\newcommand{\smallpmatrix}[1]{\p{\begin{smallmatrix}#1\end{smallmatrix}}}

\newcommand{\ord}{\operatorname{ord}}
\newcommand{\Ext}{\operatorname{Ext}}
\newcommand{\sExt}{\operatorname{\mathcal{E}xt}}

\newcommand{\id}{\operatorname{id}}

\newcommand{\Def}{\operatorname{Def}}

\newcommand{\li}{\leftindex}
\newcommand{\rank}{\operatorname{rank}}

\theoremstyle{plain}
\newtheorem{theorem}{Theorem}[section]
\theoremstyle{definition}
\newtheorem{definition}[theorem]{Definition}
\AddToHook{env/definition/begin}{\crefalias{theorem}{definition}}
\newtheorem{example}[theorem]{Example}
\AddToHook{env/example/begin}{\crefalias{theorem}{example}}
\AtBeginEnvironment{example}{\pushQED{\qed}}
\AtEndEnvironment{example}{\popQED\endexample}
\newtheorem{proposition}[theorem]{Proposition}
\AddToHook{env/proposition/begin}{\crefalias{theorem}{proposition}}
\newtheorem{lemma}[theorem]{Lemma}
\AddToHook{env/lemma/begin}{\crefalias{theorem}{lemma}}

\AddToHook{env/algorithm/begin}{\crefalias{theorem}{algorithm}}
\newtheorem{corollary}[theorem]{Corollary}
\AddToHook{env/corollary/begin}{\crefalias{theorem}{corollary}}
\newtheorem{remark}[theorem]{Remark}
\AddToHook{env/remark/begin}{\crefalias{theorem}{remark}}

\numberwithin{equation}{theorem}

\newcommand{\singularfiber}[8][]{
    \begin{center}\begin{tikzpicture}
        \readlist\dataoverzero{#3}
        \readlist\dataoverinfty{#7}
        \draw (0,3) -- (0,0);
        \node[left] at (-0.2,2.1) {#2};
        
        \ifnum\dataoverzero[1]>0
        \draw (-0.2,2.6) -- (0.8,2.8);
        \ifnum\dataoverzero[2]>1
        \draw (0.4,2.8) -- (1.4,2.6);
        \ifnum\dataoverzero[2]>2
        \draw (1,2.6) -- (2,2.8);
        \fi
        \fi
        \fi
        \ifnum\dataoverzero[1]>1
        \draw (-0.2,2.4) -- (0.8,2.6);
        \ifnum\dataoverzero[2]>1
        \draw (0.4,2.6) -- (1.4,2.4);
        \ifnum\dataoverzero[2]>2
        \draw (1,2.4) -- (2,2.6);
        \fi
        \fi
        \fi
        \ifnum\dataoverzero[2]<2
        \ifnum\dataoverzero[1]>1
        \draw[decorate,decoration = brace] (0.9,2.9) -- (0.9,2.3);
        \node[right] at (1,2.6) {#4};
        \fi\ifnum\dataoverzero[1]=1
        \node[right] at (1,2.7) {#4};
        \fi
        \fi
        \ifnum\dataoverzero[2]=2
        \ifnum\dataoverzero[1]>1
        \draw[decorate,decoration = brace] (1.5,2.9) -- (1.5,2.3);
        \node[right] at (1.6,2.6) {#4};
        \fi\ifnum\dataoverzero[1]=1
        \node[right] at (1.6,2.7) {#4};
        \fi
        \fi
        \ifnum\dataoverzero[2]>2
        \ifnum\dataoverzero[1]>1
        \draw[decorate,decoration = brace] (2.1,2.9) -- (2.1,2.3);
        \node[right] at (2.2,2.6) {#4};
        \fi\ifnum\dataoverzero[1]=1
        \node[right] at (2.2,2.7) {#4};
        \fi
        \fi

        \ifnum#5=1
        \draw (-0.2,1.5) -- (0.8,1.5);
        \fi\ifnum#5=2
        \draw (-0.2,1.6) -- (0.8,1.6);
        \draw (-0.2,1.4) -- (0.8,1.4);
        \draw[decorate, decoration = brace] (0.9,1.7) -- (0.9,1.3);
        \fi\ifnum#5=3
        \draw (-0.2,1.7) -- (0.8,1.7);
        \draw (-0.2,1.5) -- (0.8,1.5);
        \draw (-0.2,1.3) -- (0.8,1.3);
        \draw[decorate, decoration = brace] (0.9,1.8) -- (0.9,1.2);
        \fi\ifnum#5=4
        \draw (-0.2,1.8) -- (0.8,1.8);
        \draw (-0.2,1.6) -- (0.8,1.6);
        \draw (-0.2,1.4) -- (0.8,1.4);
        \draw (-0.2,1.2) -- (0.8,1.2);
        \draw[decorate, decoration = brace] (0.9,1.9) -- (0.9,1.1);
        \fi\ifnum#5>4
        \draw (-0.2,1.8) -- (0.8,1.8);
        \node[] at (0.4,1.6) {$\vdots$};
        \draw (-0.2,1.2) -- (0.8,1.2);
        \draw[decorate, decoration = brace] (0.9,1.9) -- (0.9,1.1);
        \fi
        \ifnum#5>0
        \node[right] at (1,1.5) {#6};
        \fi

        \ifnum\dataoverinfty[1]>1
        \draw (-0.2,0.4) -- (0.8,0.6);
        \ifnum\dataoverinfty[2]>1
        \draw (0.4,0.6) -- (1.4,0.4);
        \ifnum\dataoverinfty[2]>1
        \draw (1,0.4) -- (2,0.6);
        \fi
        \fi
        \fi
        \ifnum\dataoverinfty[1]>0
        \draw (-0.2,0.2) -- (0.8,0.4);
        \ifnum\dataoverinfty[2]>1
        \draw (0.4,0.4) -- (1.4,0.2);
        \ifnum\dataoverinfty[2]>2
        \draw (1,0.2) -- (2,0.4);
        \fi
        \fi
        \fi
        \ifnum\dataoverinfty[2]<2
        \ifnum\dataoverinfty[1]>1
        \draw[decorate,decoration = brace] (0.9,0.7) -- (0.9,0.1);
        \node[right] at (1,0.4) {#8};
        \fi\ifnum\dataoverinfty[1]=1
        \node[right] at (1,0.3) {#8};
        \fi
        \fi
        \ifnum\dataoverinfty[2]=2
        \ifnum\dataoverinfty[1]>1
        \draw[decorate,decoration = brace] (1.5,0.7) -- (1.5,0.1);
        \node[right] at (1.6,0.4) {#8};
        \fi\ifnum\dataoverinfty[1]=1
        \node[right] at (1.6,0.3) {#8};
        \fi
        \fi
        \ifnum\dataoverinfty[2]>2
        \ifnum\dataoverinfty[1]>1
        \draw[decorate,decoration = brace] (2.1,0.7) -- (2.1,0.1);
        \node[right] at (2.2,0.4) {#8};
        \fi\ifnum\dataoverinfty[1]=1
        \node[right] at (2.2,0.3) {#8};
        \fi
        \fi

        \if\relax\detokenize{#1}\relax\else
        \node[right] at (4.2,1.5) {#1};
        \fi
    \end{tikzpicture}\end{center}
}

\title{Stacky Framework for the Classification of Singular Fibers in Fibered Surfaces}
\date{\today}

\author{Javier Reyes Bastidas}
\thanks{}
\address[JR]{Department of Mathematics, University of Maryland, College Park MD 20742, USA}
\email{jereyes@umd.edu}

\begin{document}

    \begin{abstract}
        We develop a strategy for the classification of singular fibers in smooth fibered surfaces by exploiting the fact that, in the tame case, a fibration with smooth general fiber of genus $g > 1$ corresponds to a representable morphism from a stacky base to $\ov\M_g$ thanks to the root stack valuative criterion for properness of tame stacks. This framework serves as another justification of Dokchitser's classification of reduction types independently of Winters' abstract result.
        
    \end{abstract}
    
    \maketitle
    
    \tableofcontents

    \section{Introduction}

    The problem of the classification of singular fibers in fibrations has been extensively studied and has had good answers for over half a century. In \cite{Kod} and \cite{Neron1964-su}, Kodaira and Néron famously classified the singular fibers in genus 1 fibrations, showing us that, other than smooth fibers, there exist two infinite families of singular fibers together with seven exceptional ones. Some years later, \cite{Ogg} and \cite{Namikawa1973} tackled the problem in genus two. Ogg found 44\footnote{49 counting variations, some of which were missed.} combinatorial families, 27 of which are infinite; and Namikawa-Ueno refined the classification to 120\footnote{The discrepancies are explained by differences in what each author considers as a family.} families, 90 of which infinite. In \cite{ARTIN1971} it was shown that for any given genus, there can only be a finite number of families. Later, Winters shows in \cite{Winters1974} that over $\C$ any ``combinatorially admissible'' singular fiber indeed exists as a fiber in a fibration with possibly non connected fibers. His proof is not effective, as one cannot devise an algorithm for listing what these fibers are in a given genus just with this result alone. In a much more recent paper, Dokchitser shows in \cite{Dok25} an effective classification of singular fibers in any genus, by realizing that they can be subdivided into principal components connected by inner and outer chains with heavy restrictions.

    In this paper we give a different approach for an effective classification of the possible singular fibers in fibered surfaces $X \to B$, which relies on the root stack valuative criterion for tame stacks \cite[Theorem 3.1]{Bresciani2023}. We essentially show that under a tameness assumption, such a fibration is birational to a fibration on stable curves $\mathcal{X} \to \mathcal{B}$ between tame Deligne-Mumford stacks corresponding to a map $\mathcal{B} \to \ov\M_g$. We also show how to recover the original fibration $X \to B$ from this map.

    The approach of studying maps from a stacky curves $\mathcal{B}$ to $\ov\M_g$ to understand fibered surfaces traces back to \cite{Abramovich2000} where they study the problem of compactifying the space of such surfaces. Not much later in \cite{Abramovich2001}, the same authors develop the theory for compactifying the space of maps from stacky curves to arbitrary tame Deligne-Mumford stacks. The application of this approach to the classification of singular fibers was inspired by \cite{Ascher2018,bejleri2024}, where they are able to recover the Kodaira classification in genus 1 from a representable map $\mathcal{B} \to \M_{1,1}$.

    In an upcoming paper \cite{Re26-2} we will specialize this classification when the fibration $X \to B$ has a smooth hyperelliptic curve as a geometric generic fiber. We believe this approach should also help in the analogous classification when the generic fiber is superelliptic. Moreover, following \cite{Ascher2018}, we also expect there to be explicit descriptions of the compactified moduli spaces of these different kinds of fibrations.

    \ 

    \ 

    The set up is as follows. We fix an algebraically closed field $K$. Let $X \to B$ be a proper morphism between a smooth surface $X$ and a smooth curve $B$, everything over $K$. Assume that the geometric generic fiber is a smooth curve of genus $g \geq 2$, and that $X \to B$ is minimal, in the sense that no fibers contain $(-1)$-curves. Then this fibration induces a rational map $B \dashrightarrow \ov\M_g$ defined at the very least on the locus $B^{sm}$ of $B$ where the geometric fibers are smooth. Since $B$ is smooth, we have the diagram
    \[\begin{tikzcd}
        & \ov\M_g \ar[d] \\
        B \ar[ur,dashed] \ar[r] & \ov{M}_g,
    \end{tikzcd}\]
    and assuming that $\ov\M_g \to \ov M_g$ is tame along the image of the indeterminacy of $B$ (a condition that we later partially relax), \cite[Theorem 3.1]{Bresciani2023} tells us there is a root stack $\mathcal{B} \to B$ and a representable lift $\mathcal{B} \to \ov \M_g$ such that the diagram
    \[\begin{tikzcd}
        \mathcal{B} \ar[d] \ar[r] & \ov\M_g \ar[d] \\
        B \ar[ur,dashed] \ar[r] & \ov{M}_g,
    \end{tikzcd}\]
    is 2-commutative. Pulling back the universal family $\mathcal{U} \to \ov\M_g$ we obtain a fibration on stable curves $\mathcal{X} \to \mathcal{B}$. Let $\mathcal{X} \to X^{tw}$ be its coarse moduli space, which must induce another fibration $X^{tw} \to B$ that is flat, for instance, by miracle flatness or by \cite[Lemma 3.2]{Abramovich2000}. As we will see later, the surface $X^{tw}$ has at worst cyclic quotient singularities, but must be isomorphic to $X$ above $B^{sm}$. We call $X^{tw}$ the twisted model of $X$.

    Since $X$ is minimal smooth, we must have a commutative diagram
    \[\begin{tikzcd}
        & \ti X \ar[dr] \ar[dl]\\
        X^{tw} \ar[dr] & & X \ar[dl]\\
        & B
    \end{tikzcd}\]
    where $\ti X \to X^{tw}$ is a resolution of singularities and $\ti X \to X$ is a contraction of $(-1)$-curves. As a consequence of \Cref{equivalence-of-principal-under-stability}, it turns out that $\ti X$ is the minimal normal crossings model of $X$, that is, for any $b \in B$, $X_{b,red}$ is normal crossings, and any $(-1)$-curve in $X_b$ intersects other components of $X_b$ in at least three points.

    With this framework, we are able to translate the classification of ``tame'' minimal or minimal normal crossing fibrations $X \to B$ to the classification of representable maps from tame smooth stacky curves $\mathcal{B}$ to $\ov \M_g$. The translation from the minimal and minimal normal crossings models is straightforward, but the latter is much simpler to deal with in practice. This approach is, in principle, the same moduli-oriented one that \cite{Namikawa1973} used to justify the classification of genus 2 singular fibers. Here we are able to translate it to the language of stacks, and show that it works in every genus.

    As for the local picture, let us explain the theoretical background from which an algorithm can be constructed. We still assume we are in the tame case as before, but now set $B$ to be the disc $\Spec K[[t]]$. This way, $\mathcal{B} = [\Spec K[[t]]/\mu_n]$ for some $n$ coprime to $\chr K$, and where the action $\mu_n \curvearrowright \Spec K[[t]]$ is given by $\zeta_n \colon t \mapsto \zeta_n t$. Consider the diagram
    \[\begin{tikzcd}
        X^{st} \ar[d] \ar[r] & \mathcal{X} \ar[d] \ar[r] & \mathcal{U} \ar[d]\\
        B \ar[r,"{\pi'}"] \ar[dr,"\pi",swap] & {[B/\mu_n]} \ar[d] \ar[r] & \ov\M_g \ar[d] \\
        & B \ar[r] & \ov{M}_g
    \end{tikzcd}\]
    where $\pi$ and $\pi'$ are the categorical and stacky quotients of $B$ respectively. Here $X^{st} \to B$ is a fibration on stable curves which carries an equivariant $\mu_n$-action, whose quotient must be $X^{tw} \to B$. Let $C \subseteq X^{st}$ be the stable fiber over $0$, and $f \in \Aut(C)$ be the induced automorphism from the generator $\zeta_n$ of the $\mu_n$-action. If $p \in C$ is a smooth point which is fixed by a power $f^d$ of $f$, then $f^d$ acts on $T_pC$ as multiplication by some root of unity $\zeta_{n/d}^a$. Since we have an equivariant decomposition of the tangent space at $p$,
    \[T_pX^{st} \cong T_0B \oplus T_pC,\]
    we deduce that the corresponding point in the quotient $X^{tw}$ must be a cyclic quotient singularity $\frac{1}{n/d}(1,a)$. Instead, if $p \in C$ is a node, then $p \in X^{st}$ is not necessarily a smooth point, as it can be a singularity of type $A_{k-1}$ for some $k \geq 1$. Still, $f^d$ will act on the two dimensional tangent space $T_pC$ as some matrix
    \[\begin{pmatrix}
        \zeta_{n/d}^a & 0\\ 0 & \zeta_{n/d}^b
    \end{pmatrix}
    \qquad \text{ or } \qquad
    \begin{pmatrix}
        0 & \zeta_{n/d}^a \\ \zeta_{n/d}^b & 0
    \end{pmatrix}
    \]
    Similarly to when $p \in C$ is smooth, we see in \Cref{section:cyclic-quotients-of-Ak} that this matrix together with the data of the singularity type $A_{k-1}$ is enough to describe the image of $p$ in $X^{tw}$.

    This argument shows that the classification of (tame) singular fibers in surfaces fibered on curves of genus $g \geq 2$, reduces to the classification of cyclic actions $\mu_n \curvearrowright C$ for all stable curves $C$ of genus $g$.

    In \Cref{section:glueing-building-blocks} we explain how, at least combinatorially, arbitrary singular fibers can be obtained by glueing ``building blocks'', similarly to how stable curves are obtained by glueing smooth curves.

    In \Cref{section:existence-of-fibration-with-action} we prove a fundamental result to complete the classification. Namely, that given any stable curve $C$, any automorphism $f \in \Aut(C)$ of order coprime to $\chr K$, and any suitable choice of data $A_{k-1}$ at the nodes of $C$, there is a fibration on stable curves $X^{st} \to B$ with central fiber $C$, having the given singularity types at the nodes, and which admits an equivariant $\mu_n$-action whose restriction to the central fiber $C$ is $f$.

    In \Cref{section:classification-of-all-fibers} we draw the parallel with earlier work in \cite{Dok25} and using their combinatorial analysis together with our approach, we provide in \Cref{existence-of-geometrically-connected-fiber} another proof of \cite[Corollary 4.3]{Winters1974}. We also show in \Cref{semi-tame-valuative-criterion} a slight strengthening of \cite[Theorem 3.1]{Bresciani2023} showing that the valuative criterion of properness can be witnessed by a root stack even if the target stack is not tame, as long as the source admits a tamely ramified cover with a lift.

    Let us list two main theorems in this paper. We are purposefully vague as to not introduce all the necessary notation here, but to keep in mind, ``building blocks'' are the simplest kind of singular fibers, obtained as equivariant quotients of $X^{st} \to B$ where the central fibers happen to be smooth.

    \begin{theorem}[See \Cref{global-glueing-algorithm}]
        Let $\ti X_0 \subseteq \ti X$ be the singular fiber over $0$ of a generically smooth fibration of curves of genus $g$. Then, combinatorially, $\ti X_0$ is obtained from its building blocks (\Cref{big-definitions}) by glueing them together along marked chains (\Cref{non-swap-local-glueing-algorithm,swap-local-glueing-algorithm}).
    \end{theorem}

    \begin{theorem}[See \Cref{arbitrary-glueing-exists}]
        Let $\ti Y_0 = \bigsqcup \ti Y_{0,i}$ be a collection of building blocks, and let $\{(p_\alpha,p'_\alpha)\}_{\alpha \in \mathcal{J}}$ be a set of (disjoint) pairs of points of $\ti Y_0$. Then, assuming that the pairs $(p_\alpha,p'_\alpha)$ satisfy certain compatibility conditions, there exists a fibered surface $\ti X \to B$ such that its central fiber $\ti X_0$ is obtained, combinatorially, by glueing $\ti Y_0$ along the chains corresponding to $p_\alpha,p'_\alpha$ as in the previous theorem.
    \end{theorem}

    In other words, this framework gives another justification and point of view of the algorithm for computing singular fibers \cite[Algorithm 10.10]{Dok25}. Broadly speaking, we classify all ``building blocks of genus $\leq g$'' [Steps 1 and 2 in \cite{Dok25}]. Then we use \Cref{global-glueing-algorithm} to connect, or ``glue'' these building blocks in every possible way [Steps 3 to 8 in \cite{Dok25}]. Each of these combinatorial glueings is realized by a singular fiber due to \Cref{arbitrary-glueing-exists}.

    This approach has other benefits. For instance, the finiteness of the families of possible singular fibers in a given genus follows directly from the fact that the inertia stack $\mathcal{I}_{\ov \M_g}$ is of finite type. Another benefit is that we can show geometrically when a combinatorial singular fiber exists as a fiber in a fibration with connected fibers, which we understand as a strengthening of \cite[Corollary 4.3]{Winters1974}.

    \begin{theorem}[= \Cref{existence-of-geometrically-connected-fiber}]
        Let $d \geq 1$ be an integer and let $(\Gamma,\{m_i\})$ be a combinatorial singular fiber of Euler characteristic $2d(1-g)$ for $g \geq 2$ (here, $\Gamma = \bigcup \Gamma_i$ is a reduced nodal curve and $m_i$ plays the role of the multiplicity of $\Gamma_i$). Assume that any curve $\Gamma_i$ with $m_i$ divisible by $\chr K$ is isomorphic to $\P^1$ and intersects the rest of $\Gamma$ in at most two points. Then the following are equivalent.
        \begin{enumerate}
            \item There exists a proper morphism $\ti \pi \colon \ti X \to B$ with smooth geometric generic fibers, where $\ti X$ is a smooth surface, $0 \in B$ is a smooth point in a curve, and $\ti X_{0,red} \cong \Gamma$ where the multiplicity of $\Gamma_i$ in $\ti X_{0}$ is $m_i$, and $\dim H^0(\ti X_0,\O_{\ti X_0}) = d$.
            \item At least one of the following hold
            \begin{itemize}
                \item $g(\Gamma) \geq 1$ and $d$ divides $\gcd(\{m_i\})$.
                \item $\gcd(\{m_i\}) = d$.
            \end{itemize}
        \end{enumerate}
    \end{theorem}
    
    Throughout this paper we fix an algebraically closed field $K$ and a system of compatible roots of unity. By the latter we mean a sequence $\{\zeta_n\}$ defined for all $n$ coprime to $\chr K$, where $\zeta_n$ is a primitive $n$-th root of unity satisfying $\zeta_{nk}^k = \zeta_n$.

    \subsection*{Acknowledgments} I would like to thank Dori Bejleri for his advice throughout this project. I would also like to thank Baidehi Chattopadhyay, Myeong-Jae Jeon, Giancarlo Urzúa and Nicolás Vilches for helpful discussions. The research was partially supported by NSF grant DMS-2401483 and the Simons Dissertation Fellowship from the Simons Foundation.

    \subsection*{Tool and Computational Resource Disclosure} No artificial intelligence or computational resources were used in the writing of this paper. 

    \section{Continued Fractions and Cyclic Quotient Singularities}

    For $0 \leq q \leq m$, which we do not assume coprime, we define the Hirzebruch-Jung continued fraction expansion of $\frac{m}{q}$ as
    \[\frac{m}{q} = e_1 - \dfrac{1}{e_2 - \dfrac{1}{\ddots - \dfrac{1}{e_\ell}}} = [e_1,\ldots,e_\ell].\]
    Here, if $q = 0$, we use the convention $\frac{m}{q} = \infty$, and its continued fraction expansion $[ ~ ]$ is the empty one. If $q = m$, we let $\frac{m}{q} = [1]$. Otherwise, we impose that $e_i \geq 2$ for all $i$. In this way, the expansion is unique, and to each expansion there corresponds a unique fraction.

    For a fraction $\frac{m}{q}$, we say its that its dual fraction is $\frac{m}{m-q}$. Similarly, if $[e_1,\ldots,e_\ell]$ is the continued fraction expansion of $\frac{m}{q}$, we say its dual is the continued fraction expansion of $\frac{m}{m-q}$. Note that $[ ~ ]$ and $[1]$ are dual to each other.

    When $0 < q < m$, or when $m = 1$ and $q = 0$ we call this fraction and its expansion \textit{standard}.

    \begin{definition}
        We say that the continued fraction expansion $[e_1,\ldots,e_\ell]$ is a prefix of $[f_1,\ldots,f_{\ell'}]$ if $\ell' \geq \ell$ and $e_i = f_i$ for all $i=1,\ldots,\ell$. It is a strict prefix if furthermore $\ell' > \ell$.

        We say that $[e_1,\ldots,e_\ell]$ is lexicographically smaller than $[f_1,\ldots,f_{\ell'}]$ if it is either a strict prefix, or if when $i \leq \min(\ell,\ell')$ is the first index for which $e_i$ and $f_i$ disagree, then $e_i < f_i$.
    \end{definition}

    \begin{remark} \label{prefix-property}
        The fraction $\frac{m}{q}$ is smaller than $\frac{m'}{q'}$ if and only if either the sequence associated to $\frac{m'}{q'}$ is a strict prefix of the sequence for $\frac{m}{q}$, or if the sequence for $\frac{m'}{q'}$ is lexicographically greater than that of $\frac{m}{q}$. For example
        \begin{align*}
            \frac{3}{2} = [2,2] &< [2] = \frac{2}{1}; & \frac{4}{3} = [2,2,2] &< [2,3,2] = \frac{8}{5}; & \frac{m}{q} = [e_1,\ldots,e_\ell] &< [~] = \infty.
        \end{align*}
    \end{remark}

    Fix $n,a,b \in \Z$, and let $\mu_n := \Spec K[z]/(z^n-1)$ be the group scheme of $n$-th roots of unity, which we do not assume reduced. We let $\mu_n$ act on $\A^2 = \Spec K[x,y]$ via the co-action
    \begin{align*}
        K[x,y] &\to K[z]/(z^n - 1) \otimes K[x,y]\\
        x &\mapsto z^a \otimes x\\
        y &\mapsto z^b \otimes y.
    \end{align*}
    The quotient $\A^2/\mu_n$ will in general be singular, and any surface singularity analytically isomorphic to the germ $(0 \in \A^2/\mu_n)$ will be denoted by $\frac{1}{n}(a,b)$. Note that
    \[\A^2/\mu_n = \Spec K[x^uy^v \mid au + bv \equiv 0 \mod n],\]
    and that if $n$ is coprime to the characteristic of $K = \ov K$, then this action is the usual one generated by
    \[\zeta_n \colon (x,y) \mapsto (\zeta_n^a x, \zeta_n^b y).\]

    When talking about a cyclic quotient singularity, we will almost always keep track of two curves through it: the ones corresponding under the analytic isomorphism to the image of the coordinate axes in $\A^2/\mu_n$. This way, if we write $\frac{1}{n}(a,b)$, then $a$ (resp. $b$) will tell us how $\mu_n$ acts along the $x$ axis (resp. $y$ axis).

    Here are some basic properties we will be using throughout this paper.
    \begin{lemma}\label{cyclic-quotients-basic-properties}\
        \begin{enumerate}
            \item $\frac{1}{n}(a,b)$ and $\frac{1}{n}(b,a)$ are equal as singularities, but the order in which we keep track the two special curves through this point is swapped.
            \item For every $d > 1$ the singularities $\frac{1}{dn}(da,db)$ and $\frac{1}{n}(a,b)$ are equal.
            \item If $\gcd(d,b) = 1$, then $\frac{1}{dn}(da,b)$ and $\frac{1}{n}(a,b)$ are equal. Similarly on the right.
            \item If $r$ is coprime to $n$, then $\frac{1}{n}(a,b) = \frac{1}{n}(ra,rb)$.
            \item Using these facts, we can assume, if needed, that a non-trivial cyclic quotient singularity is in the standard form $\frac{1}{m}(1,q)$ where $0 \leq q < m$ are coprime.
        \end{enumerate}
    \end{lemma}

    When $0 < q < m$, or when $m = 1$ and $q = 0$, we say $\frac{1}{m}(1,q)$ is written in standard form.

    \begin{theorem}[{\cite[Theorem 7.4.16]{Ishii2018}}]
        Let $p \in U$ be a cyclic quotient singularity in a surface $X$ of type $\frac{1}{m}(1,q)$ in standard form. Let $\pi \colon V \to U$ be the minimal resolution. Then the exceptional divisor $\pi^{-1}(x)$ consists of $\ell$ smooth rational curves $E_1,\ldots,E_\ell$ such that
        \begin{itemize}
            \item $E_i \cdot E_{i+1} = 1$ for $i = 1,\ldots,\ell - 1$,
            \item $E_i \cdot E_j = 0$ if $|i-j| \geq 2$,
            \item $E_i^2 = -e_i$, where $\frac{m}{q} = [e_1,\ldots,e_\ell]$ is the Hirzebruch-Jung continued fraction associated to this singularity
        \end{itemize}
    \end{theorem}

    In other words, the exceptional divisor in the minimal resolution is a linear chain of $\P^1$, with self intersections given by the coefficients in the continued fraction.

    In the situation of the theorem, let $L_0$ and $L_{\ell + 1}$ be the germs of curves through $p \in U$ corresponding to the $x$ and $y$ axis in $\A^2/\mu_n$. Denote their strict transforms in $V$ as $E_0$ and $E_{\ell+1}$. Then $E_0$ intersects only $E_1$ transversally at one point, and $E_{\ell+1}$ intersects only $E_\ell$ at one point.

    \begin{remark}
        When $\frac{m}{q} = \infty$, the quotient is smooth and the minimal resolution is an isomorphism. This is consistent with the fact that the continued fraction expansion of $\infty$ is empty. When $\frac{m}{q} = 1$, again the quotient is smooth, but we associate to this fraction the (non-minimal) blowup, whose exceptional curve has self intersection $-1$.
    \end{remark}

    The following definition and propositions are standard, so we provide no proofs.

    \begin{definition} \label{a_i}
        Given a standard fraction $\frac{m}{q} = [e_1,\ldots,e_\ell]$, define the sequences $a_i$ and $b_i$, for $i=0,\ldots,\ell + 1$ by
        \begin{align*}
            a_0 &= m, & a_1 &= q & a_{i+1} &= a_i e_i - a_{i-1},\\
            b_0 &= 0, & b_1 &= 1 & b_{i+1} &= b_i e_i - b_{i-1}
        \end{align*}
    \end{definition}

    \begin{remark}\label{remark:last-multiplicity-is-one}
        By construction, for $i = 1,\ldots,\ell$,
        \begin{align*}
            \frac{a_{i-1}}{a_i} = [e_i,\ldots,e_\ell] && \frac{b_{i+1}}{b_i} = [e_i,\ldots,e_1]
        \end{align*}
        and since $\frac{a_{\ell-1}}{a_\ell} = e_\ell \in \Z$, then $a_\ell = 1$ and $a_{\ell + 1} = 0$.
    \end{remark}

    \begin{proposition}
        Given a standard fraction $\frac{m}{q} = [e_1,\ldots,e_\ell]$, if $\frac{m'}{q'} = [e_\ell,\ldots,e_1]$, then $m = m'$ and $q'$ is the inverse of $q$ modulo $m$.
    \end{proposition}

    \begin{remark}
        This proposition also implies that $b_{\ell} = q'$ and $b_{\ell + 1} = m$.
    \end{remark}

    \begin{proposition} \label{chain-multiplicities}
        Using the same notation as above, we have the following identities as $\Q$-divisors.
        \begin{align*}
            \pi^* L_0 = E_0 + \sum_{i=1}^\ell \frac{a_i}{m} E_i && \pi^* L_{\ell + 1} = E_{\ell + 1} + \sum_{i=1}^\ell \frac{b_i}{m} E_i
        \end{align*}
    \end{proposition}

    \begin{corollary} \label{self-int-change}
        Let $C \subseteq X$ be a complete curve in a surface $X$, which is $\Q$-Cartier. Suppose $p \in C$ is a smooth point in $C$, but in $X$ it is a cyclic quotient singularity of type $\frac{1}{m}(1,q)$ in standard form, such that $C$ corresponds to the $x$ axis under the analytic isomorphism to $\A^2/\mu_m$. Let $\pi \colon Y \to X$ be the minimal resolution of $p$, and let $\ti C$ be the strict transform of $C$ in $Y$. Then
        \[\ti C^2 = C^2 - \frac{q}{m}\]
    \end{corollary}
    \begin{proof}
        Under the hypothesis, the germ at $p$ of $C$ is just $L_0$, and so
        \[\pi^*C = \ti C + \sum_{i=1}^\ell \frac{a_i}{m}E_i\]
        Then, as the $E_i$ are exceptional, and intersections are well behaved for $\Q$-Cartier divisors under birational maps,
        \[C^2 = (\pi^*C)^2 = \ti C \cdot \pi^*C = \ti C^2 + \sum_{i=1}^\ell \frac{a_i}{m} \ti C \cdot E_i = \ti C^2 + \frac{a_1}{m} = \ti C^2 + \frac{q}{m}.\]
    \end{proof}

    \begin{corollary} \label{self-int-change-nodal}
        Let $C \subseteq X$ be a complete curve in a surface $X$ which is $\Q$-Cartier. Suppose $p \in C$ is a node of $C$, which is a cyclic quotient singularity of type $\frac{1}{m}(1,q)$ in $X$ written in standard form, so that both branches of $C$ correspond to the $x$ and $y$ axes under the analytic isomorphism to $\A^2/\mu_n$. Assume that $m > 1$ so that $p$ is not smooth, and let $\pi \colon Y \to X$ be the minimal resolution of $p$, and let $\ti C$ be the strict transform of $C$ in $Y$. Then
        \[\ti C^2 = C^2 - \frac{q + q' + 2}{m},\]
        where $q'$ is the inverse of $q$ modulo $m$.
    \end{corollary}

    \begin{proof}
        We have 
        \[\pi^*C = \ti C + \sum_{i=1}^\ell \frac{a_i + b_i}{m}E_i\]
        Intersecting with $\ti C$,
        \[C^2 = (\pi^*C)^2 = \ti C \cdot \pi^* C = \ti C^2 + \frac{a_1 + b_1}{m} + \frac{a_\ell + b_\ell}{m}.\]
        We obtain the result after recalling that $a_1 = q$, $a_\ell = 1$, $b_1 = 1$ and $b_\ell = q'$.
    \end{proof}

    \begin{remark}
        This result is not true when $p$ is not smooth, as the minimal resolution at this point is an isomorphism, so we would have to require that $q + q' = -2$, which does not make sense in our setup.
    \end{remark}

    \begin{definition}
        For a standard cyclic quotient singularity of type $\frac{1}{m}(1,q)$, its dual singularity is $\frac{1}{m}(1,m-q)$.
    \end{definition}

    \begin{proposition}\label{dual-property}
        Suppose $q \neq m-1$. Write the continued fraction of $\frac{m}{q}$ as
        \[\frac{m}{q} = [\underbrace{2,\ldots,2}_{s_1},r_1,\underbrace{2,\ldots,2}_{s_2},r_2,\ldots,r_{k-1},\underbrace{2,\ldots,2}_{s_k}].\]
        where $r_i \geq 3$ and $s_i \geq 0$. Then its dual fraction is
        \[\frac{m}{m-q} = [s_1+2,\underbrace{2,\ldots,2}_{r_1-3},s_2+3,\ldots,s_{k-1}+3,\underbrace{2,\ldots,2}_{r_{k-1}-3},s_k+2].\]
        When $q = m-1$, then
        \[\frac{m}{m-1} = [\underbrace{2,\ldots,2}_{m-1}], \qquad \frac{m}{1} = [m].\]
    \end{proposition}

    \begin{proof}
        By induction on the length and leftmost entry of the expansion of $\frac{m}{q}$, suppose that this holds for $\frac{m}{q}$ and $\frac{m}{m-q}$. For the ``induction on the length'' step, we need to show that inserting a 2 at tle left of the expansion of $\frac{m}{q}$ and adding 1 to the leftmost entry of $\frac{m}{m-q}$, the resulting fractions are still dual. The result of inserting a 2 to $\frac{m}{q}$ is
        \[\frac{m'}{q'} = 2 - \frac{q}{m} = \frac{2m-q}{m}.\]
        Adding 1 to the leftmost entry of the expansion of $\frac{m}{m-q}$, one obtains
        \[\frac{m''}{q''} = \frac{m}{m-q} + 1 = \frac{2m-q}{m-q}.\]
        which remain dual. For the ``induction on the leftmost entry'' of $\frac{m}{q}$, we'd have to add 1 to the leftmost entry of $\frac{m}{q}$, insert a 2 at the left of $\frac{m}{m-q}$ and verify the results are still dual. The same computation above shows this holds.
    \end{proof}

    \begin{remark}
        This shows that when $q \neq m-1$, the number of 2's trailing at the left of the continued expansion of $\frac{m}{q}$ is $\ceil{\frac{m}{m-q}-2}$, and when $q = m-1$, it is $\ceil{\frac{m}{m-q}-1}$. Since the first fraction cannot be an integer but the second one is, these two can be expressed as $\floor{\frac{m}{m-q}-1}$ independently of whether of $q$ is $m-1$ or not. 
    \end{remark}

    \begin{proposition} \label{-2-contraction}
        Suppose that
        \[\frac{m}{q} = [\underbrace{2,\ldots,2}_{s},e_1,\ldots,e_\ell]\]
        is a Hirzebruch-Jung continued fraction where $s \geq 0$, $\ell \geq 1$, and $e_1 \geq 3$. Let $m = (m-q)d + r$, so $0 < r < m-q$ is the remainder of dividing $m$ by $m-q$ (it cannot be zero since that would imply the continued fraction consists only of 2's). Then $s = d-1$ and
        \[\frac{m-q}{r} = [e_1-1,e_2,\ldots,e_\ell].\]
    \end{proposition}
    
    \begin{remark}
        Geometrically, this corresponds to the following situation: Suppose that there is a chain of $\P^1$ in a smooth surface with self intersections given by $\frac{m}{q}$, and there is another $(-1)$-curve intersecting only the first curve in the chain. Then one may contract this $(-1)$-curve together with the (possibly empty) maximal chain of $(-2)$-curves connected to it, to obtain a new chain, still in a smooth surface. The self intersections of this new chain will be given by $\frac{m-q}{r}$ as above.
    \end{remark}

    \begin{proof}[Proof of \Cref{-2-contraction}]
        By writing $m = (m-q)d + r$, we know that the number of trailing 2's is $\floor{\frac{m}{m-q} - 1} = d-1$. One can easily show by induction on $k$ that removing $k$ 2's from the left of the chain corresponds to the fraction
        \[\frac{m - k(m-q)}{q - k(m-q)}.\]
        When $k = d-1$, we obtain
        \[\frac{r + (m-q)}{r} = [e_1,\ldots,e_\ell].\]
        Finally, subtracting one from the first entry is the same as subtracting one to the fraction, so we obtain
        \[\frac{m-q}{r} = [e_1-1,\ldots,e_\ell]\]
    \end{proof}

    \begin{remark} \label{dual-removal-property}
        Note that in this case, if the dual of $\frac{m}{q} = [2,\ldots,2,e_1,\ldots,e_\ell]$ is $\frac{m}{m-q} = [\ti e_1,\ldots,\ti e_{\ti \ell}]$, then the dual of $[e_1-1,e_2,\ldots,e_\ell]$ is $[\ti e_2,\ldots,\ti e_{\ti \ell}]$. This also holds true when $\ell = 1$ and $e_1 = 2$, as the dual of $[1]$ is $[~]$.
    \end{remark}

    For the next result, we abuse slightly the notation to allow sequences $[e_1,\ldots,e_\ell]$ of arbitrary integers which do not satisfy our assumptions from the start. Even though one could still attempt to make sense of the associated fraction, we will not do so in these cases, as to avoid problems such as division by zero.
    \begin{corollary} \label{contraction-of-duals}
        Consider the sequence
        \[[e_\ell,\ldots,e_1,1,f_1,\ldots,f_{\ell'}],\]
        where $\ell,\ell' \geq 0$ and $e_i,f_i \geq 2$. Then, as long as there is exactly one $1$ which is not at either end of the sequence, if we consecutively ``blow-down'' it down, that is, remove the $1$ and decrease by $1$ the adjacent entries, we eventually arrive to the configuration $[1,1]$ if and only if the sequences
        \[[e_1,\ldots,e_\ell] \qquad \text{ and } \qquad [f_1,\ldots,f_{\ell'}]\]
        are dual to each other.
    \end{corollary}
    \begin{proof}[Sketch of proof]
        The ``reverse'' of the blow-down corresponds precisely to adding 1 to the leftmost entry of one expansion and appending a $2$ to the left of the other one. Thus, starting from $[1,1]$, which we interpret as the concatenation of the dual expansions of $\frac{1}{1} = [1]$ and $\frac{1}{0} = [~]$, by the same argument as in the proof of \Cref{dual-property}, applying reverse blow-ups we get no more and no less than all pairs of dual fraction expansions.
    \end{proof}

    \section{Cyclic Quotients of \texorpdfstring{$A_k$}{Ak} Singularities}\label{section:cyclic-quotients-of-Ak}

    Let us assume that $\chr K \nmid n$. Let $X = \Spec K[x,y,t]/(xy-t^k)$ for $k \geq 1$, and suppose there is a $\mu_n$-action on $X$ given either by
    \begin{align*}
        \begin{cases}
            x & \mapsto \zeta_n^a x\\
            y & \mapsto \zeta_n^b y\\
            t & \mapsto \zeta_n t
        \end{cases}
        && \text{ or } &&
        \begin{cases}
            x & \mapsto \zeta_n^a y\\
            y & \mapsto \zeta_n^b x\\
            t & \mapsto \zeta_n t.
        \end{cases}
    \end{align*}

    We assume that $a$ and $b$ are not necessarily coprime, but for this action to be well defined, it must be that $k = a + b + jn$ for some integer $j$. In the first case, we assume that $0 < a,b \leq n$, so that $j \geq -1$ (if $a + b < n$, then $j \geq 0$). In the second case, we choose $a$ and $b$ so that $0 < a+b \leq n$, and note that both $n$ and $a+b$ need to be even. We also write $k = a+b+jn$, for some $j \geq 0$. We will denote the germ of a singularity analytically isomorphic to $X/\mu_n$ by $\frac{1}{n}(a,b,1;j)$ in the first case, and by $\frac{1}{n}(a+b,1;j)_D$ in the second. We will later see that the second singularity depends only on $a+b$, $n$ and $j$.

    Although we will not use it, the first action can be extended naturally to when the characteristic of $K$ does divide $n$, and the following theorem still holds in that case.

    \begin{theorem}\label{non-swap-quotient-theorem}
        Let $0 < a,b \leq n$, and $k = a + b + jn$ be such that $k \geq 1$. The singularity $\frac{1}{n}(a,b,1;j)$ equals the cyclic quotient singularity
        \[Y = \frac{1}{\frac{nk}{d_1d_2}}\p{1,\frac{qk-d_1}{d_2}},\]
        where $d_1 = \gcd(a,n)$, $d_2 = \gcd(b,n)$ and $0 < q \leq \frac{n}{d_1}$ satisfies $q \equiv \p{\frac{a}{d_1}}^{-1} \mod \frac{n}{d_1}$. Moreover, if
        \begin{align*}
            \frac{n}{n-a} = [\ti f_1,\ldots,\ti f_{\alpha'}] && \text{ and } && \frac{n}{n-b} = [\ti g_1,\ldots, \ti g_{\beta'}]
        \end{align*}
        then a (possibly non-minimal) resolution of the dual singularity of $Y$ has as exceptional divisor a chain with self intersections given by
        \[[\ti f_{\alpha'},\ldots,\ti f_1, 2+j,\ti g_1,\ldots,\ti g_{\beta'}].\]
    \end{theorem}

    \begin{remark}
        This statement still makes sense when $a = n$ or $b = n$. In these cases, $\frac{n}{n-a}$ and/or $\frac{n}{n-b}$ will have empty continued fractions. The resolution of the dual singularity of $Y$ will have no $\ti f_i$ and/or $\ti g_i$, but will still contain a curve with self intersection $-2-j$.
    \end{remark}

    \begin{proof}[Proof of \Cref{non-swap-quotient-theorem}]
        We prove this when $\chr K \nmid nk$ since we will be using quotients by $nk$-th roots of unity which may not exist otherwise. The result is still true by interpreting the relevant actions as being done by the non-reduced group $\mu_{nk}$.
        
        Fix $\eta = \zeta_{nk} \in \mu_{nk}$ to be the chosen primitive $nk$-th root of unity, so that $\eta^k = \zeta_n$.

        First, we may assume that $d = \gcd(a,b,n) = 1$. This is because under the action $(x,y,t) \mapsto (\zeta_n^a x, \zeta_n^b y, \zeta_n t)$, any invariant monomial in $K[x,y,t]$ must have every instance of $t$ appearing to some power which is multiple of $d$, so they form a subring of $K[x,y,t^d]$. Furthermore, the subring
        \[K[x,y,t^d]/(xy-t^k) \subseteq K[x,y,t]/(xy-t^k)\]
        is isomorphic to $K[x,y,t]/(xy-t^{k/d})$, and the $\mu_n$-action descends as a $\mu_{n/d}$-action. This gives an isomorphism
        \[\frac{1}{n}(a,b,1;j) \cong \frac{1}{n/d}\p{\frac{a}{d},\frac{b}{d},1;j}.\]

        We can identify
        \[K[x,y,t]/(xy-t^k) \cong K[U^k,V^k,UV] \subseteq K[U,V],\]
        where we see that $K[U^k,V^k,UV]$ is the subring of $K[U,V]$ fixed by the $\mu_k$-action $(U,V) \mapsto (\eta^n U, \eta^{-n}V)$. Now, for every $s \in \Z/k\Z$, the original $\mu_n$-action on $K[U^k,V^k,UV]$ lifts to a $\mu_{nk}$-action on $K[U,V]$ given by
        \[(U,V) \mapsto (\eta^{a + ns}U,\eta^{b + n(j-1-s)}V).\]
        This means that $Y \cong \Spec K[U,V]^G$, where $G$ is the subgroup of $\GL_2(K)$ generated by
        \begin{align*}
            R = \begin{pmatrix}
                \eta^{a + ns} & 0\\
                0 & \eta^{b + n(j-s)}
            \end{pmatrix}
            &&
            S = \begin{pmatrix}
                \eta^n & 0\\
                0 & \eta^{-n}
            \end{pmatrix}.
        \end{align*}
        We observe that $R^n = S^{a + ns}$, and that we can always choose $s$ so that $\gcd(a + ns,k) = 1$. To see this pick any $s$ satisfying the two conditions
        \begin{itemize}
            \item For every prime $p \mid k$, if $p \mid a$, then $s \not\equiv 0 \mod p$.
            \item For every prime $p \mid k$, if $p \nmid a$ and $p \nmid n$, then $s \not\equiv -an^{-1} \mod p$.
        \end{itemize}
        In the first case, $p$ cannot divide $n$ as it would also divide $b$ as well, so $p$ cannot divide $a + ns$ either. In the second case, $p$ will never divide $a + ns$.

        For this choice of $s$, this shows that $R$ generates $G$. Moreover, since $\det R = \eta^k = \zeta_n$, we see that $n \mid \ord(R)$. Since we also have $\ord(R^n) = \ord(S^{a + ns}) = k$, we obtain $\ord(R) = nk$.

        Now, let $0 < q \leq \frac{n}{d_1}$ be the inverse of $\frac{a}{d_1}$ modulo $\frac{n}{d_1}$. Since $\gcd(a + ns,k) = 1$, we can choose $r$ such that $(q + nr)(a + ns) \equiv d_1 \mod nk$. For every prime $p \mid d_1$, if $p$ divides $q$, it cannot also divide $\frac{n}{d_1}k$. So similarly as how we chose $s$, we can pick $r'$ so that
        \[\gcd\p{q + nr + \frac{n}{d_1}kr',nk} = 1.\]
        This way, the following element is also a generator of $G$
        \[R^{q + nr + \frac{n}{d_1}kr'} = \begin{pmatrix}
            \eta^{d_1} & 0\\
            0 & \eta^{qk - d_1 + \frac{n}{d_1}k^2r'}
        \end{pmatrix},\]
        which in turn tells us that
        \[Y = \frac{1}{nk}\p{d_1,qk-d_1 + \frac{n}{d_1}k^2r'}.\]
        By the choice of $r$ and $r'$ which made $q + nr + \frac{n}{d_1}kr'$ coprime to $nk$, we obtain
        \[\gcd\p{qk - d_1 + \frac{n}{d_1}k^2r',nk} = \gcd(b + n(j-s),nk),\]
        but since
        \begin{align*}
            \gcd(b+n(j-s),n) &= \gcd(b,n) = d_2, && \text{and}\\
            \gcd(b + n(j-s),k) &= \gcd(k-a-ns,k) = 1,
        \end{align*}
        we deduce
        \[\gcd(qk-d_1+\frac{n}{d_1}k^2r',nk) = d_2,\]
        and so, since $d_1$ and $d_2$ are coprime, we can use (2) and (3) from \Cref{cyclic-quotients-basic-properties} to see
        \[Y = \frac{1}{nk}\p{d_1,qk-d_1 + \frac{n}{d_1}k^2r'} = \frac{1}{\frac{nk}{d_1d_2}}\p{1,\frac{qk-d_1}{d_2}}.\]

        Now, to find an explicit resolution of the dual singularity $\check Y$, consider the generator of its action
        \[\check R = \begin{pmatrix}
            \eta^{a + ns} & 0\\
            0 & \eta^{-b-n(j-s)}
        \end{pmatrix}.\]
        Note that $\check R^n = \eta^{n(a + ns)}I$, and the subring of $K[U,V]$ fixed by this power is
        \[K[U,V]^{\check R} = K[U^k,U^{k-1}V,\ldots,UV^{k-1},V^k].\]
        In this subring, $\check R$ descends as a $\mu_n$-action given by
        \[U^iV^{k-i} \mapsto \eta^{k(i-b)}U^iV^{k-i} = \zeta_n^{i-b}U^iV^{k-i}.\]
        Let $X = \Spec K[U,V]^{\check R^n}$, which sits inside $\A^{k + 1} = \Spec K[x_0,\ldots,x_k]$ via
        \begin{align*}
            K[x_0,\ldots,x_k] &\hookrightarrow K[U,V]^{\check R^n}\\
            x_i \mapsto U^{k-i}V^i,
        \end{align*}
        whose defining ideal is generated by elements of the form $x_\alpha x_\beta - x_{\alpha + 1}x_{\beta-1}$. Let $\ti X \subseteq \operatorname{BL}_0\A^{k+1}$ be the strict transform of $X$ in the blowup at $0 \in \A^{k+1}$. The exceptional set of $\ti X \to X$ is a $\P^1$ with self intersection $-k$. Considering
        \[\operatorname{Bl}_0 \A^{k+1} = V(x_\alpha y_\beta - x_\beta y_\alpha) \subseteq \A^{k+1}_{x_i} \times \P^k_{y_i},\]
        the $\check R$-action on $X$ extends to $\A^{k+1} \times \P^k$ via
        \[((x_i)_{0 \leq i \leq k, [y_i]_{0 \leq i \leq k}}) \mapsto ((\zeta_n^{k-i-b} x_i), [\zeta_n^{k-i-b}y_i]).\]

        $\ti X$ is covered by the two affine open sets $\{y_0 \neq 0\}$ and $\{y_k \neq 0\}$. In the first open $\ti X_0$ defined by setting $y_0 = 1$, we have $x_i = y_ix_0$ and $y_i = y_1^i$, so $\ti X_0 \cong \A^2_{x_0,y_1}$. Here, the $\check R$-action is given by
        \[(x_0,y_1) \mapsto \p{\zeta_n^{k-b}x_0, \frac{\zeta_n^{k-1-b}}{\zeta_n^{k-b}}y_i} = (\zeta_n^a x_0, \zeta_n^{-1}y_1),\]
        so the quotient has a single singularity of type $\frac{1}{n}(1,n-a)$. On the other hand, in the open $\ti X_k$ defined by setting $y_k = 1$, we have $x_i = y_ix_k$ and $y_i = y_{k-1}^{k-i}$, so $\ti X_k \cong \A^2_{x_k,y_{k-1}}$. The $\check R$-action is given by
        \[(x_k,y_{k-1}) \mapsto (\zeta_n^{-b}x_k,\zeta_n y_{k-1}),\]
        so there is a single singularity of type $\frac{1}{n}(1,n-b)$. Note that these two points are smooth if $a = n$ or $b = n$.

        With this, if $\pi \colon \ti X \to Z$ is the quotient by $\mu_n$, and $\kappa \colon \ti Z \to Z$ is the minimal resolution of these two singularities, then $\ti Z \to \check Y$ will be a (possibly non-minimal) resolution of $Y$. We already know the self intersections of all the curves in this resolution except for the middle curve, connecting the resolutions of $\frac{1}{n}(1,n-a)$ and $\frac{1}{n}(1,n-b)$. This curve is the strict transform of the image of exceptional of the blowup $\ti X \to X$.

        Let $E \subseteq \ti X$ be the exceptional curve of $\ti X \to X$, so $E \cong \P^1$ and $E^2 = -k$, and let $\ov E \subseteq Z$ be its image. As these surfaces are $\Q$-factorial, self intersections of these curves are well defined in the rational Chow rings. Moreover, looking at local charts, we see that $E \to \ov E$ is a degree $n$ map. Since $\pi$ is also a degree $n$ map, we must have $\pi^*\ov E = E$. Thus,
        \[\ov E^2 = \frac{1}{\deg \pi} \pi^*\ov E \cdot \pi^*\ov E = \frac{1}{n}E^2 = -\frac{k}{n} = -j - \frac{a+b}{n}.\]
        Let $\ti E$ be the strict transform of $\ov E$ in the resolution $\ti Z \to Z$. If $F_1$ and $G_1$ are the components of the resolution of $\frac{1}{n}(1,n-a)$ and $\frac{1}{n}(1,n-b)$ intersecting $\ti E$ (which we consider formally as zero if the respective point is smooth), then
        \[\kappa^*\ov E = \ti E + \frac{n-a}{n}F_1 + \frac{n-b}{n}G_2 + \sum_{\alpha > 1}\lambda_\alpha F_\alpha + \sum_{\beta > 1}\lambda_\beta' G_\beta\]
        where the $F_\alpha$ and $G_\beta$, $\alpha,\beta > 1$ are the components that do not intersect $\ti E$ and $\lambda_\alpha, \lambda'_\beta$ are rational numbers. Then we obtain
        \[-j - \frac{a+b}{n} = \ov E^2 = (\kappa^* \ov E)^2 = \kappa^*\ov E \cdot \ti E = \ti E^2 + \frac{n-a}{n} + \frac{n-b}{n}.\]
        This implies that $\ti E^2 = -2-j$ which finishes the proof.
    \end{proof}

    \begin{example}
        We will now show how the resolution of the singularity of $\frac{1}{n}(a,b,1;j)$ relates to the singularities $\frac{1}{n}(1,a)$ and $\frac{1}{n}(1,b)$. Consider the case when $n = 2200$, $a = 900$ and $b = 1540$, and note that
        \begin{align*}
            \frac{n}{a} = \frac{22}{9} = [3,2,5] && \frac{n}{b} = \frac{10}{7} = [2,2,4].
        \end{align*}
        We compute that $d_1 = 100$, $d_2 = 220$ and $k = 2440 + 2200j$. Note that $a + b > n$, or in other words, $\frac{a}{n} + \frac{b}{n} > 1$, so we can consider $j$ starting with $-1$.

        The inverse of $a/d_1$ modulo $n/d_1$ is the inverse of $9$ modulo $22$ which is $5$. We have
        \begin{align*}
            Y = \frac{1}{\frac{nk}{d_1d_2}}\p{1,\frac{qk-d_1}{d_2}} &= \frac{1}{\frac{2200(2440 + 2200j)}{22000}}\p{1,\frac{5(2440 + 2200j)-100}{220}}\\
            &= \frac{1}{244 + 220j}(1,55 + 50j)
        \end{align*}
        We compute a few continued fraction expansions with small $j$
        \begin{align*}
            j &= -1 & \frac{1}{24}&(1,5) & \frac{24}{5} &= [5,5],\\
            j &= 0 & \frac{1}{244}&(1,55) & \frac{244}{55} &= [5,2,5,2,4],\\
            j &= 1 & \frac{1}{464}&(1,105) & \frac{464}{105} &= [5,2,4,3,2,4],\\
            j &= 2 & \frac{1}{684}&(1,155) & \frac{684}{155} &= [5,2,4,2,3,2,4],\\
            j &= 3 & \frac{1}{904}&(1,205) & \frac{904}{205} &= [5,2,4,2,2,3,2,4],\\
            j &= 4 & \frac{1}{1124}&(1,255) & \frac{1124}{255} &= [5,2,4,2,2,2,3,2,4].
        \end{align*}
        The pattern is clear starting from $j=1$. The result is just the concatenation of the expansions for $\frac{n}{a}$ and $\frac{n}{b}$ (after adding $1$ to the left-most entry of both expansions and later flipping the first one), together with $(j-1)$ $2$'s in the middle. For $j=0$ it is also clear. We concatenate both expansions, adding the initial terms together.

        Both of these conclusions are easy consequences of the expression we have for the dual of $Y$ using \Cref{dual-property}. When $j=-1$, the fact that a $-1$-curve appears in the resolution of the dual singularity suggests that there is some sort of ``contraction'' or ``cancellation'' of the two chains. The procedure in this case is discussed in \Cref{non-swap-local-glueing-algorithm}.
    \end{example}

    \begin{example}
        As another example, consider $n = 215$, $a = 129$, $b = 90$. Then $d_1 = 43$, $d_2 = 5$, $q = 2$ and $k = 219 + 215j$. We have
        \[Y = \frac{1}{\frac{215(219 + 215j)}{215}}\p{1,\frac{2(219+215j) - 43}{5}} = \frac{1}{219 + 215j}(1,79 + 86j).\]
        When $j = -1$, we obtain $\frac{1}{4}(1,-7) = \frac{1}{4}(1,1)$, so the resolution of the singularity consists of a single curve with self intersection $-4$. For small $j$ we have
        \begin{align*}
            j &= -1 & \frac{1}{4}&(1,-7) & \frac{4}{1} &= [4],\\
            j &= 0 & \frac{1}{219}&(1,79) & \frac{219}{79} &= [3,5,2,3,4],\\
            j &= 1 & \frac{1}{434}&(1,165) & \frac{434}{165} &= [3,3,4,2,3,4],\\
            j &= 2 & \frac{1}{649}&(1,251) & \frac{649}{251} &= [3,3,2,4,2,3,4],\\
            j &= 3 & \frac{1}{864}&(1,337) & \frac{864}{337} &= [3,3,2,2,4,2,3,4],\\
            j &= 4 & \frac{1}{1079}&(1,423) & \frac{1079}{423} &= [3,3,2,2,2,4,2,3,4].\\
        \end{align*}
        
        The fact that this singularity is not in standard form for $j = -1$ becomes important in algorithm \Cref{global-glueing-algorithm}.
    \end{example}

    \begin{example}
        We now deal with the degenerate case $a = b = n$, corresponding to when $\frac{n}{a} = \frac{n}{b} = [1]$. Here $k = n(j+2)$ and $q = 1$, so
        \[Y = \frac{1}{\frac{n^2(j+2)}{n^2}}\p{1,\frac{qn(j+2)-n}{n}} = \frac{1}{j+2}(1,j+1).\]
        The associated continued fraction expansion is
        \[\frac{j+2}{j+1} = [\underbrace{2,\ldots,2}_{j+1}],\]
        which corresponds to the resolution of an $A_{j+1}$ singularity. In particular, when $j = -1$, the quotient is already smooth.
    \end{example}

    \begin{example}
        There are other ways in which the quotient can be smooth, which happen when $\frac{nk}{d_1d_2} = 1$. Take for example $n = 1150$, $a = 300$ and $b = 851$. Here $d_1 = 50$, $d_2 = 23$, $q = 4$ and $k = 1151 + 1150j$. We have
        \[Y = \frac{1}{\frac{1150(1151+1150j)}{1150}}\p{1,\frac{4(1151+1150j)-50}{23}} = \frac{1}{1151 + 1150j}(1,198 + 200j)\]

        When $j = -1$, we obtain $\frac{1}{1}(1,-2)$. As before, even though the point is smooth, the fact that the coefficient obtained is negative is still important in algorithm \Cref{global-glueing-algorithm}.
    \end{example}

    \begin{theorem} \label{swap-quotient-theorem}
        Let $n$ be a positive even integer and let $a,b$ be such that $0 < a + b \leq n$, and that $a+b$ is also even. Let $j \geq 0$ and define $k = a + b + jn$, and assume $\chr K \nmid n$. Then $\frac{1}{n}(a + b,1;j)_D$ obtained via quotient of $\Spec K[x,y,t]/(xy-t^k)$ under the action $(x,y,t) \mapsto (\zeta_n^ay, \zeta_n^bx,\zeta_nt)$ is a dihedral quotient singularity: its minimal resolution has dual graph
        
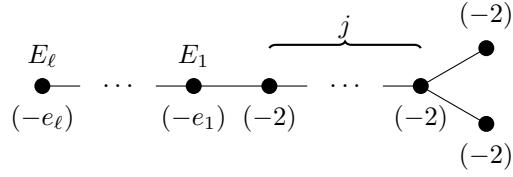
\[\begin{tikzpicture}[
            dot/.style = {circle, fill, minimum size=6pt, inner sep=0pt, outer sep=0pt}]
            \node[dot] at (1,0) {};
            \node[label=above:{$E_\ell$}] at (1,0) {};
            \node[label=below:{$(-e_\ell)$}] at (1,0) {};
            \draw (1,0) -- (1.5,0);
            \node[] at (2,0) {$\ldots$};
            \draw (2.5,0) -- (3,0);
            \node[dot] at (3,0) {};
            \node[label=above:{$E_1$}] at (3,0) {};
            \node[label=below:{$(-e_1)$}] at (3,0) {};
            \draw (3,0) -- (4,0);
            \node[dot] at (4,0) {};
            \node[label=below:{$(-2)$}] at (4,0) {};
            \draw (4,0) -- (4.5,0);
            \node[] at (5,0) {$\ldots$};
            \draw (5.5,0) -- (6,0);
            \node[dot] at (6,0) {};
            \node[label=below:{$(-2)$}] at (6,0) {};
            \draw[thick, decoration={brace},decorate] (4,0.5) -- (6,0.5) node[above,pos=0.5] {$j$};
            \draw (6,0) -- (6.866,0.5);
            \node[dot] at (6.866,0.5) {};
            \node[label=above:{$(-2)$}] at (6.866,0.5) {};
            \draw (6,0) -- (6.866,-0.5);
            \node[dot] at (6.866,-0.5) {};
            \node[label=below:{$(-2)$}] at (6.866,-0.5) {};
        \end{tikzpicture}
        \]
        where
        \[\frac{a + b + n}{a+b} = [e_1,\ldots,e_\ell].\]
        In particular, the singularity depends only on $a+b$ instead of on both $a$ and $b$. Hence the notation.
    \end{theorem}

    \begin{proof}
        We again prove this when $\chr K \nmid nk$. The result still holds true when $\chr K > 2$, but the analysis is a bit more delicate.
        
        Let $Y = \Spec K[x,y,t]/(xy-t^k)^{\mu_n}$ be the quotient under the $\mu_n$-action. Fix $\eta$ to be the chosen primitive $nk$-th root of unity, so that $\eta^k = \zeta_n$. Write
        \[K[x,y,t]/(xy-t^k) \cong K[U^k,V^k,UV] = K[U,V]^{\mu_k} \subseteq K[U,V],\]
        where $\mu_k$ acts by $(U,V) \mapsto (\eta^nU, \eta^{-n}V)$. The original $\mu_n$-action has a lift to $K[U,V]$ as the $\mu_{nk}$-action
        \[(U,V) \mapsto (\eta^a V, \eta^{b + n(j-1)}U)\]
        If we set $X = \Spec K[U,V]$, then $Y = X/G$, where $G$ is the subgroup of $\operatorname{GL}_2(K)$ generated by the two actions
        \begin{align*}
            R = \begin{pmatrix}
                \eta^n & 0\\
                0 & \eta^{-n}
            \end{pmatrix}, && S = \begin{pmatrix}
                0 & \eta^a\\
                \eta^{b + n(j-1)} & 0
            \end{pmatrix}.
        \end{align*}
        Let $\ti X = \operatorname{Bl}_0 X \subseteq X \times \P^1$ be the blowup at the origin, and let $u,v$ be the projective coordinates on $\P^1$ so that $\ti X$ is defined by $uV - vU = 0$. Both $R$ and $S$ lift to $\ti X$ as
        \begin{align*}
            R \colon \ti X &\to \ti X\\
            ((U,V),[u,v]) &\mapsto ((\eta^n U, \eta^{-n}V),[\eta^n u, \eta^{-n} v]),\\
            S \colon \ti X &\to \ti X\\
            ((U,V),[u,v]) &\mapsto ((\eta^a V, \eta^{b + n(j-1)} U), [\eta^a v, \eta^{b + n(j-1)} u]).
        \end{align*}
        Let $Z = \ti X/G$ and $\ti Z \to Z$ be its minimal resolution. In this way, $Z \to Y$ will be an (a priori non-minimal) resolution of $Y$.

        We verify
        \begin{align*}
            \ord(R) = k && \ord(S) = 2n && R^{k/2} = -I && S^2 = \zeta_nI,
        \end{align*}
        so as an abstract group, we have
        \[G = \langle R, S \mid S^{2n} = 1, \quad R^{k/2} = S^n, \quad SR = R^{-1}S \rangle.\]
        Let $E \subseteq \ti X$ be the exceptional divisor defined by $U = V = 0$, and $\ov E \subseteq Z$ be its image. We see that $E$ is $G$-invariant and the kernel of this action is $\langle S^2 \rangle$, so we have an exact sequence
        \[1 \to \langle S^2 \rangle \to G \to D_{k/2} \to 1\]
        where the dihedral group $D_{k/2}$ acts faithfully on $E$.

        Generically in $E$, the stabilizer is $\langle S^2 \rangle$ and the image in $Z$ will be smooth. There are three orbits with higher stabilizer.

        One of these orbits contains the point $((0,0),[0,1])$, where the stabilizer is $\langle R,S^2 \rangle \leqslant G$. Taking the local chart $v = 1$, where $U = uV$, this action looks like
        \begin{align*}
            R(u,V) = (\eta^{2n}u,\eta^{-n}V), && S^2(u,V) = (u,\eta^{k}V).
        \end{align*}
        Since $S^2$ fixes $u$ and acts on $V$ as multiplication by $\zeta_n$, we have
        \[K[u,V]^{\langle R,S^2 \rangle} = K[u,V^n]^{R}.\]
        $R$ acts on this subring as
        \[R(u,V^n) = (\eta^{2n} u, \eta^{-n^2}V^2) = \p{\zeta_{k/2}u, \zeta_{k/2}^{-n/2}V^n},\]
        which implies that the corresponding point in the quotient $Z$ is a cyclic quotient singularity of type
        \[\frac{1}{k/2}\p{1,-\frac{n}{2}}.\]
        Note that this point is smooth if and only if $k$ divides $n$.

        The other two orbits correspond to the two points $((0,0),[\pm \eta^{a-k/2},1])$, where the stabilizer is $\langle \sigma \rangle \leqslant G$. In the local chart given by $v = 1$ where $U = uV$, we have
        \[S(u,V) = \p{\frac{\eta^{2a-k}}{u},\eta^{k-a}uV}.\]
        At the point $(u,V) = (\pm \eta^{a-k/2},0)$, the differential of this map is
        \[dS = -du \pm \eta^{k/2}dV = \zeta_{2n}^n du + \zeta_{2n}^{1 + \frac{n}{2} \pm \frac{n}{2}}dV,\]
        which implies that the respective points in the quotient $Z$ are cyclic quotient singularities of type
        \[\frac{1}{2n}\p{n,1 + \frac{n}{2}(1 \pm 1)} = \frac{1}{2}(1,1).\]
        Let $\kappa \colon \ti Z \to Z$ be the minimal resolution of these three points, and let $\ti E \subseteq \ti Z$ be the strict transform of $\ov E$. As $\pi \colon \ti X \to Z$ has degree $nk$ but $\pi|_E \colon E \to \ov E$ has degree $k$, we have $\pi^*\ov E = nE$, and thus
        \[\ov E^2 = \frac{1}{\deg \pi} \pi^*\ov E \cdot \pi^*\ov E = \frac{n}{k}E^2 = -\frac{n}{k}.\]
        Let $\delta = \ceil{\frac{n}{k}}$, so that
        \begin{align*}
            \frac{1}{k/2}\p{1,-\frac{n}{2}} = \frac{1}{k/2}\p{1,\delta\frac{k}{2} - \frac{n}{2}}, && 0 \leq \delta\frac{k}{2}-\frac{n}{2} < \frac{k}{2},
        \end{align*}
        so that we write the singularity in standard form. We have
        \[\kappa^*\ov E = \ti E + \p{\frac{\delta \frac{k}{2} - \frac{n}{2}}{k/2}}F_1 + \frac{1}{2}G_1 + \frac{1}{2} G_2 + \sum_{\alpha > 1} \lambda_\alpha F_\alpha,\]
        where $F_1$ is the component over $\frac{1}{k/2}(1,-\frac{n}{2})$ intersecting $\ti E$ (which we consider formally as zero if the point is smooth, i.e. when $\delta = \frac{n}{k}$), and $G_1,G_2$ are the exceptional curves over the two points $\frac{1}{2}(1,1)$. This implies that
        \[-\frac{n}{k} = \ov E^2 = (\kappa^*\ov E)^2 = \kappa^*\ov E \cdot \ti E = \ti E^2 + \delta - \frac{n}{k} + 1,\]
        or in other words,
        \[\ti E^2 = -1-\delta.\]
        This already means that $\ti Z \to Y$ is the minimal resolution, as $\delta$ has to be positive and all other exceptional curves other than $\ti E$ also have self intersection less than $-1$.

        To verify the rest of the claim, recall that the self intersections of the chain over $\frac{1}{k/2}(1,-\frac{n}{2})$ are given by the continued fraction expansion
        \[\frac{k/2}{\delta\frac{k}{2} - \frac{n}{2}} = \frac{k}{\delta k - n} = [f_1,\ldots,f_{\ell'}].\]
        The fraction for this chain together with $\ti E$ must be
        \[[1 + \delta, f_1,\ldots,f_{\ell'}] = 1 + \delta - \frac{\delta k - n}{k} = \frac{k+n}{k}.\]
        If $j = 0$, then $k = a + b$, so we immediately obtain that
        \[[1 + \delta,f_1,\ldots,f_{\ell'}] = \frac{a+b + n}{a + b} = [e_1,\ldots,e_\ell].\]
        When $j \geq 1$, notice that
        \[\frac{k + n}{k} = \frac{a + b + (j+1)n}{a+b+jn} = 2 - \frac{1}{\frac{a+b+jn}{a+b+(j-1)n}}.\]
        We can inductively remove $2$'s from the left of the expansion of $\frac{k+n}{n}$ up to the point we arrive to $\frac{a+b + n}{a+b}$, and so obtain
        \[[1 + \delta,f_1,\ldots,f_{\ell'}] = [\underbrace{2,\ldots,2}_{j},e_1,\ldots,e_{\ell}]\]
    \end{proof}

    \begin{remark}
        The theorem still holds true when $p = \chr K > 2$, although we do need to assume that $p \nmid n$ to even make sense of the action.

        Let us go back to characteristic zero for a while, as to keep using the notation from before, but we fix a prime $p$ which we assume coprime to $n$. We defined the lift $S$ acting on $X = \Spec K[U,V]$ as
        \[S = \begin{pmatrix}
            0 & \eta^a\\
            \eta^{k - a} &  0
        \end{pmatrix},\]
        but we could have lifted $S$ as 
        \[S = \begin{pmatrix}
            0 & \eta^{a + ns}\\
            \eta^{k - a - ns} & 0
        \end{pmatrix}\]
        for any $s \in \Z$. Since $n$ is coprime to $p$, we can choose $s$ so that $a + ns$ is as divisible by $p$ as we want. In particular, if $k = p^\ell k'$, we can consider
        \[S = \begin{pmatrix}
            0 & \zeta_{nk'}^{(a + ns)/p^\ell}\\
            \zeta_{nk'}^{(k - a - ns)/p^\ell} & 0
        \end{pmatrix}\]
        which does not involve any $p$-th roots of unity. We write $a' = \frac{a + ns}{p^\ell}$.

        Now we need an analogue of $G$ in characteristic $p$ that acts on monomials in the same way as $G$ does in characteristic zero. Recall that we defined the group
        \[G = \langle R,S \mid S^{2n} = 1, \quad R^{k/2} = S^n, \quad SR = R^{-1}S \rangle.\]
        Using as guides the functions $\rho,\sigma,\omega \colon G \to \C$ given by
        \begin{align*}
            \rho(R^cS^d) = \zeta_{k/2}^c, && \sigma(R^cS^d) = (-1)^d, && \omega(R^cS^d) = \zeta_k^c\zeta_{2n}^d,
        \end{align*}
        we define the group scheme
        \[G_p = \Spec \F_p\left[\rho,\sigma,\omega\right]/(\rho^{k/2}-1, \sigma^2-1,\omega^n-\rho^{n/2}\sigma),\]
        with multiplication and inverse given by
        \begin{align*}
            m^*\rho &= \frac{1 + \sigma}{2}\rho \otimes \rho + \frac{1-\sigma}{2}\rho \otimes \rho^{-1}, & i^*\rho &= \frac{1 + \sigma}{2}\rho^{-1} + \frac{1-\sigma}{2}\rho,\\
            m^*\sigma &= \sigma \otimes \sigma, & i^*\sigma &= \sigma,\\
            m^*\omega &= \frac{1 + \sigma}{2}\omega \otimes \omega + \frac{1 - \sigma}{2}\omega \otimes \rho^{-1}\omega, & i^*\omega &= \frac{1+\sigma}{2}\rho\omega^{-1} + \frac{1-\sigma}{2}\omega^{-1}.
        \end{align*}
        Here $\sigma = (\omega^2\rho^{-1})^{n/2}$ is redundant, but for simplicity we include anyways.

        We define the action
        \[\phi \colon G_p \times \Spec K[U,V] \to \Spec K[U,V]\]
        via
        \begin{align*}
            \phi^*U &= \frac{1+\sigma}{2}\omega\otimes U + \frac{1-\sigma}{2}\omega\otimes \zeta_{nk'}^{-\frac{k'}{2} + a'} V,\\
            \phi^*V &= \frac{1 + \sigma}{2}\omega\rho^{-1} \otimes V + \frac{1 - \sigma}{2}\omega\rho^{-1} \otimes \zeta_{nk'}^{\frac{k'}{2} - a'}U.
        \end{align*}

        As in characteristic zero, this group scheme lies in two exact sequences. The first one is
        \[1 \to \mu_k \to G \to \mu_n \to 1\]
        where $\mu_k$ is the subgroup determined by setting $\rho = \omega^2$ and $\sigma = 1$, and $\mu_n$ is the spectrum of the subring generated by $\omega^2\rho^{-1}$. This sequence relates the action of $G$ to the action of $\mu_n$ on $\Spec K[x,y,t]/(xy-t^k)$. The second exact sequence is
        \[1 \to \mu_n \to G \to D_{k/2} \to 1\]
        where $\mu_n$ is the subgroup determined by setting $\rho = \sigma = 1$, and $D_{k/2} \cong \mu_{k/2}\rtimes \Z/2\Z$ is the spectrum of the subring generated by $\rho$ and $\sigma$. This sequence corresponds to the action of $G$ on the exceptional divisor of the blowup of $\Spec K[U,V]$.

        Using this group action, the rest of the proof follows analogously.
    \end{remark}

    \section{Decomposition of a Singular Fiber into Building Blocks} \label{section:glueing-building-blocks}

    In what follows, we assume that $n$ is coprime to the characteristic of the base field $K$, and we let $0 = \Spec K \hookrightarrow D = \Spec K[[t]]$ be the disk with the $\mu_n$ action given by $\zeta_n \cdot t = \zeta_nt$. 
    
    In this section, we discuss the results seen in \Cref{section:cyclic-quotients-of-Ak} applied to our case. We suppose we are given a generically smooth fibration $X \to D$, where $\mu_n = \langle \zeta_n \rangle$ acts equivariantly on $X \to D$, and where $\zeta_n$ acts as an automorphism $f \colon C \to C$ on the central fiber, which we assume nodal. Let $\ti X \to X/\mu_n$ be the minimal resolution. Then if $p \in C_1 \cap C_2 \subseteq C$ is a node, we wish to understand the chain of exceptional curves in $\ti X$ over the image of $p$, and also the self intersection of the images of $C_1$ and $C_2$, only in terms of how $f$ acts on $C_1$ and $C_2$ and the extra piece of data $j$ as in \Cref{section:cyclic-quotients-of-Ak}.

    More generally, given $C \subseteq X \to D$ as above with a $\mu_n$-action restricting to $f \in \Aut(C)$, let $C^\nu$ be the normalization of $C$. Then the automorphism $f$ lifts to $C^{\nu}$, a lift we still call $f$, and we can define a new smooth fibration (but with possibly disconnected fibers) $X^{\nu} = C^{\nu} \times D \to D$ with the induced action $\zeta_n \cdot (x,t) = (f(x),\zeta_n t)$. We will see how to relate the central fiber of the minimal resolution of the quotient $\ti X_0 \subseteq \ti X \to X/\mu_n$ in terms of the central fiber of the minimal resolution $\ti X_0^\nu \subseteq X^\nu \to X^\nu/\mu_n$, which we interpret as the ``building blocks'' from which $\ti X_0$ is constructed.

    \begin{remark} \label{the-promise}
        The process of ``glueing'' building blocks is purely combinatorial. Our claim is that we can recover the dual graph, self intersections and multiplicities of $\ti X_0$ from $\ti X_0^\nu$. The fact we define $\ti X_0^\nu$ as a quotient of a trivial family is already a non-canonical choice. Moreover, we do not yet claim that an arbitrary glueing of building blocks $\ti Y_0^\nu$ is realizable by a singular fiber in a surface $\ti Y_0$ or that there is a unique such $\ti Y_0$. Indeed, uniqueness fails as we will explain later in \Cref{glueing-is-not-canonical-scheme-structure}. The question of existence is the essential content of \Cref{section:classification-of-all-fibers}.
    \end{remark}

    In our approach, We also only require $C$ to be nodal, and not necessarily stable, as long, of course, that the automorphism $f \colon C \to C$ has finite order $n$ coprime to $\chr K$.
    
    We start with an analysis on the local picture at nodes of $C$. If $p \in C$ is a node with étale local branches $L_1,L_2$, then the minimal power of $f$ fixing $p$ must act on $T_pL_i$ in a way associated to the fraction $\frac{n_i}{a_i}$ for some coprime integers $n_i$ and $a_i$. With this picture in mind, we will define some local invariants that will be useful later.

    \begin{definition}
        Let $\infty > \frac{n_1}{a_1}, \frac{n_2}{a_2} \geq 1$ be fractions such that $\frac{a_1}{n_1} + \frac{a_2}{n_2} > 1$, or in other words, that $\frac{n_1}{n_1-a_1} > \frac{n_2}{a_2}$ (c.f. \Cref{prefix-property}). Write
        \[\frac{n_1}{n_1 - a_1} = [\tilde e_1,\ldots,\tilde e_{\tilde \ell}] \qquad \frac{n_2}{a_2} = [f_1,\ldots,f_{\ell'}]\]
        Define $\delta(\frac{n_1}{a_1},\frac{n_2}{a_2})$ by
        \[\delta\p{\frac{n_1}{a_1},\frac{n_2}{a_2}} = \begin{cases}
            f_{\tilde \ell + 1} - 1 & \text{ if } [\tilde e_1,\ldots,\tilde e_{\tilde \ell}] \text{ is a prefix of } [f_1,\ldots,f_{\ell'}]\\
            0 & \text{ otherwise.}
        \end{cases}\]
        In particular, when $n_1 = a_1 = 1$, as $[~]$ is a prefix of every expansion, we always have $\delta(\frac{1}{1},\frac{n_2}{a_2}) = f_1-1$.
    \end{definition}
    
    \begin{proposition} \label{deficit-proposition}
        Under the same hypothesis as above, and assuming that $\gcd(n_1,a_1) = \gcd(n_2,a_2) = 1$, if $q_1$ is the inverse of $a_1$ modulo $n_1$, such that $0 < q_1 \leq n_1$, then
        \[\delta\p{\frac{n_1}{a_1},\frac{n_2}{a_2}} = \ceil{\frac{1}{n_1^2\p{\frac{a_1}{n_1} + \frac{a_2}{n_2} - 1}} - \frac{q_1}{n_1}}\]
    \end{proposition}

    \begin{proof}
        First, let us write
        \[q_1a_1 = 1 + n_1\e,\]
        and note that since $q_1 \leq n_1$ and $a_1 \leq n_1$, then we have $q_1 > \e$, $a_1 > \e$ and $\e \geq 0$.
        
        For this proof, we first note that $\delta$ remains invariant under the operation of appending the same integer to the left of $[\tilde e_1,\ldots,\tilde e_{\tilde \ell}]$ and $[f_1,\ldots,f_{\ell'}]$. Therefore it is enough to show that the left hand side of the equation is also invariant under the same operation, and that the the equality holds when either $e_1 > f_1$ or when $\frac{n_1}{n_1-a_1} = [~]$ ($e_1 < f_1$ cannot happen by \Cref{prefix-property}). In the latter case, we must have $n_1 = a_1 = 1$, and the expression evaluates to
        \[\ceil{\frac{n_2}{a_2} - 1}.\]
        This number is by definition $f_1 - 1$, which is what we needed to show.

        Now let us assume that $n_1 > q_1$ so the first expansion is non-empty, and $e_1 > f_1 \geq 1$. We need to show that the right hand side is zero. First note that since $\frac{q_1}{n_1} \geq 1$, then it is always true that
        \[\frac{1}{n_1^2\p{\frac{a_1}{n_1} + \frac{a_2}{n_2} - 1}} - \frac{q_1}{n_1} > -1,\]
        so we only need to show that
        \[\frac{1}{n_1^2\p{\frac{a_1}{n_1} + \frac{a_2}{n_2} - 1}} - \frac{q_1}{n_1} \leq 0,\]
        or in other words, that
        \[n_2 \leq q_1(a_1n_2 + a_2n_1 - n_1n_2).\]
        Using that $q_1a_1 = 1 + n_1\e$, this is equivalent to showing that
        \[\e n_2 - q_1(n_2-a_2) \geq 0.\]
        When $f_1 = 1$, that is, when $n_2 = a_2 = 1$, this inequality is trivially satisfied, so we may assume that $f_2 \geq 2$, in which case $\tilde e_1 \geq 3$. Also observe that this implies that $n_1 > 1$ and $a_1 > 1$. By construction of the continued fraction expansion, we must have $\frac{n_1}{n_1-a_1} > \tilde e_1 - 1$, and so
        \[\frac{n_1}{a_1} < 1 + \frac{1}{\tilde e_1 - 2}.\]
        As these are fractions of integers, this strict inequality implies that
        \[\frac{n_1}{a_1} \leq 1 + \frac{1}{\tilde e_1 - 2} - \frac{1}{a_1(\tilde e_1 - 2)}.\]
        Multiplying by $\e$, we obtain
        \begin{equation}\label{horrible-thing}
            q_1 - \frac{1}{a_1} = \frac{n_1\e}{a_1} \leq \e + \frac{\e}{\tilde e_1 - 2} - \frac{\e}{a_1(\tilde e_1-2)}.
        \end{equation}
        With this we obtain
        \begin{align*}
            \e n_2 - q_1(n_2-a_2) &\geq \e n_2 - \e(n_2-a_2) - \frac{\e(n_2-a_2)}{\tilde e_1-2} + \frac{\e(n_2-a_2)}{a_1(\tilde e_1 - 2)} - \frac{n_2-a_2}{a_1}\\
            &= \e\p{a_2 - \frac{n_2-a_2}{\tilde e_1-2}} + \frac{n_2-a_2}{a_1}\p{\frac{\e}{\tilde e_1-2} -1}\\
            &= \e a_2\p{\frac{\tilde e_1 - \frac{n_2}{a_2}-1}{\tilde e_1-2}} + \frac{n_2-a_2}{a_1}\p{\frac{\e}{\tilde e_1-2} -1}
        \end{align*}
        Now, since $\frac{n_2}{a_2} \leq f_1$, we further have
        \begin{align*}
            \e n_2 - q_1(n_2-a_2) &\geq \e a_2\p{\frac{\tilde e_1 - f_1 -1}{\tilde e_1-2}} + \frac{n_2-a_2}{a_1}\p{\frac{\e}{\tilde e_1-2} -1}
        \end{align*}
        As we are assuming that $\tilde e_1 > f_1$, this reduces our problem to showing that $\frac{\e}{\tilde e_1-2} \geq 1$. We can rearrange inequality \ref{horrible-thing} into
        \[\frac{\e}{\tilde e_1-2}\p{1 - \frac{1}{a_1}} \geq q_1 - \e - \frac{1}{a_1}.\]
        Recalling that $q_1 > \e$, we finally obtain
        \[\frac{\e}{\tilde e_1-2}\p{1 - \frac{1}{a_1}} \geq 1 - \frac{1}{a_1},\]
        and as $a_1 > 1$, this proves that $\frac{\e}{\tilde e_1 - 2} \geq 1$.

        Now let us show that the expression is invariant under appending a 2 to the left of both chains. This corresponds to replacing $\frac{n_1}{a_1}$ and $\frac{n_2}{a_2}$ by $\frac{n_1 + a_1}{a_1}$ and $\frac{2n_2 - a_2}{n_2}$ respectively. Using that $a_1q_1 = 1 + n_1\e$, we can write
        \[(q_1 + \e)a_1 = 1 + (n_1 + a_1)\e,\]
        and so an inverse of $a_1$ modulo $(n_1 + a_1)$ is $(q_1 + \e)$, and our inequalities for $\e$ imply that $0 < q_1 + \e < n_1 + a_1$, so this candidate of inverse is in fact in the form that this hypothesis requires. Under these changes, the expression inside the bracket in the right hand side evaluates to
        \[\frac{1}{(n_1 + a_1)^2\p{\frac{a_1}{n_1 + a_1} + \frac{n_2}{2n_2-a_2} - 1}} - \frac{q_1 + \e}{n_1 + a_1}\]
        A tedious and not very enlightening algebraic manipulation shows that this equals
        \[\frac{1}{n_1^2\p{\frac{a_1}{n_1} + \frac{a_2}{n_2} - 1}} - \frac{q_1}{n_1},\]
        so the expression is indeed invariant.

        Finally, let us show that the right hand side is also invariant under adding 1 to the left most entry of both expansions, where we need to assume that $\frac{n_1}{a_1} \neq 1$, or in other words, that $n_1 > a_1$, so that $\frac{n_1}{n_1 - a_1} \neq [~]$. This operation replaces $\frac{n_1}{a_1}$ and $\frac{n_2}{a_2}$ by $\frac{2n_1-a_1}{n_1}$ and $\frac{n_2 + a_2}{a_2}$ respectively. From $a_1q_1 = 1 + n_1\e$, we obtain
        \[(2q_1 - \e)n_1 = 1 + (2n_1 - a_1)q_1,\]
        so an inverse of $n_1$ modulo $(2n_1-a_1)$ is $(2q_1-\e)$. We claim that $0 < 2q_1-\e < 2n_1-a_1$. Indeed, since $q_1 > \e$ the first inequality is clear, and as for the second, we have
        \begin{align*}
            2q_1 - \e < 2n_1- a_1 &\iff 2q_1a_1 - \e a_1 < (2n_1-a_1)a_1\\
            &\iff 2 + \e(2n_1-a_1) < (2n_1-a_1)a_1\\
            &\iff 2 < (2n_1-a_1)(a_1-\e)
        \end{align*}
        This inequality is certainly true, as $a_1 > \e$ and since $2n_1-a_1 \geq 2(a_1 + 1) - a_1 = a_1 + 2 \geq 3$. This verification shows that $(2q_1-\e)$ is the inverse of $n_1$ that is in the form required in the hypothesis. The expression on the right for these new fractions then becomes
        \[\frac{1}{(2n_1-a_1)^2\p{\frac{n_1}{2n_1-a_1} + \frac{a_2}{n_2 + a_2} - 1}} - \frac{2q_1-\e}{2n_1-a_1}.\]
        Another computation shows that this equals
        \[\frac{1}{n_1^2\p{\frac{a_1}{n_1} + \frac{a_2}{n_2} - 1}} - \frac{q_1}{n_1},\]
        which finishes the proof.
    \end{proof}

    \begin{remark}
        In the proof we actually showed that the fraction in the expression itself is invariant under the operations. Considering the case when the first chain is indeed a prefix of the second one, we can deduce an arguably stronger result. It can be stated as follows.

        If $\frac{n}{a} = [e_1,\ldots,e_\ell]$ and $\frac{n'}{a'} = [e_1,\ldots,e_s]$ with $s < \ell$, then
        \[\frac{1}{(n')^2\p{\frac{a}{n}-\frac{a'}{n'}}} + \frac{q}{n'} = \frac{n+q(n'a-na')}{n'(n'a-na')}= [e_{s+1},\ldots,e_\ell],\]
        where $q$ is the inverse of $a'$ modulo $n'$. This result is a direct consequence of interpreting Hirzebruch-Jung continued fractions as products of matrices in $\SL(2,\Z)$.
    \end{remark}

    \begin{remark} \label{no-delta}
        Note that, assuming the fractions are reduced, if $n_1 \geq n_2$, then $\delta\p{\frac{n_1}{a_1},\frac{n_2}{a_2}} = 0$. In particular it is always zero when $\frac{n_1}{a_1} = \frac{n_2}{a_2}$, and at least one of $\delta\p{\frac{n_1}{a_1},\frac{n_2}{a_2}}$ and $\delta\p{\frac{n_2}{a_2},\frac{n_1}{a_1}}$ must be zero.
    \end{remark}

    \subsection{Non-swapping Local Glueing Algorithm} \label{non-swap-local-glueing-algorithm}
    We now describe an algorithm that determines the exceptional chains in $\ti X_0$ in terms of the exceptional chains over the corresponding points in $\ti X_0^{\nu}$. We interpret this as a ``glueing'' of exceptional chains, although this glueing is purely combinatorial, as we cannot make sense of the scheme structure of $\ti X_0$ just from looking at $\ti X_0^{\nu}$. An analog of this algorithm is \cite[Algorithm 5.16]{Dok25}, but the main tool for this construction is \Cref{non-swap-quotient-theorem}.

    For $i = 1,2$, let $n_i,a_i$ such that $0 < a_i \leq n_i$ are coprime, and let $q_i$ be the inverse of $a_i$ modulo $n_i$. Let $d_i$ be two integers such that $n_1d_1 = n_2d_2$, and let $n$ be that number. Write the continued fraction expansions of $\frac{n_1}{q_1}$ and $\frac{n_2}{q_2}$ as
    \[\frac{n_1}{q_1} = [e_1,\ldots,e_\ell], \qquad \frac{n_2}{q_2} = [f_1,\ldots,f_{\ell'}]\]
    Let $p_1 \in U_1$, $p_2 \in U_2$ and $p \in U$ be point in three (étale neighborhoods in) surfaces, which are singular points of types
    \[\frac{1}{n_1}(a_1,1), \qquad \frac{1}{n_2}(1,a_2), \qquad \frac{1}{n}(a_1d_1,a_2d_2,1;j)\]
    respectively. Let $L_1' \subseteq U_1$ be a curve corresponding to the $x$ axis through $p_1$, $L_2' \subseteq U_2$ be a curve corresponding to the $y$ axis through $p_2$, and $L_1,L_2 \subseteq U$ be curves through $p$ corresponding to the $x$ and $y$ axis respectively. Let $\rho \colon Z \to U$ be the minimal resolution at $p$, and let $\rho_i \colon Z_i \to U_i$ be the resolution at $p_i$ corresponding to the fraction $\frac{n_i}{a_i}$, that is, $\rho_i$ is a blow-up if $p_i$ is smooth, and the minimal resolution if it is not. Let $\alpha_0,\ldots,\alpha_\ell$ and $\beta_0,\ldots,\beta_{\ell'}$ be defined by
    \begin{align*}
        \alpha_0 &= n_1, & \alpha_1 &= q_1,  & \alpha_{i+1} &= e_i\alpha_i - \alpha_{i-1}\\
        \beta_0 &= n_2, & \beta_1 &= q_2,  & \beta_{i+1} &= f_i\beta_i - \beta_{i-1},
    \end{align*}
    so that the multiplicities and self intersections of the exceptional chain of $\rho_1^*(n_1L_1')$ and $\rho_2^*(n_2L_2')$ are 
    \begin{align*}
        \begin{tikzpicture}[
            dot/.style = {circle, fill, minimum size=6pt, inner sep=0pt, outer sep=0pt},
            square/.style = {rectangle, fill, minimum size=6pt, inner sep=0pt, outer sep=0pt}]
            \node[square] at (-1,0) {};
            \node[label=below:{$L_1'$}] at (-1,0) {};
            \node[label=above:{$n_1$}] at (-1,0) {};
            \draw (-1,0) -- (0,0);
            \node[dot] at (0,0) {};
            \node[label=above:{$\alpha_1$}] at (0,0) {};
            \node[label=below:{$(-e_1)$}] at (0,0) {};
            \draw (0,0) -- (0.5,0);
            \node[] at (1,0) {$\ldots$};
            \draw (1.5,0) -- (2,0);
            \node[dot] at (2,0) {};
            \node[label=above:{$\alpha_{\ell-1}$}] at (2,0) {};
            \node[label=below:{$(-e_{\ell-1})$}] at (2,0) {};
            \draw (2,0) -- (3.5,0);
            \node[dot] at (3.5,0) {};
            \node[label=above:{$1$}] at (3.5,0) {};
            \node[label=below:{$(-e_\ell)$}] at (3.5,0) {};
        \end{tikzpicture} &&
        \begin{tikzpicture}[
            dot/.style = {circle, fill, minimum size=6pt, inner sep=0pt, outer sep=0pt},
            square/.style = {rectangle, fill, minimum size=6pt, inner sep=0pt, outer sep=0pt}]
            \node[dot] at (-0.5,0) {};
            \node[label=above:{$1$}] at (-0.5,0) {};
            \node[label=below:{$(-f_{\ell'})$}] at (-0.5,0) {};
            \draw (-0.5,0) -- (1,0);
            \node[dot] at (1,0) {};
            \node[label=above:{$\beta_{\ell'-1}$}] at (1,0) {};
            \node[label=below:{$(-f_{\ell'-1})$}] at (1,0) {};
            \draw (1,0) -- (1.5,0);
            \node[] at (2,0) {$\ldots$};
            \draw (2.5,0) -- (3,0);
            \node[dot] at (3,0) {};
            \node[label=above:{$\beta_1$}] at (3,0) {};
            \node[label=below:{$(-f_1)$}] at (3,0) {};
            \draw (3,0) -- (4,0);
            \node[square] at (4,0) {};
            \node[label=below:{$L_2'$}] at (4,0) {};
            \node[label=above:{$n_2$}] at (4,0) {};
        \end{tikzpicture},
    \end{align*}
    where we recall that $\alpha_\ell = \beta_{\ell'} = 1$ (\Cref{remark:last-multiplicity-is-one}).

    \begin{theorem} The minimal resolution of $\frac{1}{n}(a_1d_1,a_2d_2,1;j)$ and the multiplicities in $\rho^*(n_1L_1 + n_2L_2)$ are the following depending on $j$.
        \begin{enumerate}
            \item If $j > 0$, they are given by
            \begin{align*}
                \begin{tikzpicture}[
                    dot/.style = {circle, fill, minimum size=6pt, inner sep=0pt, outer sep=0pt},
                    square/.style = {rectangle, fill, minimum size=6pt, inner sep=0pt, outer sep=0pt}]
                    \node[square] at (-2,0) {};
                    \node[label=below:{$L_1$}] at (-2,0) {};
                    \node[label=above:{$n_1$}] at (-2,0) {};
                    \draw (-2,0) -- (-1,0);
                    \node[dot] at (-1,0) {};
                    \node[label=above:{$\alpha_1$}] at (-1,0) {};
                    \node[label=below:{$(-e_1)$}] at (-1,0) {};
                    \draw (-1,0) -- (-0.5,0);
                    \node[] at (0,0) {$\ldots$};
                    \draw (0.5,0) -- (1,0);
                    \node[dot] at (1,0) {};
                    \node[label=above:{$\alpha_{\ell-1}$}] at (1,0) {};
                    \node[label=below:{$(-e_{\ell-1})$}] at (1,0) {};
                    \draw (1,0) -- (2.5,0);
                    \node[dot] at (2.5,0) {};
                    \node[label=above:{$1$}] at (2.5,0) {};
                    \node[label=below:{$(-e_\ell-1)$}] at (2.5,0) {};
                    \draw (2.5,0) -- (4,0);
                    \node[dot] at (4,0) {};
                    \node[label=above:{$1$}] at (4,0) {};
                    \node[label=below:{$(-2)$}] at (4,0) {};
                    \draw (4,0) -- (4.5,0);
                    \node[] at (5,0) {$\ldots$};
                    \draw (5.5,0) -- (6,0);
                    \node[dot] at (6,0) {};
                    \node[label=above:{$1$}] at (6,0) {};
                    \node[label=below:{$(-2)$}] at (6,0) {};
                    \draw[thick, decoration={brace},decorate] (4,0.75) -- (6,0.75) node[above,pos=0.5] {$j-1$};
                    \draw (6,0) -- (7.5,0);
                    \node[dot] at (7.5,0) {};
                    \node[label=above:{$1$}] at (7.5,0) {};
                    \node[label=below:{$(-f_{\ell'}-1)$}] at (7.5,0) {};
                    \draw (7.5,0) -- (9,0);
                    \node[dot] at (9,0) {};
                    \node[label=above:{$\beta_{\ell'-1}$}] at (9,0) {};
                    \node[label=below:{$(-f_{\ell'-1})$}] at (9,0) {};
                    \draw (9,0) -- (9.5,0);
                    \node[] at (10,0) {$\ldots$};
                    \draw (10.5,0) -- (11,0);
                    \node[dot] at (11,0) {};
                    \node[label=above:{$\beta_1$}] at (11,0) {};
                    \node[label=below:{$(-f_1)$}] at (11,0) {};
                    \draw (11,0) -- (12,0);
                    \node[square] at (12,0) {};
                    \node[label=below:{$L_2$}] at (12,0) {};
                    \node[label=above:{$n_2$}] at (12,0) {};
                \end{tikzpicture}
            \end{align*}
            \item If $j = 0$, they are given by
            \begin{align*}
                \begin{tikzpicture}[
                    dot/.style = {circle, fill, minimum size=6pt, inner sep=0pt, outer sep=0pt},
                    square/.style = {rectangle, fill, minimum size=6pt, inner sep=0pt, outer sep=0pt}]
                    \node[square] at (-1,0) {};
                    \node[label=below:{$L_1$}] at (-1,0) {};
                    \node[label=above:{$n_1$}] at (-1,0) {};
                    \draw (-1,0) -- (0,0);
                    \node[dot] at (0,0) {};
                    \node[label=above:{$\alpha_1$}] at (0,0) {};
                    \node[label=below:{$(-e_1)$}] at (0,0) {};
                    \draw (0,0) -- (0.5,0);
                    \node[] at (1,0) {$\ldots$};
                    \draw (1.5,0) -- (2,0);
                    \node[dot] at (2,0) {};
                    \node[label=above:{$\alpha_{\ell-1}$}] at (2,0) {};
                    \node[label=below:{$(-e_{\ell-1})$}] at (2,0) {};
                    \draw (2,0) -- (4,0);
                    \node[dot] at (4,0) {};
                    \node[label=above:{$1$}] at (4,0) {};
                    \node[label=below:{$(-e_\ell-f_{\ell'})$}] at (4,0) {};
                    \draw (4,0) -- (6,0);
                    \node[dot] at (6,0) {};
                    \node[label=above:{$\beta_{\ell'-1}$}] at (6,0) {};
                    \node[label=below:{$(-f_{\ell'-1})$}] at (6,0) {};
                    \draw (6,0) -- (6.5,0);
                    \node[] at (7,0) {$\ldots$};
                    \draw (7.5,0) -- (8,0);
                    \node[dot] at (8,0) {};
                    \node[label=above:{$\beta_1$}] at (8,0) {};
                    \node[label=below:{$(-f_1)$}] at (8,0) {};
                    \draw (8,0) -- (9,0);
                    \node[square] at (9,0) {};
                    \node[label=below:{$L_2$}] at (9,0) {};
                    \node[label=above:{$n_2$}] at (9,0) {};
                \end{tikzpicture}
            \end{align*}
            \item If $j = -1$ and $\frac{n_1}{a_1} = \frac{n_2}{a_2} = 1$, that is, $n_1=a_1=n_2=a_2=1$, then $U$ is already smooth, so the resolution is an isomorphism and there are no multiplicities to compute.
            \item If $j = -1$ but at least one of $\frac{n_1}{a_2}$ or $\frac{n_2}{a_2}$ is not 1,  consider the sequence
            \[[e_1,\ldots,e_{\ell}, 1, f_{\ell'},\ldots,f_1],\]
            and ``blow down'' successively whenever there is a single 1 in this sequence. The condition $\frac{a_1}{n_1} + \frac{a_2}{n_2} > 1$ is equivalent to this process eventually reaching the configuration
            \[[e_1,\ldots,e_{r},1,1,f_{s},\ldots,f_1]\]
            for some $0 \leq r \leq \ell$ and $0 \leq s \leq \ell'$ which cannot be both zero at the same time, where if $r = 0$ or $s = 0$, then the sequence either starts or ends with $[1,1]$ respectively. Then the self intersections and multiplicities of $\rho^*(n_1L_1 + n_2L_2)$ are as follows:
            \begin{itemize}
                \item If $r > 0$ and $s > 0$, then $\alpha_r = \beta_s$ and
                \begin{align*}
                    \begin{tikzpicture}[
                        dot/.style = {circle, fill, minimum size=6pt, inner sep=0pt, outer sep=0pt},
                        square/.style = {rectangle, fill, minimum size=6pt, inner sep=0pt, outer sep=0pt}]
                        \node[square] at (-1,0) {};
                        \node[label=below:{$L_1$}] at (-1,0) {};
                        \node[label=above:{$n_1$}] at (-1,0) {};
                        \draw (-1,0) -- (0,0);
                        \node[dot] at (0,0) {};
                        \node[label=above:{$\alpha_1$}] at (0,0) {};
                        \node[label=below:{$(-e_1)$}] at (0,0) {};
                        \draw (0,0) -- (0.5,0);
                        \node[] at (1,0) {$\ldots$};
                        \draw (1.5,0) -- (2,0);
                        \node[dot] at (2,0) {};
                        \node[label=above:{$\alpha_{r-1}$}] at (2,0) {};
                        \node[label=below:{$(-e_{r-1})$}] at (2,0) {};
                        \draw (2,0) -- (4,0);
                        \node[dot] at (4,0) {};
                        \node[label=above:{$\alpha_r = \beta_s$}] at (4,0) {};
                        \node[label=below:{$(-e_r-f_{s} + 1)$}] at (4,0) {};
                        \draw (4,0) -- (6,0);
                        \node[dot] at (6,0) {};
                        \node[label=above:{$\beta_{s-1}$}] at (6,0) {};
                        \node[label=below:{$(-f_{s-1})$}] at (6,0) {};
                        \draw (6,0) -- (6.5,0);
                        \node[] at (7,0) {$\ldots$};
                        \draw (7.5,0) -- (8,0);
                        \node[dot] at (8,0) {};
                        \node[label=above:{$\beta_1$}] at (8,0) {};
                        \node[label=below:{$(-f_1)$}] at (8,0) {};
                        \draw (8,0) -- (9,0);
                        \node[square] at (9,0) {};
                        \node[label=below:{$L_2$}] at (9,0) {};
                        \node[label=above:{$n_2$}] at (9,0) {};
                    \end{tikzpicture}
                \end{align*}
                Also, $\delta\p{\frac{n_1}{a_1},\frac{n_2}{a_2}} = \delta\p{\frac{n_2}{a_2}, \frac{n_1}{a_1}} = 0$.
                \item If $r = 0$ and $s > 0$, then $\beta_s = n_1$ and
                
                \begin{align*}
                    \begin{tikzpicture}[
                        dot/.style = {circle, fill, minimum size=6pt, inner sep=0pt, outer sep=0pt},
                        square/.style = {rectangle, fill, minimum size=6pt, inner sep=0pt, outer sep=0pt}]
                        \node[square] at (5,0) {};
                        \node[label=above:{$n_1$}] at (5,0) {};
                        \node[label=below:{$L_1$}] at (5,0) {};
                        \draw (5,0) -- (6,0);
                        \node[dot] at (6,0) {};
                        \node[label=above:{$\beta_{s-1}$}] at (6,0) {};
                        \node[label=below:{$(-f_{s-1})$}] at (6,0) {};
                        \draw (6,0) -- (6.5,0);
                        \node[] at (7,0) {$\ldots$};
                        \draw (7.5,0) -- (8,0);
                        \node[dot] at (8,0) {};
                        \node[label=above:{$\beta_1$}] at (8,0) {};
                        \node[label=below:{$(-f_1)$}] at (8,0) {};
                        \draw (8,0) -- (9,0);
                        \node[square] at (9,0) {};
                        \node[label=below:{$L_2$}] at (9,0) {};
                        \node[label=above:{$n_2$}] at (9,0) {};
                    \end{tikzpicture}
                \end{align*}
                Here, if $s = 1$, then $p \in U$ is smooth. Either way, $\delta\p{\frac{n_1}{a_1},\frac{n_2}{a_2}} = f_s-1 > 0$ and $\delta\p{\frac{n_2}{a_2},\frac{n_1}{a_1}} = 0$.
                \item If $r > 0$ and $s = 0$, then $\alpha_r = n_2$ and
                \begin{align*}
                    \begin{tikzpicture}[
                        dot/.style = {circle, fill, minimum size=6pt, inner sep=0pt, outer sep=0pt},
                        square/.style = {rectangle, fill, minimum size=6pt, inner sep=0pt, outer sep=0pt}]
                        \node[square] at (-1,0) {};
                        \node[label=below:{$L_1$}] at (-1,0) {};
                        \node[label=above:{$n_1$}] at (-1,0) {};
                        \draw (-1,0) -- (0,0);
                        \node[dot] at (0,0) {};
                        \node[label=above:{$\alpha_1$}] at (0,0) {};
                        \node[label=below:{$(-e_1)$}] at (0,0) {};
                        \draw (0,0) -- (0.5,0);
                        \node[] at (1,0) {$\ldots$};
                        \draw (1.5,0) -- (2,0);
                        \node[dot] at (2,0) {};
                        \node[label=above:{$\alpha_{r-1}$}] at (2,0) {};
                        \node[label=below:{$(-e_{r-1})$}] at (2,0) {};
                        \draw (2,0) -- (3,0);
                        \node[square] at (3,0) {};
                        \node[label=above:{$n_2$}] at (3,0) {};
                        \node[label=below:{$L_2$}] at (3,0) {};
                    \end{tikzpicture}
                \end{align*}
                Here, if $r = 1$, then $p \in U$ is smooth. Either way, $\delta\p{\frac{n_1}{a_1},\frac{n_2}{a_2}} = 0$ and $\delta\p{\frac{n_2}{a_2},\frac{n_1}{a_1}} = e_r-1 > 0$.
            \end{itemize}
        \end{enumerate}
    \end{theorem}

    \begin{proof}
        When $j \geq 0$, the statement about the self intersection of the exceptional curves in the minimal resolution follows directly from applying \Cref{dual-property} to the dual chain statement in \Cref{non-swap-quotient-theorem}. If $j > 0$, we let $\eta_i,\varphi_i$ and $\lambda_i$ be the multiplicities of the curves in $\rho^*(n_1L_1 + n_2L_2) \subseteq Z$ in the following way
        \begin{align*}
            \begin{tikzpicture}[
                dot/.style = {circle, fill, minimum size=6pt, inner sep=0pt, outer sep=0pt},
                square/.style = {rectangle, fill, minimum size=6pt, inner sep=0pt, outer sep=0pt}]
                \node[square] at (-2,0) {};
                \node[label=below:{$L_1$}] at (-2,0) {};
                \node[label=above:{$n_1$}] at (-2,0) {};
                \draw (-2,0) -- (-1,0);
                \node[dot] at (-1,0) {};
                \node[label=above:{$\eta_1$}] at (-1,0) {};
                \node[label=below:{$(-e_1)$}] at (-1,0) {};
                \draw (-1,0) -- (-0.5,0);
                \node[] at (0,0) {$\ldots$};
                \draw (0.5,0) -- (1,0);
                \node[dot] at (1,0) {};
                \node[label=above:{$\eta_{\ell-1}$}] at (1,0) {};
                \node[label=below:{$(-e_{\ell-1})$}] at (1,0) {};
                \draw (1,0) -- (2.5,0);
                \node[dot] at (2.5,0) {};
                \node[label=above:{$\eta_\ell$}] at (2.5,0) {};
                \node[label=below:{$(-e_\ell-1)$}] at (2.5,0) {};
                \draw (2.5,0) -- (4,0);
                \node[dot] at (4,0) {};
                \node[label=above:{$\lambda_1$}] at (4,0) {};
                \node[label=below:{$(-2)$}] at (4,0) {};
                \draw (4,0) -- (4.5,0);
                \node[] at (5,0) {$\ldots$};
                \draw (5.5,0) -- (6,0);
                \node[dot] at (6,0) {};
                \node[label=above:{$\lambda_{j-1}$}] at (6,0) {};
                \node[label=below:{$(-2)$}] at (6,0) {};
                \draw[thick, decoration={brace},decorate] (4,0.75) -- (6,0.75) node[above,pos=0.5] {$j-1$};
                \draw (6,0) -- (7.5,0);
                \node[dot] at (7.5,0) {};
                \node[label=above:{$\varphi_{\ell'}$}] at (7.5,0) {};
                \node[label=below:{$(-f_{\ell'}-1)$}] at (7.5,0) {};
                \draw (7.5,0) -- (9,0);
                \node[dot] at (9,0) {};
                \node[label=above:{$\varphi_{\ell'-1}$}] at (9,0) {};
                \node[label=below:{$(-f_{\ell'-1})$}] at (9,0) {};
                \draw (9,0) -- (9.5,0);
                \node[] at (10,0) {$\ldots$};
                \draw (10.5,0) -- (11,0);
                \node[dot] at (11,0) {};
                \node[label=above:{$\varphi_1$}] at (11,0) {};
                \node[label=below:{$(-f_1)$}] at (11,0) {};
                \draw (11,0) -- (12,0);
                \node[square] at (12,0) {};
                \node[label=below:{$L_2$}] at (12,0) {};
                \node[label=above:{$n_2$}] at (12,0) {};
            \end{tikzpicture}
        \end{align*}
        We also define $\lambda_0 = \eta_\ell$, $\lambda_j = \varphi_{\ell'-1}$. When $j = 0$, we do the same, so that $\eta_\ell = \varphi_{\ell'}$.
        
        By \Cref{non-swap-quotient-theorem}, the singularity at $p$ is of type
        \[\frac{1}{\frac{nk}{d_1d_2}}\p{1,\frac{q_1k - d_1}{d_2}} = \frac{1}{\frac{nk}{d_1d_2}}\p{\frac{q_2k - d_2}{d_1},1}.\]
        Since $j \geq 0$, we have that $k \geq n$ and thus, both of these representations of the singularity are in standard form. This implies, by \Cref{chain-multiplicities}, that the multiplicity of the first curve in $\rho^*(L_1)$ and $\rho^*(L_2)$ are
        \[\frac{1}{\frac{nk}{d_1d_2}}\cdot \frac{q_1k - d_1}{d_2} = \frac{q_1k - d_1}{n_1k}, \qquad \text{ and } \qquad \frac{1}{\frac{nk}{d_1d_2}} = \frac{d_1}{n_2k},\]
        respectively. Therefore, $\eta_1$, which the multiplicity of this curve in $\rho^*(n_1L_1 + n_2L_2)$, is $q_1$ which equals $\alpha_1$ by \Cref{chain-multiplicities}. Now that we know that the first multiplicity agrees with the one for $\rho_1^*(n_1L_1')$, it follows that $\eta_i = \alpha_i$ for every $i = 1,\ldots,\ell$. This is because both $\{\eta_i\}$ and $\{\alpha_i\}$ satisfy the equations
        \[\eta_{i+1} = e_i\eta_i - \eta_{i-1}, \qquad \alpha_{i+1} = e_i\alpha_i - \alpha_{i-1}\]
        and they agree at the terms with $i = 0$ and $i = 1$. A similar argument shows that $\varphi_i = f_i$ for all $i = 1,\ldots,\ell'$, and proves the statement when $0 \leq j \leq 1$. Now assume $j \geq 2$, so we need to show that $\lambda_1 = \ldots = \lambda_{j-1}$, and we already know that $\lambda_0 = \lambda_j = 1$. Since these are $(-2)$-curves, they satisfy the system of equations
        \[\lambda_{i+1} = 2\lambda_i - \lambda_{i-1},\]
        which we can rewrite as
        \[\lambda_{i+1} - \lambda_i = \lambda_i - \lambda_{i-1}.\]
        This means that the multiplicities change in a linear fashion as $i$ increases. However, since the first $\lambda_0$ and last $\lambda_j$ are equal to 1, this means all the $\lambda_i$ in the middle must be equal to 1 too.

        If $j = -1$ and $n_1=a_1=n_2=a_2 = 1$, it follows that $k = n = d_1 = d_2$, and so $\frac{nk}{d_1d_2} = 1$, so $U$ is smooth.

        Now, still for $j = -1$, we assume that not both of $\frac{n_1}{a_1}$ and $\frac{n_2}{a_2}$ are equal to $1$, and we write
        \[\frac{n_1}{n_1 - a_1} = [\ti e_{\ti \ell},\ldots,\ti e_1], \qquad \frac{n_2}{n_2-a_2} = [\ti f_{\ti \ell'},\ldots, \ti f_1]\]
        for the expansions of the dual fractions. Note that the numbering of these coefficients are opposite to the one in \Cref{non-swap-quotient-theorem}. Let $0 \leq u \leq \ti \ell$ be the largest integer such that $[\ti e_{\ti \ell},\ldots, \ti e_{\ti \ell-u+1}]$ is a prefix of $[f_{\ell'},\ldots,f_1]  = \frac{n_2}{a_2}$. By \Cref{prefix-property}, the condition $\frac{a_1}{n_1} + \frac{a_2}{n_2} > 1$ is equivalent to $u = \ti \ell$ so $\frac{n_1}{n_1-a_1}$ is itself a prefix of $\frac{n_2}{a_2}$, or $\ti e_{\ti \ell - u} > f_{\ell' - u}$.
        
        Now note that for a similar argument as we saw in \Cref{dual-property}, the blow-down operations on 
        \[[e_1,\ldots,e_\ell,1,f_{\ell'},\ldots,f_1] \qquad \text{ and } \qquad [\ti e_1,\ldots, \ti e_{\ti \ell}, 1, \ti f_{\ti \ell'},\ldots,\ti f_1]\]
        ``mirror'' each other. By this, we mean that as long as there is exactly one $1$ in the middle for both, if we blow-down both sequences, the right sides (resp. left sides) will remain dual to each other. Note that when blowing one of them down, and we arrive to a configuration containing $[1,1]$, in the other one there will not appear another $1$.

        First suppose that $u = \ti \ell$ so $\frac{n_1}{n_1-a_1}$ is a prefix of $\frac{n_2}{a_2}$. They by \Cref{contraction-of-duals}, blowing down completely the first chain, we must eventually arrive to
        \[[1,1,f_{\ell'-u},\ldots,f_1]\]
        By this mirroring property and by \Cref{dual-property}, at this step, the other sequence must look like
        \[[1,\underbrace{2,\ldots,2}_{f_{\ell'-u}-2},\ti f_{\ti \ell' - v},\ldots,\ti f_1]\]
        where if $\ti \ell' - v > 1$ then $\ti f_{\ti \ell' - v} > 2$. Contracting the initial 1 and subsequently the chain of $2$'s, we arrive to the dual chain
        \[[\ti f_{\ti \ell' - v} - 1,\ldots,\ti f_1],\]
        which by \Cref{dual-removal-property}, must correspond to
        \[[f_{\ell'-u-1},\ldots f_1].\]
        Note that if $\ti \ell' - v = 1$ and $\ti f_{\ti \ell' - v} = 2$, this chain is empty. This covers the self intersections when $r = 0$ as well as the computation of $\delta\p{\frac{n_1}{a_1},\frac{n_2}{a_2}}$.

        Now suppose that $u > \ti \ell$ and $\ti e_{\ti \ell-u} > f_{\ell'-u}$. Letting $\lambda = \ti e_{\ti \ell - u} - f_{\ell' - u}$, after blowing-down enough times the first chain, we obtain
        \[[e_1,\ldots,e_{r'},\underbrace{2,\ldots,2}_{\ti e_{\ti \ell - u}-2},1,\ti e_{\ti \ell - u}-\lambda, f_{\ell'-u-1},\ldots,f_1]\]
        where if $r' > 1$ then $e_{r'} > 2$. Blowing down $\ti e_{\ti \ell - u} - \lambda - 1$ more times, we finally obtain
        \[[e_1,\ldots,e_{r'},\underbrace{2,\ldots,2}_{\lambda-1},1,1, f_{\ell'-u-1},\ldots,f_1] = [e_1,\ldots,e_r,1,1,f_s,\ldots,f_1].\]
        If $s > 0$, this must be mirrored by the chain
        \[[\ti e_1,\ldots,\ti e_{\ti \ell - u-1}, \lambda + 1,\underbrace{2,\ldots,2}_{f_s-2},\ti f_{v},\ldots, \ti f_1],\]
        whose dual is
        \[[e_1,\ldots,e_r + f_s-1,\ldots,f_1].\]
        If $s = 0$, it is instead mirrored by
        \[[\ti e_1,\ldots,\ti e_{\ti \ell - u-1}, \lambda + 1],\]
        which using \Cref{dual-removal-property} one can see that its dual is
        \[[e_1,\ldots,e_{r-1}]\]
        independently of whether $\lambda = 1$ or not. This covers the self intersections and $\delta\p{\frac{n_1}{a_1},\frac{n_2}{a_2}}$ in the rest of the cases. For computing $\delta\p{\frac{n_2}{a_2},\frac{n_1}{a_1}}$ we do the same analysis but instead of dualizing the left chain at the start, we dualize the right one.

        Finally, to compute the multiplicities, we apply the exact same argument as in the case $j \geq 0$, but first we assume that $(r,s) \neq (0,1)$ or $(1,0)$ so the chain is non-empty. The only thing we need to be careful in this case is that we used that the singularities
        \[\frac{1}{\frac{nk}{d_1d_2}}\p{1,\frac{q_1k - d_1}{d_2}} = \frac{1}{\frac{nk}{d_1d_2}}\p{\frac{q_2k - d_2}{d_1},1}\]
        are represented in standard form, where $k = a_1d_1 + a_2d_2 - n$. However, by \Cref{deficit-proposition}, $\delta\p{\frac{n_1}{a_1},\frac{n_2}{a_2}}$ is defined in such a way that
        \[0 \leq \frac{q_1k -d_1}{d_2} + \delta\p{\frac{n_1}{a_1},\frac{n_2}{a_2}}\frac{nk}{d_1d_2} < \frac{nk}{d_1d_2},\]
        so the first singularity is written in standard if and only if $\delta\p{\frac{n_1}{a_1},\frac{n_2}{a_2}} = 0$, and a similar story happens for the second one. This is just enough to make the argument work.
        
        When $(r,s) = (0,1)$ (the case $(1,0)$ being analogous), the argument does not work as the chain is empty, but we still wish to show that $q_2 = n_1$. Since $p \in U$ is smooth, we must have
        \[n(a_1d_1 + a_2d_2 -n) = d_1d_2,\]
        which implies that
        \[n_2a_1 + n_1a_2 = 1 + n_1n_2.\]
        This shows that $n_1$ is an inverse of $a_2$ modulo $n_2$. By \Cref{no-delta}, $n_1 < n_2$, which implies that $n_1 = q_2$.
    \end{proof}

    \begin{example}
        Let us show some examples of the contraction when $j = -1$ only knowing the continued fractions of $\frac{n_i}{a_i}$, but without actually computing these numbers. Consider the sequences
        \[[3,2,3] \qquad \text{ and } \qquad [3,2,3,2,2].\]
        Now, we flip the second one and connect them with a 1. Afterwards, we contract until we run out of $1$'s, or until we see that the sequence contains $[1,1]$.
        \begin{align*}
            &\quad[3,2,3,1,2,2,3,2,3]\\
            \rightsquigarrow & \quad[3,2,2,1,2,3,2,3]\\
            \rightsquigarrow & \quad[3,2,1,1,3,2,3].
        \end{align*}
        Since we stopped this process with a $[1,1]$, we now know that these two sequences can be glued together with $j=-1$, and the resulting chain is
        \[[3,4,2,3].\]
        As we did not blow-down completely any of the chains, then $\delta\p{\frac{n_1}{a_1},\frac{n_2}{a_2}} = \delta\p{\frac{n_2}{a_2},\frac{n_1}{a_1}} = 0$.

        Now consider the sequences
        \[[2,2,3] \qquad \text{ and } \qquad [2,5,4,2].\]
        Again, we flip the second, connect them and contract.
        \begin{align*}
            &\quad[2,2,3,1,2,4,5,2]\\
            \rightsquigarrow & \quad[2,2,2,1,4,5,2]\\
            \rightsquigarrow & \quad[2,2,1,3,5,2]\\
            \rightsquigarrow & \quad[2,1,2,5,2]\\
            \rightsquigarrow & \quad[1,1,5,2].
        \end{align*}
        Since the process stops with $[1,1]$, we know we can glue them together with $j=-1$ and the resulting chain is
        \[[2].\]
        The left chain was completely blown-down in the process. Thus $\delta\p{\frac{n_1}{a_1},\frac{n_2}{a_2}} = 4$ and $\delta\p{\frac{n_2}{a_2},\frac{n_1}{a_1}} = 0$.
    \end{example}

    \begin{remark}
        In our algorithm, $j$ plays the role of the depth in \cite[Definition 4.3]{Dok25}. However, the depth and $j$ disagree in the cases when $n_1 = 1$ or $n_2 = 1$. This is because we use the convention $0 < a_i \leq n_i$, instead of the convention $0 \leq a_i < n_i$ (see \cite[Definition 5.8]{Dok25} for $i(d,m)$). According to \cite[Definition 5.1]{Dok25}, if the depth is $n$, then $n+1$ is the number of multiplicity 1 components counting the curves $L_1,L_2$ and the chain joining them. For us, $j+1$ is the number of multiplicity 1 components only on the chain joining them.
    \end{remark}

    \begin{remark}
        We can think of the combinatorial operation $[\ldots,e_r,1,1,f_s,\ldots] \rightsquigarrow [\ldots,e_r+f_s-1,\ldots]$ as a topological contraction of the corresponding sphere plumbing. As opposed to the traditional blow-down, one cannot contract the two curves to a point as the result is not a smooth manifold, however we can do a different surgery.
        
        We claim that tubular neighborhoods of the union of the two $(-1)$-spheres have boundary $S^1 \times S^2$. Indeed, one may contract, say, the first $(-1)$-sphere to obtain a plumbing corresponding to the sequence $[\ldots,e_r-1,0,f_s,\ldots]$, and a neighborhood of a $(0)$-sphere looks like $D^2 \times S^2$. Now, remove this neighborhood and attach instead $S^1 \times D^3$. Notice that in the resulting manifold, the branches that belonged to the spheres having self intersections $(-e_r+1)$ and $(-f_s)$ must be connected along $S^1 \times D^3$ and thus, form part of the same sphere. We obtain a new plumbing of spheres, where this new sphere turns out to have Euler number equal to the sum of the euler numbers of the previous spheres, which is $-(e_r + f_s - 1)$, so it corresponds to the sequence $[\ldots,e_r+f_s-1,\ldots]$.
        
        However, in general this construction never preserves the complex structure.
    \end{remark}

    \begin{proposition}
        Let $X$ be an oriented compact 4-manifold containing a sphere $\Sigma$ of self intersection $0$ which is non-zero in $H_2(X,\Q)$. Let $Y$ be obtained from $X$ by the surgery described above. Then $X$ and $Y$ cannot simultaneously admit a complex structure.
    \end{proposition}

    \begin{proof}
        Assume they are both complex surfaces. Let $B = S^1 \times S^2$ be the boundary described above. Using that $\Sigma$ is not zero in $H_2(X,\Q)$, a Mayer-Vietoris argument gives that $H_3(Y) \cong H_3(X)$, $H_1(X - \Sigma) \cong H_1(Y)$, and the exact sequences
        \[0 \to H_2(X-\Sigma) \to H_2(X) \xrightarrow{\rho} H_1(B) \to H_1(X - \Sigma) \to H_1(X) \to 0,\]
        and
        \[0 \to H_2(B) = \Z[\Sigma] \to H_2(X - \Sigma) \to H_2(Y) \to 0.\]
        Here, $\rho$ is intersection with the boundary $B$. By Poincaré duality, it must have full rank as $b_1(Y) = b_1(X)$, so we obtain
        \[0 \to H_2(X - \Sigma) \to H_2(X) \xrightarrow{\cdot \Sigma} \Z,\]
        where we interpret the last map as intersection pairing with $\Sigma$. We conclude that, as intersection forms,
        \[H_2(Y) \cong [\Sigma]^\perp/\Z[\Sigma].\]
        Computing rank and index for the two intersection forms one obtains $r(H_2(Y)) = r(H_2(X)) - 2$ and $\tau(H_2(Y)) = \tau(H_2(X))$, and therefore $b_2^+(Y) = b_2^+(X) - 1$.

        Using the Thom-Hirzebruch Index Theorem \cite[Theorem I.3.1]{BHPV} and the Noether formula, we obtain
        \[b_2^+(X) = 2\chi(X) - 1 - b_1(X),\]
        and similarly for $Y$. In particular, $b_2^+(X) \equiv b_2^+(Y) \mod 2$. This is a contradiction.
    \end{proof}

    \subsection{Swapping Local Glueing Algorithm} \label{swap-local-glueing-algorithm}
    Here we describe how to obtain the exceptional divisor in $\ti X_0$ over a point obtained as a quotient of a node, whose branches are swapped by the action. This algorithm is essentially a rewrite of \Cref{swap-quotient-theorem}. We interpret this algorithm as a ``glueing'' of a chain in $\ti X_0^{\nu}$ to ``itself'', but it is easier to understand as attaching a ``D-tail'' \cite[Definition 10.1]{Dok25}.

    Let $0 < a \leq n$ be coprime, with $0 < q \leq n$ the inverse of $a$ modulo $n$ and let $u,v$ integers with $u + v = 2a$. Write the continued fraction expansion of $\frac{n}{q}$ as
    \[\frac{n}{q} = [e_1,\ldots,e_\ell].\]
    Let $p' \in U'$ and $p \in U$ be points of two (étale neighborhoods in) surfaces, which are singular points of type
    \[\frac{1}{n}(a,1), \qquad \frac{1}{2n}(2a,1;j)_D.\]
    Let $L' \subseteq U'$ (resp. $L \subseteq U$) be a line through $p'$ corresponding to the $x$ axis (resp. through $p$ corresponding to both axes in the quotient). Let $\rho \colon Z \to U$ be the minimal resolution and $\rho' \colon Z' \to U'$ be the resolution corresponding to $\frac{n}{q}$, that is, $\rho'$ is the blow-up if $p_i$ is smooth, and is otherwise the minimal resolution at $p$. Let $\alpha_0,\ldots,\alpha_{\ell}$ be defined by
    \[\alpha_0 = n, \qquad \alpha_1 = 1, \qquad \alpha_{i+1} = e_i\alpha_i - \alpha_{i-1},\]
    so that the multiplicities of the exceptional chain of $(\rho')^*(2nL')$ correspond to
    \begin{align*}
        \begin{tikzpicture}[
            dot/.style = {circle, fill, minimum size=6pt, inner sep=0pt, outer sep=0pt},
            square/.style = {rectangle, fill, minimum size=6pt, inner sep=0pt, outer sep=0pt}]
            \node[square] at (-1,0) {};
            \node[label=below:{$L'$}] at (-1,0) {};
            \node[label=above:{$2n$}] at (-1,0) {};
            \draw (-1,0) -- (0,0);
            \node[dot] at (0,0) {};
            \node[label=above:{$2\alpha_1$}] at (0,0) {};
            \node[label=below:{$(-e_1)$}] at (0,0) {};
            \draw (0,0) -- (0.5,0);
            \node[] at (1,0) {$\ldots$};
            \draw (1.5,0) -- (2,0);
            \node[dot] at (2,0) {};
            \node[label=above:{$2\alpha_{\ell-1}$}] at (2,0) {};
            \node[label=below:{$(-e_{\ell-1})$}] at (2,0) {};
            \draw (2,0) -- (3.5,0);
            \node[dot] at (3.5,0) {};
            \node[label=above:{$2$}] at (3.5,0) {};
            \node[label=below:{$(-e_\ell)$}] at (3.5,0) {};
        \end{tikzpicture}
    \end{align*}
    \begin{theorem}
        In the above scenario, the self intersections and multiplicities of exceptional components in $\rho^*(2nL)$ are
        \[\begin{tikzpicture}[
            dot/.style = {circle, fill, minimum size=6pt, inner sep=0pt, outer sep=0pt},
            square/.style = {rectangle, fill, minimum size=6pt, inner sep=0pt, outer sep=0pt}]
            \node[square] at (-1,0) {};
            \node[label=below:{$L$}] at (-1,0) {};
            \node[label=above:{$2n$}] at (-1,0) {};
            \draw (-1,0) -- (0,0);
            \node[dot] at (0,0) {};
            \node[label=above:{$2\alpha_1$}] at (0,0) {};
            \node[label=below:{$(-e_1)$}] at (0,0) {};
            \draw (0,0) -- (0.5,0);
            \node[] at (1,0) {$\ldots$};
            \draw (1.5,0) -- (2,0);
            \node[dot] at (2,0) {};
            \node[label=above:{$2\alpha_{\ell-1}$}] at (2,0) {};
            \node[label=below:{$(-e_{\ell-1})$}] at (2,0) {};
            \draw (2,0) -- (3.5,0);
            \node[dot] at (3.5,0) {};
            \node[label=above:{$2$}] at (3.5,0) {};
            \node[label=below:{$(-(e_\ell+1))$}] at (3.5,0) {};
            \draw (3.5,0) -- (5,0);
            \node[dot] at (5,0) {};
            \node[label=above:{$2$}] at (5,0) {};
            \node[label=below:{$(-2)$}] at (5,0) {};
            \draw (5,0) -- (5.5,0);
            \node[] at (6,0) {$\ldots$};
            \draw (6.5,0) -- (7,0);
            \node[dot] at (7,0) {};
            \node[label=above:{$2$}] at (7,0) {};
            \node[label=below:{$(-2)$}] at (7,0) {};
            \draw[thick, decoration={brace},decorate] (5,1) -- (7,1) node[above,pos=0.5] {$j$};
            \draw (7,0) -- (7.866,0.5);
            \node[dot] at (7.866,0.5) {};
            \node[label=above:{$1$}] at (7.866,0.5) {};
            \node[label=right:{$(-2)$}] at (7.866,0.5) {};
            \draw (7,0) -- (7.866,-0.5);
            \node[dot] at (7.866,-0.5) {};
            \node[label=above:{$1$}] at (7.866,-0.5) {};
            \node[label=right:{$(-2)$}] at (7.866,-0.5) {};
        \end{tikzpicture}
    \]
    \end{theorem}

    \begin{proof}
        The statement about the self intersections is a direct application of \Cref{swap-quotient-theorem}. For the statement about the multiplicities, all we need to show is that $E\cdot \p{2n\ti L + \sum \lambda_i E_i} = 0$ for each exceptional $E$, where $\ti L$ is the strict transform of $L$ and the $\lambda_i$ are the coefficients appearing in the statement. This is true since the $\alpha_i$ satisfy the equations $\alpha_{i-1}-e_i\alpha_i + \alpha_{i+1} = 0$.
    \end{proof}

    \begin{definition}
        We say that an exceptional divisor as the one above is a \textit{D-tail}.
    \end{definition}

    Again, \Cref{non-swap-local-glueing-algorithm,swap-local-glueing-algorithm} let us understand the exceptional chains in the resolution $\ti X_0 \subseteq \ti X \to X/\mu_n$ in terms of the exceptional chains in the resolution $\ti X_0^\nu \subseteq \ti X^\nu \to X^\nu/\mu_n$. The next step is to analyze the global picture, most importantly, how the self intersection of the non-exceptional components of $\ti X_0$ compare to the ones in $\ti X_0^\nu$. Let us standardize our notation from above and introduce other that will be needed.

    \begin{definition} \label{big-definitions}
        Let $\mu_n \curvearrowright C \subseteq X \xrightarrow{\pi} D$ be an equivariant action where $n$ is coprime to $\chr K$. We let $C_i \subseteq C$ the components of $C$. We denote by $\ov C$ the quotient $C/\mu_n$ and by $\ov C_i$ the image of $C_i$. Define $\tau \colon X \to X^{tw} = X/\mu_n$ the quotient, and $\pi^{tw} \colon X^{tw} \to D/\mu_n$ the \textit{twisted} model of $X$, with central fiber $X^{tw}_0$. In particular,
        \[\ov C = X^{tw}_{0,red} \subseteq X^{tw}_0 \subseteq X^{tw}.\]
        Let $\rho \colon \ti X \to X^{tw}$ be the minimal resolution. We define $\ti X_0 = \rho^*X^{tw}_0$ the fiber over $0$ and we let $\ti C$, $\ti C_i$ the strict transforms of $\ov C$ and $\ov C_i$, respectively. We say that $\ti X_0$ is a \textit{tame singular fiber}, or just \textit{singular fiber} for our purposes. We say that the $\ti C_i \subseteq \ti X_0$ are the \textit{principal components} of $\ti X_0$.

        Let $C^\nu \to C$ be the normalization and $X^{\nu} = C^\nu \times D \xrightarrow{\pi^\nu} D$. Let $\mu_n$ act on $X^{\nu}$ by $\zeta_n(x,t) = (f(x),\zeta_n t)$, and let $\tau^\nu \colon X^\nu \to X^{\nu,tw} = X^\nu/\mu_n$. To each $C_i$ there corresponds a component $C^\nu_i$. We define $\ov C^\nu$, $\ov C^\nu_i$, $\pi^{\nu,tw}$, $X^{\nu,tw}_0$, $\ti C^\nu$, $\ti C^\nu_i$, $\ti X^\nu$, $\rho^\nu$ and $\ti X^\nu_0$ analogously.

        When $\ti X_0$ is connected and is obtained as from a quotient of $C \times D \to D$ where $C$ is a smooth, possibly disconnected curve, we say is $\ti X_0$ is a \textit{building block}. If moreover, $C$ is connected, we say that $\ti X_0$ is a \textit{primitive building block}.

        If $\ti X_0$ is a primitive block obtained as the quotient of $C$ by $f$, for every $c \geq 1$ we may define the building block $[c]\ti X_0$ obtained as the quotient of $C' = \bigsqcup_{i=1}^c C$ by the automorphism $f'$ which permutes cyclically the copies of $C$, and such that $(f')^c$ acts as $f$ on each component. Then $[c]\ti X_0$ is combinatorially identical to $\ti X_0$, except each multiplicity is multiplied by $c$.

        In general, $\ti X^\nu_0$ is a collection of building blocks. We say that $\ti X^\nu_{0,i}$ is the building block corresponding to $C_i$.
    \end{definition}

    \begin{remark}
        Our definition of a principal component differs slightly from the definition given in \cite{Dok25}. For us, a principal component is the strict transform of a curve in $\ov C$, and thus, depends on the particular presentation of $\ti X_0$ as a quotient, meanwhile \cite{Dok25} defines them as those components which have positive genus or intersect at least 3 other components. This means that, for instance, the curve intersecting three components in a D-tail is a principal component for \cite{Dok25} but is not a principal component for us, if it was obtained by ``glueing of a point to itself'' in the sense of \Cref{swap-local-glueing-algorithm}.
    \end{remark}

    \begin{definition} \label{numbers}
        Let $\mu_n \curvearrowright C \subseteq X \to D$ as before inducing $f \in \Aut(C)$. Let $C_i \subseteq C$ be a component of $C$. We define the following numbers:
        \begin{align*}
            \gamma_i &= \min\left\{\gamma > 0 \mid f^\gamma(C_i) = C_i\right\}, & m_i &= \min \left\{m > 0 \mid f^m|_{C_i} = \id_{C_i}\right\}, & d_i &= \frac{n}{m_i}.
        \end{align*}
        Now if $p \in C$ is a point, which is not necessarily a node, define
        \begin{align*}
            h_p &= \min\left\{h > 0 \mid f^h(p) = p\right\} & n_p &= \frac{n}{h_p}
        \end{align*}
        If moreover, a given component $C_i$ contains $p$, also define
        \begin{align*}
            n_{i,p} &= \begin{cases}
                \ds\frac{m_i}{h_p} &  \text{if $p$ is not a node, or if $p$ is a node and $f^{h_p}$ does not swap the branches at $p$},\\ \\
                \ds\frac{m_i}{2h_p} & \text{if $p$ is a node and $f^{h_p}$ swaps the branches at $p$}.
            \end{cases}
        \end{align*}
        Finally, if $p$ is a node in $C$ which is a singularity of $A_{k-1}$ in $X$, we let $k_p = k$.
    \end{definition}

    \begin{remark}
        These numbers we just defined have the following properties:
        \begin{itemize}
            \item $\gamma_i$ is the number of components in the $f$-orbit of $C_i$.
            \item The multiplicity of $\ov C_i$ in the scheme-theoretic fiber $X^{tw}_0$ of $X^{tw} \to D/\mu_n$ is $m_i = \frac{n}{d_i}$. For $p \in C_i$, this number equals $h_pn_{i,p}$ or $2h_pn_{i,p}$ depending on whether $f$ does not swap the branches at $p$ or if it does, respectively.
            \item $\frac{m_i}{\gamma_i}$ is the degree of the map $C_i \to \ov C_i$.
            \item $d_i$ is the ramification index of $\pi$ along $C_i$.
            \item $h_p$ is the number of $f$-orbits of $p$.
            \item $n_p$ is the order of the stabilizer subgroup $\langle f^{h_p} \rangle$ acting on $p$.
            \item $n_{i,p}$ is the local degree, or ramification index, of $C_i^\nu \to \ov C_i$ at the normalizations of $p$.
            \item If $C_i$ and $C_{i'}$ intersect at $p$, the local intersection number of both of these components at $p$ is $\frac{1}{k_p}$.
        \end{itemize}
        Also note that whenever $p$ is a node lying in two components $C_i$ and $C_{i'}$, then $d_in_{p,i} = d_{i'}n_{p,i'}$ is independent of the component, as these are equal to either $n_p$ or $\frac{n_p}{2}$ depending on if $f$ swaps the branches at $p$. Also, all these numbers are invariants under the $f$-action. For this reason, we may for example, say that $h_p$ corresponds to a point $p \in \ov C$ instead of $C$. If $C$ is smooth and connected and $\mu_n$ acts faithfully (for example, if the associated singular fiber $\ti X_0$ is a building block), then $d = 1$ and $n_p = n_{i,p}$, so there is no ambiguity.
    \end{remark}

    \begin{remark}
        If $C^\nu \to C$ is the normalization and $X^{\nu} \to D$ is the associated smooth fibration, then for every component $C^{\nu}_i$ of $C^{\nu}$, and every point $p \in C^\nu$, we can define the numbers $\gamma_i^\nu$, $m_i^\nu$, etc. in the same way. They mostly agree with the corresponding numbers associated their images in $C$. The only time the numbers disagree is when $p^{\nu} \in C^\nu$ maps to a node $p \in C$ and some power of $f$ permutes the branches that pass through $p$. In this case the only difference is that $h_p^\nu = 2h_p$, $n_p^\nu = \frac{1}{2}n_p$. Note that by definition $n_{i,p}^\nu = n_{i,p}$ in every case.
    \end{remark}

    \begin{remark}
        As all singularities of $X$, $X^{tw}$ and $X^{\nu,tw}$ are quotient singularities, all curves in these surfaces are $\Q$-Cartier, so it makes sense for us to intersect them. In particular, the projection formula
        \[\tau_*A \cdot B = A \cdot \tau^*B\]
        holds, where $A \subseteq X$ and $B \subseteq X^{tw}$ are any proper curves.
    \end{remark}

    Fix representatives $C_1,\ldots,C_s$ among the $f$-orbits of components of $C$, which induces representatives $C_1^\nu,\ldots,C_s^{\nu}$. We first study what happens in the normalization $C^\nu \subseteq X^\nu \xrightarrow{\rho^\nu} X^{\nu,tw}$. As the $C_i^\nu$ are disjoint, so are the $\ov C_i^\nu$, and thus, $(\ov C_i^\nu)^2 = 0$ for all $i$. Let $p \in C_i^\nu$. If $n_p = 1$ then $\rho$ is étale at $p$ and so the quotient point $\ov p \in X^{\nu,tw}$ is smooth. Otherwise, $f^{h_p}$ acts on $T_{C^\nu_i,p}$ as multiplication by $\p{\zeta_{n_{i,p}}}^{a_{i,p}}$ for some $a_{i,p}$. If $0 < q_{i,p} < n_{i,p}$ is the inverse of $a_{i,p}$ modulo $n_{i,p}$, then over $\ov p \in \ov C_i^\nu$ we find a chain corresponding to the continued fraction $\frac{n_{i,p}}{q_{i,p}}$. By \Cref{self-int-change}, we must have
    \[\p{\ti C_i^\nu}^2 = -\sum_{\substack{\ov p \in \ov C_i^\nu,\\n_{i,p} > 1}} \frac{q_{i,p}}{n_{i,p}}.\]
    Over $p$, the exceptional chain consists of curves with self intersections given by the continued fraction
    \[\frac{n_{i,p}}{q_{i,p}} = [e_1,\ldots,e_\ell].\]
    If $\alpha_i$ are defined by $\alpha_0 = n_{i,p}$, $\alpha_1 = q_{i,p}$ and $\alpha_{j+1} = \alpha_j e_j - \alpha_{j-1}$, then following \Cref{chain-multiplicities} and since the multiplicity of $\ov C_i^\nu$ in the fiber is $m_i$, then the $j$-th component in the exceptional chain over $p$ of $\ti X^\nu \to X^{\nu,tw}$ has multiplicity
    \[m_i \cdot \frac{\alpha_j}{n_{i,p}} = h_p^\nu\alpha_j.\]

    Now let us compare this computation to the one we would obtain in the original singular fibration $X \to D$. Consider a component $C_i \subseteq C$, and let $C_{i,0},\ldots,C_{i,\gamma_i-1}$ be the components in the $f$-orbit of $C_i$, where $C_{i,0} = C_i$.
    \begin{definition} \label{point-types}
        In the above scenario, we categorize the non-free and non-smooth orbits of $C$ as follows.
        \begin{itemize}
            \item Let $p^1_1,\ldots,p^1_{r_1}$ be (a full set of) representatives for $f$-orbits of smooth points of $C$ in $C_i$ with $n_{i,p^1_j} > 1$. We say the are \textit{type 1}.
            \item Let $p^{2}_1,\ldots,p^{2}_{r_{2}}$ be representatives for $f$-orbits of nodes of $C$ in $C_i$, where the other component through $p^{2}_j$ is not in the $f$-orbit of $C_i$. We say they are \textit{type 2}.
            \item Let $p^{3}_1,\ldots,p^{3}_{r_{3}}$ be representatives for $f$-orbits of nodes of $C$ in $C_i$, where both branches through $p^3_j$ belong to components in the $f$-orbit of $C_i$, and no power of $f$ swaps the branches. We say they are \textit{type 3}.
            \item Let $p^{4}_1,\ldots,p^{4}_{r_{4}}$ be representatives for $f$-orbits of nodes of $C$ in $C_i$, where both branches through $p^4_j$ belong to components in the $f$-orbit of $C_i$, and the branches are swapped by a power of $f$. We say they are \textit{type 4}.
        \end{itemize}
    \end{definition}
        
    With this, since $C_i\cdot C = 0$ in $X$ we obtain
    \[C_i \cdot \sum_{j=0}^{\gamma_i-1}C_{i,j} = -C_i \cdot \sum_{\substack{C' \subseteq C,\\C' \neq C_{i,j}}} C' = -\sum_{j=1}^{r_2} \frac{h_{p^2_j}}{\gamma_{i}k_{p^2_{j}}}.\]
    As $\Q$-divisors we have
    \[\tau^* \ov C_i = \sum_{j=0}^{\gamma_i-1}d_iC_{i,j}.\]
    By projection formula, we have
    \[n\ov C_i^2 = (\tau^*\ov C_i)^2 = \sum_{j=0}^{\gamma_i-1} d_iC_{i,j} \sum_{j=0}^{\gamma_i-1}d_iC_{i,j} = \gamma_id_i^2C_i \cdot \sum_{j=0}^{\gamma_i-1}C_{i,j} = -\sum_{j=1}^{r_2} \frac{h_{p^2_j}\gamma_id_i^2}{\gamma_i k_{p^2_j}},\]
    where the third equality holds since $D = \sum C_{i,j}$ is invariant under the $\mu_n$-action, so $C_{i,j}\cdot D = C_i \cdot D$. We can conclude that
    \[\ov C_i^2 = -\sum_{j=1}^{r_2} \frac{d_i}{n_{i,p^2_j}k_{p^2_j}}.\]
    
    Without loss of generality, let us focus on the curve $C_1$ and a point $p \in C_1$.
    
    If $p = p^1_j$ is a point of the first kind, then above its quotient there will be a chain corresponding to the fraction $\frac{n_{1,p}}{q_{1,p}}$ identical to the one appearing in the corresponding building block $\ti X^\nu_{0,1}$. This contributes to a change in the self intersection of the strict transform $\ti C_1$ by $\frac{q_{1,p}}{n_{1,p}}$, which again agrees with the contribution to $\ti C_1^\nu$.

    Now, suppose that $p = p_j^2$ is a point of the second kind. Suppose, without loss of generality that $C_{2}$ is the other component through $p$. Then $f^{h_p}$ acts on $T_{C_1,p}$ as $(\zeta_{n_{1,p}})^{a_{1,p}}$ and on $T_{C_{2}}$ as $(\zeta_{n_{2,p}})^{a_{2,p}}$ for some integers $0 < a_{1,p} \leq n_{1,p}$, and $0 < a_{2,p} \leq n_{2,p}$. Then, by \Cref{non-swap-quotient-theorem}, the image of $p$ is a singularity of type
    \[\frac{1}{n_p}(d_1a_{1,p}, d_2a_{2,p},1;j_p) = \frac{1}{\frac{n_pk_p}{d_1d_2}}\p{1,\frac{q_1k_p - d_1}{d_2}} = \frac{1}{\frac{n_pk_p}{d_1d_2}}\p{\frac{q_2k_p - d_2}{d_1},1},\]
    where $0 < q_{1,p} \leq n_{1,p}$ and $0 < q_{2,p} \leq n_{2,p}$ are inverses of $a_{1,p}$ and $a_{2,p}$ modulo $n_{1,p}$ and $n_{2,p}$ respectively, and $k_p = d_1a_{1,p} + d_2a_{2,p} + j_pn_p$. Note that when $n_{1,p} = 1$, we choose $q_{1,p}$ to be also equal to 1, and similarly with $q_{2,p}$.

    Let $\delta_{1,p}$ be the integer such that
    \[0 \leq \frac{q_{1,p}k_p-d_1}{d_2} + \delta_{1,p}\frac{n_pk_p}{d_1d_2} < \frac{n_pk_p}{d_1d_2},\]
    or in other words, $\delta_{1,p}$ can be defined by
    \[\delta_{1,p} = \ceil{-\frac{(q_{1,p}k_p-d_1)/d_2}{(n_pk_p)/(d_1d_2)}} = \ceil{\frac{1}{n_{1,p}^2\p{\frac{a_{1,p}}{n_{1,p}} + \frac{a_{2,p}}{n_{2,p}} + j_p}}- \frac{q_{1,p}}{n_{1,p}}}.\]
    Note that if $j_p \geq 0$, then $\delta_{1,p} = 0$. Otherwise, by \Cref{deficit-proposition}, when $j_p=-1$, $\delta_{1,p} = \delta(\frac{n_{1,p}}{a_{1,p}},\frac{n_{2,p}}{a_{2,p}})$. Either way can write the same singularity in standard form
    \[\frac{1}{n_p}(d_1a_{1,p}, d_2a_{2,p},1;j_p) = \frac{1}{\frac{n_pk_p}{d_1d_2}}\p{1,\frac{q_{1,p}k_p - d_1}{d_2} + \delta_{1,p}\frac{n_pk_p}{d_1d_2}}.\]
    In case the point happens to be smooth, that is, when $n_pk_p = d_1d_2$, then $\delta_{1,p}$ has the property that the expression evaluates to $\frac{1}{1}(1,0)$ whose associated fraction $\frac{1}{0} = [~]$ is empty, thus still corresponds to its minimal resolution, which is an isomorphism at $\ov p$. If $Z_p \to X^{tw}$ is the minimal resolution at $\ov p$, and $\ti C_{1,p}$ is the strict transform of $\ov C_1$, then
    \[\ti C_{1,p}^2 = \ov C_1^2 - \frac{1}{\frac{n_pk_p}{d_1d_2}}\cdot\p{\frac{q_{1,p}k_p - d_1}{d_2} + \delta_{1,p}\frac{n_pk_p}{d_1d_2}} = \ov C_1^2 + \frac{d_1}{n_{1,p}k_p} - \frac{q_{1,p}}{n_{1,p}} - \delta_{1,p}.\]
    After resolving all other singularities in $\ov C$, there will be a contribution to the self intersection of $\ti C_1$ not coming from $p$, and a contribution from $p$ that we just computed. Something similar happened in the case of $\ti C_1^\nu$, so we can separate the local with non-local contributions, that is, coming from $p$ or points other than $p$, and compare. We obtain
    \begin{align*}
        \ti C_1^2 &= \ov C_1^2 + \p{\text{non-local contribution}} + \frac{d_1}{n_{1,p}k_p} - \frac{q_{1,p}}{n_{1,p}} - \delta_{1,p}\\
        &= (\ti C_1^\nu)^2 + \p{\text{non-local contribution}} - \delta_{1,p} - \e_{1,p},
    \end{align*}
    where we define
    \[\e_{1,p} = \begin{cases}
        0 & \text{ if } n_{1,p} \neq 1\\
        1 & \text{ if } n_{1,p} = 1
    \end{cases},\]
    as in the second case, the corresponding point in $\ov C_1^\nu$ is smooth so no contribution appears in the resolution $\ti C_1^\nu$, but $\frac{q_{1,p}}{n_{1,p}} = 1$ still appears in this contribution for $\ti C_1$.
    
    Now assume that $p = p^3_j$ is a point of the third kind, then there are two points $p^\nu$ and ${p^{\nu}}'$ over $p$ in $C^\nu_1$ and $C^\nu_{1,j}$ that correspond to two points $\ov p^\nu$ and ${\ov p^\nu}'$ in $\ov C_1$. These two points must have the same invariant $n_{1,p}$, but could have different coefficient $q_{1,p}$ and $q_{1,p}'$. Either way they determine two chains in $\ti C^\nu$ corresponding to the fractions $\frac{n_{1,p}}{q_{1,p}}$ and $\frac{n_{1,p}}{q_{1,p}'}$. We also have $n_p = d_1n_{1,p}$ and $k_p = (a_{1,p} + a_{1,p}')d_1 + j_{1,p}d_1n_{1,p}$. Then by \Cref{non-swap-quotient-theorem}, the image of $p$ in $\ov C_1$ is a singularity of type
    \begin{equation}\label{singularity-equality}
        \frac{1}{\frac{n_pk_p}{d_1^2}}\p{1,\frac{q_{1,p}k_p - d_1}{d_1}} = \frac{1}{\frac{n_pk_p}{d_1^2}}\p{\frac{q_{1,p}'k_p-d_1}{d_1},1}.
    \end{equation}
    This singularity is already written in standard form by \Cref{no-delta}. First, assume that $\frac{n_pk_p}{d_1^2} \neq 1$, so this point is singular. From \Cref{singularity-equality} we can see that the inverse of $\frac{q_{1,p}k_p-d_1}{d_1}$ modulo $\frac{n_pk_p}{d_1^2}$ is $\frac{q_{1,p}'k_p-d_1}{d_1}$ and thus, if $Z_p \to X^{tw}$ is the resolution of $p$, by \Cref{self-int-change-nodal}, we have
    \[\ti C_{1,p}^2 = \ov C_1^2 - \frac{\frac{q_{1,p}k_p - d_1}{d_1} + \frac{q_{1,p}'k_p - d_1}{d_1} + 2}{\frac{n_pk_p}{d_1^2}} = \ov C_1^2 - \frac{q_{1,p}}{n_{1,p}} - \frac{q_{1,p}'}{n_{1,p}}.\]
    The point $\ov p$ happens to be smooth in $X$ if and only if $n_{1,p} = a_{1,p} = q_{1,p} = 1$ and $j_{1,p} = -1$, and in this case, we have $\ti C_{1,p}^2 = \ov C_1^2$. Thus, comparing to the building block $\ti X^\nu_{0,1}$ and separating local and non-local contributions, we have
    \[\ti C_1^2 = (\ti C_1^\nu)^2 + (\text{non-local contribution}) - 2\e_{1,p},\]
    where
    \[\e_{1,p} = \begin{cases}
        0 & \text{ if } n_{1,p} \neq 1\\
        0 & \text{ if } n_{1,p} = 1 \text{ and } j = -1\\
        1 & \text{ if } n_{1,p} = 1 \text{ and } j \geq 0\\
    \end{cases}.\]

    Now assume that $p = p^4_j$ is a point of the fourth kind. Recall that in this case $h_p = \frac{h_p^\nu}{2}$ and $n_p = 2n_p^\nu$. Similarly as before, we define $a_{1,p}$ and $q_{1,p}$ to be the induced action of $f^{2h_p}$ on $T_{p^\nu} C^{\nu}$, which is independent of the choice of point over $p$. Assuming that $n_{1,p} > 1$ so the corresponding points are singular, on $\ti C_1^\nu$, there are two isomorphic chains corresponding to the fraction
    \[\frac{n_{1,p}}{q_{1,p}} = [e_1,\ldots,e_\ell],\]
    with multiplicities given by
    \[(2h_p\alpha_1,\ldots,2h_p\alpha_\ell),\]
    as $h_p^\nu = 2h_p$. Recall that $\alpha_1 = q_{i,p}$, as well that $a_\ell = 1$ for any such chain. On $C$, $f^{h_p}$ acts on $T_{p}C$ as multiplication $\smallpmatrix{0 & \zeta_{n_p}^u\\\zeta_{n_p}^v & 0}$ on a basis induced by the branches for some $u,v$, which we choose so that $0 < u + v \leq n_p$. As $n_p = 2n_p^\nu$, these two imply that $2a_{1,p}d_1 = u + v$. By \Cref{swap-quotient-theorem}, the point $\ov p \in \ov C \subseteq X^{tw}$ is a singularity of type
    \[\frac{1}{2n_{1,p}d_1}\p{2a_{1,p}d_1,1;j}_D.\]
    \Cref{swap-local-glueing-algorithm} implies that if $Z_p \to X^{tw}$ is the minimal resolution at the image of $p$ and $\ti C_{1,p}$ is the strict transform of $\ov C_1$, then
    \[\ti C_{1,p}^2 = \ov C_1^2 - \frac{q_{1,p}}{n_{1,p}}.\]
    So, comparing with the building block $\ti X_{0,1}^\nu$, we obtain
    \[\ti C_1^2 = (C_1^\nu)^2 + (\text{non-local contribution}) - \e_{1,p}\]
    where
    \[\e_{1,p} = \begin{cases}
        0 & \text{ if } n_{1,p} \neq 1\\
        1 & \text{ if } n_{1,p} = 1.
    \end{cases}.\]

    We condense all this analysis in the following result

    \subsection{Global Glueing Algorithm} \label{global-glueing-algorithm}
    Let $X \to D$ be a generically smooth fibration over the disc on nodal curves of genus $g$. Suppose $\mu_n$ acts on $X$ and $D$ equivariantly, such that $\zeta_n \cdot t = \zeta_nt$ on $D$ and acts as $f \in \Aut(C)$ on the central fiber $C \subseteq X$ over $0$.

    Let $C^\nu$ be the normalization of $C$ and $X^\nu = C^\nu \times D$ with the action $\zeta_n \cdot (x,t) = (f(x),\zeta_n t)$. Let $\ti X \to X^{tw} = X/\mu_n$ and $\ti X^\nu \to X^{\nu,tw} = X^\nu/\mu_n$ be the minimal resolutions. Define $\ov C$ and $\ov C^\nu$ to be the quotients of $C$ and $C^\nu$ by $f$, and let $\ti X_0, \ti X_0^\nu$ the central fibers in the resolutions $\ti X$ and $\ti X^\nu$.

    Write $\ov C = \bigcup_{i \in I} \ov C_i$ and $\ov C^\nu = \bigcup_{i \in I} \ov C_i^\nu$ as unions of irreducible components, and let $\ti C_i$ and $\ti C_i^\nu$ their strict transforms of $\ov C_i$ and $\ov C_i^\nu$ on $\ti X$ and $\ti X^\nu$. We say a point $p \in \ov C$ is type 1 if it the image of a smooth point in $C$, is of type 2 if it is a node connecting two different components, type 3 if it is a node in a single component, and type 4 if it is not a node in $\ov C$, but is the image of a node in $C$.
    
    Then the combinatorial data of $\ti X_0$, that is, components, intersections and multiplicities, can be obtained from $\ti X_0^\nu$ by applying \Cref{non-swap-local-glueing-algorithm} on chains corresponding to points of type 2 and 3, applying \Cref{swap-local-glueing-algorithm} on chains corresponding to points of type 4, and modifying the self intersection of $\ti C_i^\nu$ with the formula
    \[\ti C_i^2 = \p{\ti C_i^\nu}^2 - \sum \p{\delta_{i,p^2_j} + \e_{i,p^2_j}} - \sum 2\e_{i,p^3_j} - \sum \e_{i,p^4_j},\]
    where the sum is over type 2,3, and 4 points in $\ov C_i$. Here
    \begin{itemize}
        \item $\delta_{i,p_j^2}$ is obtained from \Cref{non-swap-local-glueing-algorithm} while computing the chain over $p_j^2$. In particular it is always $0$ if $j \geq 0$.
        \item $\e_{i,p_j^2}$ (and $\e_{i,p_j^4}$) equals $1$ if the corresponding point in $\ti X_0^\nu$ is smooth (that is, if it is the image of a free $f$-orbit), and otherwise it is zero.
        \item $\e_{i,p_j^3}$ equals $1$ if $j_{i,p_j^3} \geq 0$ and the corresponding point in $\ti X_0^\nu$ is smooth, and equals $0$ otherwise.
    \end{itemize}
    Notice that in the way \Cref{non-swap-local-glueing-algorithm} is set up, the multiplicities of the components in the glued chains over $p$ are calculated using $\tau^*(n_{1,p}L_1 + n_{2,p}L_2)$, but the correct multiplicities are instead the ones appearing in
    \[\tau^*(m_1L_1 + m_2L_2) = h_p\tau^*(n_{1,p}L_1 + n_{2,p}L_2),\]
    so the multiplicity of each chain is obtained by multiplying by $h_p$ the multiplicities that the theorem gives. Similarly, we need to multiply the multiplicities from \Cref{swap-local-glueing-algorithm} by $h_p = \frac{h_p^\nu}{2}$ as well to obtain the correct result.

    The multiplicity of $\ti C_i$ and the chains over smooth points of $\ti C$ are preserved from $\ti C$.

    \begin{remark}\label{glueing-is-not-canonical-scheme-structure}
        Although the building blocks have by definition a scheme structure, we make no claims about how the non-reduced structure is modified in the process of glueing. We already mentioned how setting $X^\nu = C^\nu \times D$, meaning that the non-reduced structure on $\ti X^\nu_0$ is a quotient of a trivial deformation of $C^\nu$, is already not canonical. We will now show how uniqueness of a glueing fails by an example where multiple non-reduced structures appear by glueing together the same building blocks.
    \end{remark}

    \begin{example}
        Consider the curve $C$ which is the disjoint union of two rational curves with a single node, and $C'$ to be the union of two $\P^1$'s intersecting transversally at two points. Let $f \colon C \to C$ the action permuting both curves, and $g \colon C' \to C'$ an action that permutes both components, while also permuting the two intersection points. Then there is an isomorphism $C^\nu \cong (C')^\nu$ under which $f^\nu$ and $g^\nu$ agree. Moreover, $\ti C^\nu = (\ti C')^\nu$ consists of a $\P^1$ with two special points, which are glued in the way above to obtain $\ti C$ and $\ti C'$. Here $\ti C_{red} \cong \ti C'_{red}$ as they are both rational curves with a single node, but their non-reduced structure differ. As a general fiber of a fibration $X \to D$ in which $C$ lives is disconnected, we have $H^0(\ti C,\O_{\ti C}) = K^2$, but for the same reason, $H^0(\ti C', \O_{\ti C'}) = K$. This computation will be a consequence of \Cref{geometrically-connected-theorem}.
    \end{example}

    For the following examples, recall that the \textit {signature} of an action $f$ on a smooth connected curve $C$ is a tuple $(n_1,\ldots,n_k)$, where $k$ is the number of non-free orbits for the action, and $n_i$ is the size of the stabilizer group at the $i$-th non-trivial orbit.
    \begin{example}[Kodaira's classification, smooth case]
        Let $E_0$ be an elliptic curve with $j$ invariant $0$. If $p \in E_0$ is the base point, the automorphism group of $E_0$ fixing $p$ is $\Z/6\Z$ and one of its generators $f$ has signature $(2,3,6)$, where $f^3$, $f^2$ and $f$ act on the tangent spaces as $\zeta_2$, $\zeta_3$ and $\zeta_6$ respectively. Since $E_0/f \cong \P^1$, then the building block $\ti X_0$ associated to this pair $(E_0,f)$ consists of a rational curve of multiplicity $6$, attached to three chains appearing as the resolutions of the singularities $\frac{1}{2}(1,1)$, $\frac{1}{3}(1,1)$ and $\frac{1}{6}(1,1)$. This is the minimal normal crossings resolution of the Kodaira fiber $II$. Similarly, the pairs $(E_0,f^2)$, $(E_0,f^3)$, $(E_0,f^4)$ and $(E_0,f^5)$ are associated to minimal normal crossings resolutions of fibers $IV$, $I_0^*$, $IV^*$ and $II^*$ respectively.

        Now let $E_1$ be a genus 1 curve with $j$ invariant $1728$. One of the two generators $f$ of its (pointed) automorphism group $\Z/4\Z$ has signature $(2,4,4)$, and acts on the tangent spaces as multiplication by $\zeta_2$, $\zeta_4$ and $\zeta_4$ respectively. For the pairs $(E_1,f)$, $(E_1,f^2)$ and $(E_1,f^3)$, the associated singular fibers are the minimal normal crossings resolutions of the fibers $III$, $I_0^*$ and $III^*$.
        
        Finally, let $E$ be any elliptic curve and $e \in E$ be a point of order $n$. Consider the automorphism $f \in \Aut(E)$ defined by $f(x) = x + e$. Then $f$ has no fixed points and $E/f$ is also an elliptic curve. The associated singular fiber is of Kodaira type $\li_{n}I_0$.
    \end{example}

    \begin{example}[Kodaira's classification, nodal case 1]
        Let $E_\infty$ be a rational curve with a single node, and suppose that it is the central fiber of a stable fibration $X \to D$, where there is an $A_{k-1}$ singularity at the node. The full automorphism group of $E_\infty$ is $\G_m \rtimes \Z/2\Z$. Here $E_\infty^\nu \cong \P^1$, and $E_\infty$ is obtained from it by glueing $0$ with $\infty$.
        
        Any automorphism $\sigma \not\in \G_m$ is an involution that fixes two smooth points, and swaps the branches of the node. If $X \to D$ admits an equivariant action which acts as $\sigma$ on the central fiber, then $E_\infty/\sigma \cong \P^1$, but on $\ov X$ there will be the singularities $\frac{1}{2}(1,1)$, $\frac{1}{2}(1,1)$ and $\frac{1}{2}(2,1;j)_D$, where $k = 2j+2$. The central fiber in the resolution is thus the Kodaira fiber $I_{j+1}^*$.

        Let's see how to recover it with the glueing algorithm. The associated building block for this fiber is the one obtained from the involution $z \mapsto -z$ of $\P^1 \cong E_\infty^\nu$. So it is consists of a central $\P^1$ of multiplicity $2$ and two $(-2)$-curves intersecting it.
        \singularfiber[$\ti X_0^\nu$]{$(-1)$}{1,1}{$(-2)$}{0}{}{1,1}{$(-2)$}
        We want to apply \Cref{swap-local-glueing-algorithm} over an empty chain in $\ti X_0^\nu$, since $f^\nu$ acts freely on the normalization of the node. Therefore we need to first blow it up:
        \singularfiber[$\ti X_0^{\nu+}$]{$(-2)$}{1,1}{$(-2)$}{1}{$(-1)$}{1,1}{$(-2)$}
        Then, according to \Cref{swap-local-glueing-algorithm}, we must replace this $(-1)$ curve with a D-tail consisting of a chain of $j+1$ $(-2)$-curves and two $(-2)$-curves at the end. This is again a Kodaira $I_{j+1}^*$ fiber.

        Let $f = \zeta_n^a \in \G_m$ be an automorphism of order $n$, where $\gcd(a,n) = 1$. We think of it as $z \mapsto \zeta_n^a z$ on its normalization $\P^1$. It follows that $f$ acts as $\zeta_n^a$ and $\zeta_n^{n-a}$ on the tangent spaces at $0$ and $\infty$, so $k = n + jn$ for $j \geq 0$. The associated building block consists of a central fiber isomorphic to $\P^1$ and multiplicity $n$, and two chains, dual to each other, appearing as the resolutions of $\frac{1}{n}(1,q)$ and $\frac{1}{n}(1,n-q)$, where $q$ is the inverse of $a$ modulo $n$.
        \singularfiber{$(-1)$}{1,3}{$\frac{1}{n}(1,q)$}{0}{}{1,3}{$\frac{1}{n}(1,n-q)$}
        According to \Cref{non-swap-local-glueing-algorithm}, we must subtract 1 to the self intersections of the last components of the chains and connect them together with another chain of $j-1$ $(-2)$-curves of multiplicity $1$. This may not look like a Kodaira fiber, but the central curve is a $(-1)$ that we can contract. Since the singularities were duals to begin with, by \Cref{contraction-of-duals} we can keep contracting until the very last curves of the chains become connected, each having self intersection $-2$. This is a Kodaira fiber $I_{j+1}$. 
    \end{example}

    \begin{example}[Kodaira's classification, nodal case 2]
        Now consider $C$ to be a cycle of $m$ $\P^1$'s (that is, an $I_m$), and let $f \in \Aut(C)$ be the automorphism of order $m$ sending the $i$-th component to the $(i+1)$-th one. For now, let us assume that such an automorphism can be extended to an equivariant action on a generically smooth fibration $X \to D$. Here, $C^\nu$ is a disjoint union of $m$ $\P^1$'s where $f^\nu$ acts transitively on components. The associated building block is a single $\P^1$ of multiplicity $m$, with two marked points that we would need to glue. We obtain the Kodaira fiber $\li_mI_{j+1}$.
    \end{example}

    \begin{example}[Cusps]
        Let $H$ be a smooth (hyperelliptic) curve of genus $g$. It is possible to choose it such that it admits an automorphism $f$ of order $4g + 2$ having signature $(2,2g+1,4g+2)$, where $f$ acts as multiplication by $\zeta_2$, $\zeta_{2g+1}^{2g-1}$ and $\zeta_{4g+2}$ on the respective tangent spaces. The quotient $H/f$ is isomorphic to $\P^1$, and the associated building block is
        \singularfiber{$(-1)$}{1,1}{$(-4g-2)$}{1}{$(-2)$}{1,3}{$\frac{1}{2g+1}(1,g)$}
        Here, we use that $g$ is the inverse of $2g - 1$ modulo $2g + 1$. The lower chain is easy to understand. Since
        \[\frac{2g + 1}{g} = 3 - \frac{1}{~\frac{g}{g-1}~},\]
        we deduce that
        \[\frac{2g + 1}{g} = [3,\underbrace{2,\ldots,2}_{g-1}]\]
        We can contract the central $(-1)$-curve, then contract the image of the $(-2)$ curve. At this point, there is a $(-4g)$-curve intersecting the first curve of a chain $[1,2,\ldots,2]$ at a tacnode. Contracting this chain completely, we obtain a singular curve
        \begin{center}\begin{tikzpicture}
            \draw[smooth,domain=-1.25:1.25,variable=\x] plot ({0.5*\x*\x},{0.25*\x*\x*\x});
            \node[right] at (0.75,0.4) {$(0)$};
            \node[left] at (-0.5,0) {$A_{2g}$};
            \draw[->] (-0.5,0) -- (-0.1,0);
        \end{tikzpicture}\end{center}
        where the curve singularity (which is a smooth point in the ambient surface) is the cusp $y^2 = x^{2g+1}$.
    \end{example}

    \begin{example}[Multiple cusps]
        Let $g_1,g_2 \geq 1$, and let $H_1$, $H_2$ be the hyperelliptic curves of genus $g_1$ and $g_2$ respectively with actions $f_1, f_2$ from the example above. Let $C$ be the glueing of $H_1$ and $H_2$ along their unique fixed point, the ones with signatures $4g_i + 2$, and let $f$ be the induced automorphism of $C$. Again, let us assume that $(C,f)$ can be extended to an equivariant action on a generically smooth surface $X \to D$, where the node of $C$ is a singularity of type $A_{k-1}$. Since the actions on the tangent spaces of the fixed points are $\zeta_{4g_i+2}$, then
        \[k = d\p{4g_2 + 2 + 4g_1 + 2 + (4g_1 + 2)(4g_2 + 2)j}, \qquad \text{ for some } d \in \frac{1}{\gcd(4g_1 + 2,4g_2 + 2)}\Z.\]
        Here $d = \gcd(4g_1 + 2,4g_2 + 2)$ if and only if the equivariant action on $X \to D$ acts faithfully on the central fiber $C$. However, the philosophy behind \Cref{non-swap-local-glueing-algorithm} tells us that it does not matter to us what exactly $k$ is, but it does help us see that $j$ cannot be $-1$ as in that case $k$ would be negative. This fact can also be deduced by \Cref{non-swap-local-glueing-algorithm} (4), since if we attempt to contract the chain
        \[[4g_1 + 2, 1, 4g_2 + 2] \qquad \rightsquigarrow \qquad [4g_1 + 1, 4g_2 + 1],\]
        we do not arrive to a configuration containing $[1,1]$.

        So, if $j = 0$, we obtain
        
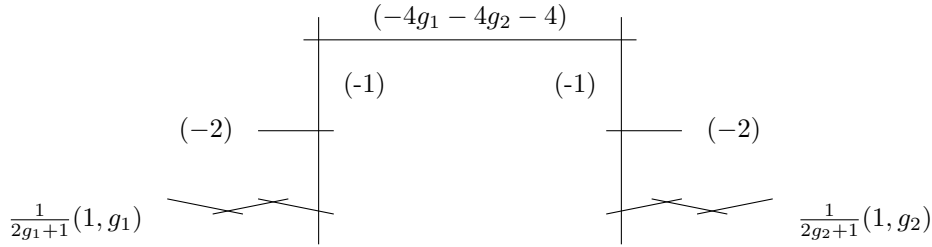
\begin{center}\begin{tikzpicture}
            \draw (2,3) -- (2,0);
            \node[left] at (1.8,2.1) {(-1)};
            
            \draw(-2.2,2.7) -- (2.2,2.7);

            \node[above] at (0,2.7) {$(-4g_1-4g_2-4)$};

            \draw (1.8,1.5) -- (2.8,1.5);
            \node[right] at (3,1.5) {$(-2)$};

            \draw (1.8,0.4) -- (2.8,0.6);
            \draw (2.4,0.6) -- (3.4,0.4);
            \draw (3,0.4) -- (4,0.6);
            \node[right] at (4.2,0.3) {$\frac{1}{2g_2 + 1}(1,g_2)$};

            \draw (-2,3) -- (-2,0);
            \node[right] at (-1.8,2.1) {(-1)};

            \draw (-1.8,1.5) -- (-2.8,1.5);
            \node[left] at (-3,1.5) {$(-2)$};

            \draw (-1.8,0.4) -- (-2.8,0.6);
            \draw (-2.4,0.6) -- (-3.4,0.4);
            \draw (-3,0.4) -- (-4,0.6);
            \node[left] at (-4.2,0.3) {$\frac{1}{2g_1 + 1}(1,g_1)$};
        \end{tikzpicture}\end{center}
        After all contractions, we obtain a double cuspidal curve
        
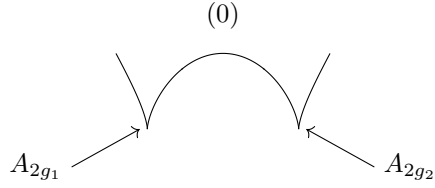
\begin{center}\begin{tikzpicture}
            \draw[smooth,samples=50,domain=0:1,variable=\y] plot ({sqrt(sqrt(\y*\y*\y)+1)},{\y});
            \draw[smooth,samples=50,domain=0:{sqrt(-1/sqrt(8)+1)},variable=\y] plot ({sqrt(-sqrt(\y*\y*\y)+1)},{\y});
            \draw[smooth,samples=50,domain=-0.55:0.55,variable=\x] plot({\x},{pow((\x*\x-1)*(\x*\x-1),1/3)});
            \draw[smooth,samples=50,domain=0:{sqrt(-1/sqrt(8)+1)},variable=\y] plot ({-sqrt(-sqrt(\y*\y*\y)+1)},{\y});
            \draw[smooth,samples=50,domain=0:1,variable=\y] plot ({-sqrt(sqrt(\y*\y*\y)+1)},{\y});
            \node[above] at (0,1.2) {$(0)$};
            \node[left] at (-2,-0.5) {$A_{2g_1}$};
            \draw[->] (-2,-0.5) -- (-1.1,0);
            \node[right] at (2,-0.5) {$A_{2g_2}$};
            \draw[->] (2,-0.5) -- (1.1,0);
        \end{tikzpicture}\end{center}
        If $j > 0$, we have
        
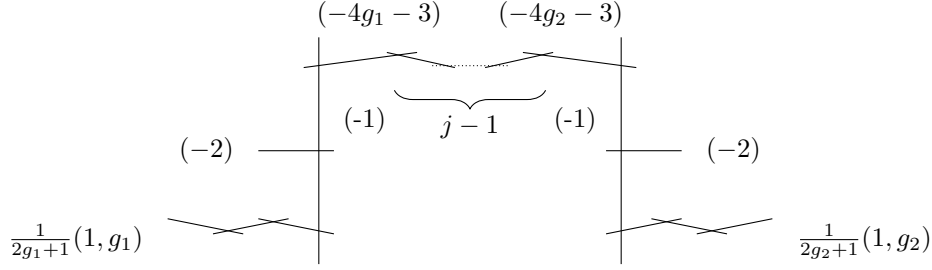
\begin{center}\begin{tikzpicture}
            \draw (2,3) -- (2,0);
            \node[left] at (1.8,1.9) {(-1)};
            
            \draw(0.7,2.8) -- (2.2,2.6);
            \node[above] at (1.2,3) {$(-4g_2-3)$};
            
            \draw(-0.7,2.8) -- (-2.2,2.6);
            \node[above] at (-1.2,3) {$(-4g_1-3)$};

            \draw (-1.1,2.8) -- (-0.2,2.6);
            \draw (1.1,2.8) -- (0.2,2.6);
            \draw[densely dotted] (-0.5,2.625) -- (0.5,2.625);

            \draw[decoration={brace,mirror,amplitude=7},decorate] (-1,2.3) -- node[below=6pt] {$j-1$} (1,2.3);

            \draw (1.8,1.5) -- (2.8,1.5);
            \node[right] at (3,1.5) {$(-2)$};

            \draw (1.8,0.4) -- (2.8,0.6);
            \draw (2.4,0.6) -- (3.4,0.4);
            \draw (3,0.4) -- (4,0.6);
            \node[right] at (4.2,0.3) {$\frac{1}{2g_2 + 1}(1,g_2)$};

            \draw (-2,3) -- (-2,0);
            \node[right] at (-1.8,1.9) {(-1)};

            \draw (-1.8,1.5) -- (-2.8,1.5);
            \node[left] at (-3,1.5) {$(-2)$};

            \draw (-1.8,0.4) -- (-2.8,0.6);
            \draw (-2.4,0.6) -- (-3.4,0.4);
            \draw (-3,0.4) -- (-4,0.6);
            \node[left] at (-4.2,0.3) {$\frac{1}{2g_1 + 1}(1,g_1)$};
        \end{tikzpicture}\end{center}
        Here, there are $j-1$ $(-2)$-curves connecting both exceptionals. Again, after all contractions, we obtain the following configuration
        
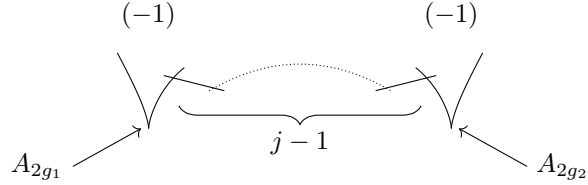
\begin{center}\begin{tikzpicture}
            \draw[smooth,samples=50,domain=0:1,variable=\y] plot ({1+sqrt(sqrt(\y*\y*\y)+1)},{\y});
            \draw[smooth,samples=50,domain=0:{sqrt(-1/sqrt(8)+1)},variable=\y] plot ({1+sqrt(-sqrt(\y*\y*\y)+1)},{\y});
            \draw[smooth,samples=50,domain=0:{sqrt(-1/sqrt(8)+1)},variable=\y] plot ({-1-sqrt(-sqrt(\y*\y*\y)+1)},{\y});
            \draw[smooth,samples=50,domain=0:1,variable=\y] plot ({-1-sqrt(sqrt(\y*\y*\y)+1)},{\y});
            \node[above] at (2,1.2) {$(-1)$};
            \node[above] at (-2,1.2) {$(-1)$};
            \node[left] at (-3,-0.5) {$A_{2g_1}$};
            \draw[->] (-3,-0.5) -- (-2.1,0);
            \node[right] at (3,-0.5) {$A_{2g_2}$};
            \draw[->] (3,-0.5) -- (2.1,0);
            
            \draw (-1.8,0.7) -- (-1,0.5);
            \draw[densely dotted] (-1.2,0.5) to[bend left] (1.2,0.5);
            \draw (1.8,0.7) -- (1,0.5);

            \draw[decoration={brace,mirror,amplitude=7},decorate] (-1.6,0.3) -- node[below=6pt] {$j-1$} (1.6,0.3);
        \end{tikzpicture}\end{center}
    \end{example}

    \section{Existence of Fibrations with Specified Action on Fiber}\label{section:existence-of-fibration-with-action}

    In the previous sections we saw how singular fibers arise from a family of stable curves $X \to D$, together with an equivariant $\mu_n$-action, and we saw that, at least combinatorially, the data of the singular fiber depends only on the central curve $C \subseteq X$, an automorphism $f \colon C \to C$ and the singularity type $A_{k_p-1}$ on each node $p \in C$. Our purpose now is to show the converse: that given the data $(C,f,\{k_p\})$, up to a compatibility condition that we already saw at the start of \Cref{section:cyclic-quotients-of-Ak}, then there exists a generically smooth fibration $X \to D$ with an equivariant $\mu_n$-action which restricts to $f$ on the central fiber.

    Let $G$ be a group acting on a scheme $X$, and let $\Def(X)$ be its first order deformations. Then $G$ acts on $\Def(X)$ in the following natural way: if $\iota \colon X \hookrightarrow X'$ is a deformation of $X$ over $\Spec K[\e]/(\e^2)$ corresponding to a class $[X'] \in \Def(X)$, and $g \in G$ is an automorphism $g \colon X \to X$, then $g \cdot [X'] \in \Def(X)$ corresponds to the isomorphism class of
    \[\iota \circ g^{-1} \colon X \hookrightarrow X'.\]
    More specifically, suppose we have the exact sequence
    \[0 \to \O_X \xrightarrow{\e} \O_{X'} \xrightarrow{\pi} \O_X \to 0.\]
    Pushing this sequence forward along $g$, and noting that $g^\# \colon \O_X \to g_*\O_X$ is an isomorphism, we obtain

    \[\begin{tikzcd}
        0 \ar[r] & \O_X \ar[d] \ar[r,"\e \circ g^{\#}"] & g_*\O_{X'} \ar[d,equal] \ar[r,"g^{\#-1} \circ \pi"] & \O_X \ar[r] \ar[d] & 0\\
        0 \ar[r] & g_*\O_X \ar[r,"\e"] & g_*\O_{X'} \ar[r,"\pi"] & g_*\O_X \ar[r] & 0.
    \end{tikzcd}\]
    The upper sequence corresponds to the extension class of $g \cdot [X'] \in \Def(X)$.

    When $X$ is a local complete intersection, (or alternatively, reduced and generically smooth, see \cite[Theorem 1.1.10]{Sernesi2006}) as is the case with nodal curves, we can naturally identify $\Def(X)$ with $\Ext^1(\Omega_X,\O_X)$.

    \begin{example}\label{nodal-ext}
        Let $X = \Spec K[x,y]/(xy)$. Then $\Def(X) = \Ext^1(\Omega_X,\O_X) \cong K$, where $\lambda \in K$ corresponds to the deformation $X_{\lambda}$ defined by
        \[0 \to K[x,y]/(xy) \to K[x,y,\e]/(xy-\lambda \e, \e^2) \to K[x,y]/(xy) \to 0.\]
        Let $g \colon X \to X$ be defined by $(x,y) \mapsto (ax,by)$ for $a,b \in K^\times$. Then the map $\O_X \to g_*\O_{X_\lambda}$ is given by
        \[p(x,y) \mapsto p(ax,by)\e,\]
        and $g_*\O_{X_\lambda} \to \O_X$ is given by
        \[p(x,y) + q(x,y)\e \mapsto p(a^{-1}x,b^{-1}y).\]
        Therefore, we see that the map $g_*\O_{X_\lambda} \to \O_{X_{ab\lambda}}$ determined by
        \[(x,y,\e) \mapsto (a^{-1}x,b^{-1}y,\e)\]
        is an isomorphism of extensions. So $g$ acts as multiplication by $ab$ in $\Def(X)$.

        Similarly, if $g \colon X \to X$ is defined by $(x,y) \mapsto (ay,bx)$, then $g$ also acts as multiplication by $ab$ in $\Def(X)$.

        Up to choosing an ordering of branches of $X$, $g$ acts on $T_{X,0}$ as either the matrix $\smallpmatrix{a&0\\0&b}$ or $\smallpmatrix{0&b\\a&0}$, so the action of $g$ on $\Def(X)$ is determined by its action on $T_{X,0}$.
    \end{example}

    Now, assume that $C$ is a nodal curve and $G$ is linearly reductive. From the Grothendieck spectral sequence for $\Ext$ we have the exact sequence
    \[0 \to H^1(T_C) \to \Ext^1(\Omega_C,\O_C) \to H^0(\sExt^1(\Omega_C,\O_C))\to 0.\]
    Here $T_C$ is the tangent sheaf on $C$, and $H^1(T_C)$ classifies those deformations of $X$ that do not smooth the nodes of $C$, so we interpret $H^0(\sExt^1(\Omega_C,\O_C))$ as the induced deformations at the nodes. Since $G$ is linearly reductive and acts on each term of this sequence, there exists a non-canonical $G$-equivariant splitting
    \[\Ext^1(\Omega_C,\O_C) \cong H^1(T_C) \oplus H^0(\sExt^1(\Omega_C,\O_C)).\]
    Since $\sExt^1(\Omega_C,\O_C)$ is supported at the nodes of $C$, we further have a natural decomposition
    \begin{equation}\label{ext-decomposition}
        H^0(\sExt^1(\Omega_C,\O_C)) \cong \bigoplus_{i=1}^\ell \Ext^1(\Omega_{C,p_i},\O_{C,p_i})
    \end{equation}
    where $p_1,\ldots,p_\ell$ are the nodes of $C$. For each of these nodes, by passing to the completion we obtain natural isomorphisms
    \[\Ext^1(\Omega_{C,p_i},\O_{C,p_i}) \cong \Def(\widehat C_{p_i})\]
    and this group is, up to choosing branches, isomorphic to $\Def(K[[x,y]]/(xy)) \cong K$ as in \Cref{nodal-ext}. Thus, if $g \in G$ fixes $p_i$, then it acts on the summand $\Ext^1(\Omega_{C,p_i},\O_{C,p_i})$ as multiplication by some constant determined by the action of $g$ on $T_{C,p_i}$. More generally,
    \[g(\Ext^1(\Omega_{C,p_i},\O_{C,p_i})) = \Ext^1(\Omega_{C,g(p_i)},\O_{C,g(p_i)})\]
    as subspaces of $H^0(\sExt^1(\Omega_C,\O_C))$.

    If $D$ is a divisor on $C$ away from the nodes, the deformations of the pair $(C,D)$ give an analogous exact sequence
    \[0 \to H^1(T_C(-D)) \to \Ext^1(\Omega_C(D),\O_C) \to H^0(\sExt^1(\Omega_C,\O_C))\to 0.\]
    We still have a $G$-equivariant decomposition
    \[\Ext^1(\Omega_C(D),\O_C) \cong H^1(T_C(-D)) \oplus H^0(\sExt^1(\Omega_C,\O_C)).\]
    
    \begin{theorem}\label{weak-realization-theorem}
        Let $n$ be an integer coprime to $\chr K$. Let $C$ be a possibly disconnected, nodal curve and let $f \in \Aut(C)$ be an automorphism of finite order dividing $n$. Let $p_1,\ldots,p_\ell$ be the nodes of $C$ and let $k_1,\ldots,k_\ell$ be positive integers. The following are equivalent
        \begin{enumerate}
            \item There exists a family of nodal curves $X \to B$ such that
            \begin{itemize}
                \item $B = \Spec R$ is a smooth finite type curve over $K$ with a $\mu_n$-action, with $0 \in B$ a fixed point around which, formally locally, $\zeta_n \colon \Spec K[[t]] \to \Spec K[[t]]$ looks like $t \mapsto \zeta_n t$.
                \item The (geometric) general fiber is smooth, $X_0 \cong C$ and $\zeta_n$ acts as $f$ on $X_0$.
                \item Each node $p_i \in X_0$ is a singular point of type $A_{k_i-1}$ in $X$.
            \end{itemize}
            \item There exists a $\mu_n$-equivariant map
            \[\psi \colon \A^1 \to H^0(\sExt^1(\Omega_C,\O_C)),\]
            whose projection onto the summand $\Ext^1(\Omega_{C,p_i},\O_{C,p_i}) \cong K$ is of the form $\psi_i(t) = \lambda_i t^{k_i}$ for some $\lambda_i \neq 0$. Here $\mu_n$ acts on $\A^1$ as $t \mapsto \zeta_n t$ and on $H^0(\sExt^1(\Omega_C,\O_C))$ as $f$.
            \item The $k_i$ are invariant along $f$-orbits and for every node $p_i$ and $d \in \Z$, if $f^d(p_i) = p_i$, then $f^d$ acts on $\Ext^1(\Omega_{C,p_i},\O_{C,p_i})$ as multiplication by $\zeta_n^{k_id}$ .
        \end{enumerate}
    \end{theorem}

    \begin{remark}
        Suppose that $p$ is a node of $C$ and $f^{h}$ is the smallest power of $f$ fixing $p$. Formally near $p$, $f$ acts on $\Spec K[[x,y]]/(xy)$ as $(x,y) \mapsto (\zeta_{n/h}^ax,\zeta_{n/h}^by)$ or $(x,y) \mapsto (\zeta_{n/h}^ay,\zeta_{n/h}^bx)$. Then the statement in (2) is equivalent to saying that $k_i \equiv a + b \mod n/h$, which we already saw at the start of \Cref{section:cyclic-quotients-of-Ak} as a necessary condition for $f$ to be extended on a surface containing $C$.
    \end{remark}

    \begin{proof}[Proof of \Cref{weak-realization-theorem}]
        We can easily reduce to the case when $C$ is connected, as we can treat each component independently. Let $\M_{g,m}^D$ be the moduli space of reduced divisorially marked curves, parametrizing pairs $(\mathcal{C},\mathcal{D})$, where $\mathcal{C} \to S$ is a family of nodal curves of genus $g$, and $\mathcal{D} \subseteq \mathcal{C}$ is a flat divisor whose restriction to fibers are reduced of degree $m$, and such that $\omega_{\mathcal{C}/S}(\mathcal{D})$ is relatively ample. This moduli space can also be understood as the quotient $\ov \M_{g,m}/S_m$.

        Let $D \subseteq C$ be a divisor which is stable under the $f$-action and such that $(C,D)$ is a stable divisorially marked curve.

        $(1) \iff (2)$: By the local structure on Deligne-Mumford stacks \cite[Theorem 5.3.1]{alper-stacks}, there is a scheme $\Spec A \ni w$ with an $\Aut(C,D)$-action, and an étale morphism $[\Spec A/\Aut(C,D)] \to \ov \M^D_{g,|D|}$ which induces isomorphisms on stabilizer groups at every point (see \cite[Proof of Theorem 5.4.6]{alper-stacks}). Now, since the order of $f$ is coprime to $\chr K$, $[\Spec A/\langle f\rangle]$ is a tame Deligne-Mumford stack, and after composing with $[\Spec A/\langle f \rangle] \to [\Spec A/\Aut(C,D)]$, the same proof of \cite[Theorem 7.7.1]{alper-stacks} applies to show that there are étale morphisms
        \[\begin{tikzcd}
            \mathcal{W} = ([\Spec A/\langle f \rangle],w) \ar[d] \ar[r] & (\ov\M^D_{g,|D|},[C])\\
            \mathcal{T} = ([T_{\ov\M^D_{g,|D|},[C]}/\langle f \rangle],0), 
        \end{tikzcd}\]
        where the vertical (resp. horizontal) one induces isomorphisms (resp. embeddings) on stabilizer groups at every point of $\mathcal{W}$. Here $T_{\ov\M^D_{g,|D|},[C]}$ is naturally isomorphic to $\Ext^1(\Omega_C(D),\O_C)$, and the action of $\langle f \rangle$ is the natural one described at the beginning of this section.

        The boundary $\partial \M^D_{g,|D|} \subseteq \ov\M^D_{g,|D|}$ is a normal crossing divisor, which at $[C]$ has $\ell$ branches, the $i$-th of them corresponding to directions in moduli which do not smooth the point $p_i$. The map $\mathcal{W} \to \mathcal{T}$ is constructed in a natural way, and as such, the $i$-th branch of the boundary divisor $\partial \ov\M^D_{g,|D|}$ at $[C]$ corresponds to the $i$-th hyperplane given by the zero set of the projection
        \[\Ext^1(\Omega_C(D),\O_C) \to \Ext^1(\Omega_{C,p_i},\O_{C,p_i}).\]

        The existence of a family of stable curves $X \to B$ as in $(1)$ is equivalent to a classifying map $\varphi \colon B \to \ov\M^D_{g,|D|}$ together with a $2$-isomorphism $\alpha \colon \varphi \Rightarrow \varphi \circ \zeta_n$ which restricts to
        \[f \in \Aut_{\ov\M^D_{g,|D|}}(C)\]
        at $0 \in B$.

        Since $\mathcal{W} \to \ov\M^D_{g,|D|}$ (resp. $\mathcal{W} \to \mathcal{T}$) induce embeddings (resp. isomorphisms) on stabilizer groups, then the induced morphisms on inertia
        \[\begin{tikzcd}
            \mathcal{I}_\mathcal{W} \ar[d] \ar[r] & \mathcal{I}_{\ov\M^D_{g,|D|}} \times_{\ov\M^D_{g,|D|}} \mathcal{W}\\
            \mathcal{I}_{\mathcal{T}} \times_{\mathcal{T}} \mathcal{W} 
        \end{tikzcd}\]
        are open and closed embeddings (resp. isomorphisms). We use this fact to conclude that the existence of the data
        \[(\varphi \colon B \to \ov\M^D_{g,|D|}, \alpha \colon \varphi \Rightarrow \varphi \circ \zeta_n), \quad \text{ such that } \alpha|_0 = f\]
        is, up to étale covers $B' \leftarrow B'' \to B$, equivalent to the existence of analogous data
        \[(\psi \colon B' \to \mathcal{T}, \beta \colon \psi \Rightarrow \psi \circ \zeta_n).\]

        A morphism $\psi \colon B' \to [T_{\ov\M^D_{g,|D|},[C]}/\langle f \rangle]$ together with a $2$-isomorphism $\beta \colon \psi \Rightarrow \psi \circ \zeta_n$ is, étale locally at $0 \in B'$, equivalent to a morphism $\psi \colon B' \to \Ext^1(\Omega_C(D),\O_C)$ mapping $0 \in B'$ to $0$, which is $\mu_n$-equivariant. Furthermore, this morphism corresponds to a generically smooth family if and only if, locally around 0, the image of $\psi$ is not supported on one of the coordinate hyperplanes determined by the $i$-th projection $\Ext^1(\Omega_C(D),\O_C) \to \Ext^1(\Omega_{C,p_i},\O_{C,p_i})$. Having a singularity of type $A_{k_i-1}$ at $p_i$ corresponds to $B'$ intersecting said hyperplane with order $k_i$.

        Now, under a $\langle f \rangle$-equivariant decomposition
        \[\Ext^1(\Omega_C(D),\O_C) \cong H^1(T_C(-D)) \oplus H^0(\sExt^1(\Omega_C,\O_C)),\]
        the existence of a morphism $B' \to \Ext^1(\Omega_C(D),\O_C)$ satisfying the above properties is equivalent to the existence of a morphism from $B'$ to $H^0(\sExt^1(\Omega_C,\O_C))$ instead. This shows that $(2)$ implies $(1)$.

        For the other way around, consider the composition
        \[\Spec K[[t]] \to B' \to H^0(\sExt^1(\Omega_C,\O_C)) = \bigoplus_{i=1}^\ell \Ext^1(\Omega_{C,p_i},\O_{C,p_i}).\]
        On the $i$-th summand, this morphism looks like a power series of lowest degree $k_i$. We can truncate each of these power series and leave only the term with degree $k_i$, so we obtain a map $\A^1 \to H^0(\sExt^1(\Omega_C,\O_C))$. Since $\Aut(C)$ acts linearly on $H^0(\sExt^1(\Omega_C,\O_C))$, this new morphism still satisfies the required properties.

        We will next show the equivalence of $(2)$ and $(3)$. We may assume that there is a single $f$-orbit of nodes, since $H^0(\sExt^1(\Omega_C,\O_C))$ decomposes $f$-equivariantly as a sum of $\Ext$ groups over orbits, and we can treat each independently. Under this assumption, $f^\ell$ is the minimal power of $f$ that fixes one (and thus all) $p_i$. Also, note that both $(2)$ and $(3)$ imply that the $k_i$ are all equal, so let us name that common integer $k$.
        
        $(2) \Rightarrow (3)$. As $f^\ell(p_i) = p_i$, then $f^\ell$ must act as multiplication by some $\tau_i \in \G_m$ on $\Ext^1(\Omega_{C,p_i},\O_{C,p_i})$. We need only to verify the condition when $d = \ell$. On one hand, we have that
        \[\psi(\zeta_n^\ell t) = \p{\lambda_1\zeta_n^{k\ell}t^k,\ldots,\lambda_\ell \zeta_n^{k\ell}t^k}.\]
        On the other hand, by the way we defined $\tau_i$ we have
        \[f^\ell(\psi(t)) = \p{\lambda_1\tau_1t^k,\ldots,\lambda_\ell\tau_\ell t^k}.\]
        As $\psi$ is $\mu_n$-equivariant, it must follow that $\tau_i = \zeta_n^{k\ell}$ for all $i$.
        
        $(3) \Rightarrow (2)$. We re-order the nodes $p_i$ while finding a basis of $H^0(\sExt^1(\Omega_C,\O_C))$ in the following way. Fix a node $p_1$ and an arbitrary non-zero element $e_1 \in \Ext^1(\Omega_{C,p_1},\O_{C,p_1})$. We define inductively
        \[p_{i+1} = f(p_i),\]
        and
        \[e_{i+1} = \zeta_n^{-k}f(e_i) \in \Ext^1(\Omega_{C,p_{i+1}},\O_{C,p_{i+1}}) \subseteq H^0(\sExt^1(\Omega_C,\O_C)).\]

        By hypothesis, in this basis $f$ acts as the matrix
        \[ f = \zeta_n^k\begin{bmatrix}
            0 & 0 & \cdots & 0 & 1\\
            1 & 0 & \cdots & 0 & 0\\
            0 & 1 & \cdots & 0 & 0\\
            \vdots & \vdots & \ddots & \vdots & \vdots\\
            0 & 0 & \cdots & 1 & 0 
        \end{bmatrix}.\]
        We define the map
        \begin{align*}
            \psi \colon \A^1 &\to \bigoplus_{i=1}^{\ell} \Ext^1(\Omega_{C,p_i},\O_{C,p_i})\\
            t &\mapsto (t^k,\ldots,t^k).
        \end{align*}
        We have
        \[f(\psi(t)) = (\zeta_n^kt^k,\ldots,\zeta_n^kt^k) = \psi(\zeta_n t),\]
        so $\psi$ is $\mu_n$-equivariant.
    \end{proof}

    \section{The Other Approach to the Classification of Tame Singular Fibers} \label{section:classification-of-all-fibers}

    In this section our aim is to give another strategy for the classification of singular fibers in fibrations of curves. Rather than a proof, we will see how we can classify and glue building blocks in a way analogous to \cite{Dok25}. As an application, together with the combinatorial analysis of \cite{Dok25} we can show another proof of \cite[Corollary 4.3]{Winters1974}.

    In \Cref{section:glueing-building-blocks} we showed how, given a stable curve with an action and an extra parameter for each node, we can construct the corresponding fiber by glueing together the building blocks obtained from the normalization of the original stable curve. \Cref{weak-realization-theorem} can be used to prove the converse. Up to some compatibility conditions that we will see in this section, it tells us that any combinatorial glueing of building blocks is realized as the special fiber of some generically smooth fibration.

    Since the combinatorics of a singular fiber are determined only by how the building blocks are glued together, we can skip the first step of classifying stable curves with an action. Instead we may classify building blocks and describe the rules on how they can be glued together.

    Let $C$ be a smooth, connected curve and $f \in \Aut(C)$ be an automorphism of order $n$. Let $\ti X_0$ be the central fiber of the resolution $\ti X \to C \times D/\mu_n$ and let $\ti C_1 \subseteq \ti C$ be the strict transform of the image of $C$. Then $\ti X_0$ consists of $\ti C_1$ together with a number of chains emanating from the points corresponding to orbits of non-free locus in $C$.

    With the notation of \Cref{section:glueing-building-blocks}, we have seen that
    \begin{equation*}
        \ti C_1^2 = -\sum_p \frac{q_{1,p}}{n_{1,p}} = -\frac{1}{n}\sum_p h_pq_{1,p},
    \end{equation*}
    so in particular,
    \begin{equation}\label{cover-action-restriction}
        \sum_p h_pq_{1,p} \equiv 0 \mod n.
    \end{equation}
    In characteristic zero, or more generally if $f$ has order coprime to the characteristic of the base field, the converse is also true in the following sense.
    \begin{theorem}[Riemann Existence Theorem for Cyclic Galois Covers]\label{existence-cyclic-galois-cover}
        Let $B$ be a smooth projective curve over $K$, let $x_1,\ldots,x_\ell \in X$ and let $n \geq 1$ be an integer coprime to $\chr K$. Suppose that for each $i = 1,\ldots, \ell$ we have chosen a factorization $n = h_in_i$, and integers $a_i$ coprime to $n_i$. Let $q_i$ be the inverse of $a_i$ modulo $n_i$ and assume that
        \[\sum_{i=1}^\ell h_iq_i \equiv 0 \mod n.\]
        Then there exists a cyclic Galois cover $C \to B$ and a generator $f \in \Aut(C \to B)$ such that $f$ is free away from $x_1\ldots,x_\ell$, and such that over $x_i$ there are $h_i$ pre-images, all of which have ramification order $n_i$ and $f^{h_i}$ acts as multiplication by $\zeta_{n_i}^{a_i}$ on their tangent spaces. Moreover, if $g(C) = 0$, $C$ has exactly $\gcd(\{h_i\})$ connected components. If $g(B) > 0$, then for every divisor $d$ of $\gcd(\{h_i\})$, there exists such $C$ with $d$ components.
    \end{theorem}

    \begin{proof}
        Define $B' = B \setminus\{x_1,\ldots,x_\ell\}$ and $g = g(B)$. Let $p = \chr K \geq 0$ and let $p'$ be the set of primes different to $p$. Pick a geometric generic point $\ov\eta_B \colon \Spec \ov{K(B)} \to B$. Let $D_i = \Spec \O_{B,x_i}^{sh} \xrightarrow{v_i} B$ be the spectrum of the strict henselization at $x_i$, and let $D_i^* = \Spec K(\O_{B,x_i}^{sh})$. By \cite[Exposé XIII Corollaire 2.12]{SGA1}, there exist geometric generic points
        \[\ov\eta_i \colon \Spec \ov{K(B)} \to D_i,\]
        elements $\alpha_1,\beta_1,\ldots,\alpha_g,\beta_g \in \pi_1^{p'}(B',\ov\eta_B)$ and generators $\tau_i \in \pi_1^{p'}(D_i^*,\ov\eta_i)$ such that $v_i(\ov\eta_i) = \ov\eta_B$, and where if $\sigma_i$ is the image of $\tau_i$ under the map $\pi_1^{p'}(D_i^*,\ov\eta_i) \to \pi_1^{p'}(B',\ov\eta_B)$, then $\pi_1^{p'}(B',\ov\eta_B)$ is the pro-$p'$ group generated by $\alpha_i,\beta_i,\sigma_j$ subject to the condition
        \[[\alpha_1,\beta_1]\cdot\ldots\cdot[\alpha_g,\beta_g] \cdot \sigma_1\cdot\ldots\cdot \sigma_\ell = 1.\]
        Recall that we have chosen a compatible sequence of roots of unity $\{\zeta_m\}_{(m,p) = 1}$. This determines a canonical isomorphism
        \[\pi_1^t(D_i^*) = \pi_1^{p'}(D_i^*) \overset{\sim}{\longrightarrow} \Z_{p'} = \varprojlim_{(m,p)=1} \Z/m\Z = \prod_{q \neq p} \Z_{q}\]
        in the following way. If $t_i$ is a uniformizing parameter of $\O_{B,x_i}^{sh}$, then any tamely ramified cover of $D_i$ is of the form
        \[Y_i = \Spec \O_{B,x_i}^{sh}[T]/(T^m - t_i) \to \Spec \O_{B,x_i}^{sh},\]
        where the Galois group is generated by $T \mapsto \zeta_m T$. Therefore, any $\tau \in \pi_1^{p'}(D_i^*)$ must act on the tangent space of the closed point $y_i \in Y_i$ as multiplication by $\zeta_m^{e_m}$ for some integer $e_m$ defined modulo $m$. We may define the image of $\tau$ in $\Z_{p'}$ as $\{e_m\}_{m}$.

        Now, \Cref{cover-action-restriction} tells us that all the $\tau_i$ must have the same image in $\Z_{p'}$. Indeed, there exist homomorphisms $\varphi \colon \pi_1^{p'}(B') \to \Z/m\Z = G$ with $\varphi(\sigma_i) = c_i$ whenever $\sum c_i \equiv 0 \mod m$. Such morphism corresponds to a cyclic Galois cover $C \to B$ ramified over $x_i$, and we let $f \in \Aut(C \to B) \cong G^{op}$ be the generator of deck transformations corresponding to $1$. Let $\ti h_i = \gcd(c_i,m)$ and $m_i = m/\ti h_i$, so that there are $\ti h_i$ points over $x_i$. By looking at the induced cover $D_i \times_B C \to D_i$, we see that $f^{c_i}$ acts on the tangent space of points over $x_i$ as multiplication by $\zeta_{m_i}^{e_{i,m_i}}$, where $\{e_{i,m}\}$ is the image of $\tau_i$ in $\Z_{p'}$. This means that $f^{\ti h_i}$ acts as $\zeta_{m_i}^{e_{i,m_i}\ti a_i}$, where $\ti a_i$ is the inverse of $c_i/h_i$ modulo $m_i$. Then we know that
        \[\sum_{i=1}^\ell e_{i,m_i}c_i \equiv 0 \mod m.\]
        This equation must be true for every choice of $c_i$ with $\sum c_i \equiv 0 \mod m$, which forces all the $e_{i,m_i}$ to be equal. We will just denote the common image of all the $\tau_i$ as $\{e_m\}$.

        Consider covers $C \to B$ associated to morphisms
        \[\varphi \colon \pi_1^{p'}(B') \to \Z/n\Z,\]
        satisfying $\varphi(\sigma_i) = \ov e h_iq_i$, where $\ov e$ is the inverse of $e_n$ modulo $n$. These morphisms are well defined by hypothesis, and the same analysis shows that the action of $f^{h_i}$ on the tangent spaces above $x_i$ is multiplication by $\zeta_{n_i}^{a_i}$.

        Finally, the number of connected components of $C$ equals the index of the image of $\varphi$ in $\Z/m\Z$. When $g(B) = 0$, there is a unique $\varphi$ satisfying the above requirement, and for which, the index is $\gcd(\{h_i\})$. When $g(B) > 0$, we may choose $\varphi$ by setting $\varphi(\alpha_1) = d$ and $\varphi(\alpha_j) = \varphi(\beta_j) = 0$ for the other generators.
    \end{proof}
    
    Using Riemann-Hurwitz, we can define Euler characteristic of a building block $\ti X_0$ as follows. Recall that the principal component $\ti C \subseteq \ti X_0$ is the image of the connected smooth curve $C$ with which we obtain $\ti X_0$ as a quotient.

    \begin{definition}[{compare with \cite[Definition 6.3]{Dok25}}]\label{Euler-characteristic-definition}
        Let $\ti C \subseteq \ti X_0$ be the principal component of a building block of multiplicity $n$ coprime to $\chr K$, and let $x_1,\ldots,x_k \in \ti C$ be a list of points which at least contains all $x_i$ over which $C \to \ti C$ is ramified. We define the Euler characteristic of $\ti X_0$ as the topological Euler characteristic of $C \subseteq X$ before the quotient, which equals
        \[\chi(\ti X_0) = \chi(C) = -\deg \omega_C = n(\chi(\ti C) - k) + \sum_{i=1}^k h_{x_i}.\]
        where $k$ and $h_i$ are as in \Cref{existence-cyclic-galois-cover}. If $\mathcal{I} \subseteq \{x_1,\ldots,x_k\}$, we define the Euler characteristic of the pair $(\ti X_0,\mathcal{I})$ as
        \[\chi(\ti X_0,\mathcal{I}) = -\deg \omega_C(-\pi^{-1}(\mathcal{I})_{red}) = n(\chi(\ti C) - k) + \sum_{x_i \not\in \mathcal{I}} h_{x_i},\]
        where $\pi \colon C \to \ti C$ is the quotient map. Note that $\chi(\ti X_0,\mathcal{I})$ depends only on $\ti X_0$ and the subset $\mathcal{I} \subseteq \ti C$, not on the list $\{x_1,\ldots,x_k\}$.
    \end{definition}

    More generally,
    \begin{definition}
        Let $\mu_n \curvearrowright C \subseteq X \to D$ be an equivariant action on a generically smooth fibration with $C$ nodal, and let $\ti X_0$ be the central fiber on the minimal resolution $\ti X \to X^{tw} \to D/\mu_n$. We define
        \[\chi(\ti X_0) = -\deg \omega_C.\]
    \end{definition}
    
    Let $\ti X_0$ be a singular fiber as above, and let $\ti X^\nu_0 = \bigcup \ti X^\nu_{0,i}$ the associated collection of building blocks. Let $\mathcal{I}_i \subseteq \ti C^\nu_i$ be the subset of points along which a glueing operation (either of both) is applied to obtain $\ti X_0$. Then the next two results follow easily.

    \begin{theorem}[{compare with \cite[Theorem 6.4]{Dok25}}]
        With $\ti X_0$ and $\ti X^\nu_{0,i}$ as above,
        \begin{align*}
            \chi(\ti X_0) &= \sum_{i} \chi(\ti X_{0,i}^\nu,\mathcal{I}_i).
        \end{align*}
    \end{theorem}
    \begin{proof}
        Let $\iota^\nu$ be the composition $C_i^\nu \to C_i \hookrightarrow C$, then
        \[(\iota^{\nu})^*(\omega_C) \cong \omega_{C_i^\nu}(-(\iota^{\nu})^{-1}\pi^{-1}(\mathcal{I}_{i})_{red}).\]
        The rest follows by adding over the degrees over all orbits of all $C_i$.
    \end{proof}

    Note that two singular fibers with possibly different non-reduced structure but obtained by glueing the same building blocks along the same points must have the same Euler characteristic.

    \begin{lemma}[Nodal Riemann-Hurwitz for Galois Covers]\label{nodal-RH}
        Let $C \to B$ be a Galois cover of nodal curves of degree $n$, and for $p \in B$ let $h_p$ be as in \Cref{numbers}. Then
        \[\chi(C) =n(\chi(B) - k_1 - k_4) + \sum_{\text{type 1}} h_p.\]
        where $k_1$ and $k_4$ are, respectively, the number of type 1 and 4 points in the base $B$, as in \Cref{point-types}.
    \end{lemma}
    \begin{proof}
        Let $k_2$ and $k_3$ the number of type 2 and 3 respectively in $B$. Then usual Riemann-Hurwitz on $C^\nu \to B^\nu$ gives us
        \[\chi(C^\nu) = n(\chi(B^\nu) - k_1 - 2k_2 - 2k_3 - k_4) + \sum_{\text{type 1}} h_p^\nu + 2\p{\sum_{\text{type 2,3}} h_p^\nu} + \sum_{\text{type 4}} h_p^\nu.\]
        On $C$ we have
        \[\chi(C) = \chi(C^\nu) - 2\#\text{nodes}(C) = \chi(C^\nu) - 2\p{\sum_{\text{type 2,3,4}} h_p} = \chi(C^\nu) - 2\p{\sum_{\text{type 2,3}} h_p^\nu} - \sum_{\text{type 4}} h_p^\nu.\]
        On $B$ we have
        \[\chi(B) = \chi(B^\nu) - 2\#\text{nodes}(B) = \chi(B^\nu) - 2k_2 - 2k_3.\]
        The result follows from an easy substitution.
    \end{proof}

    For the following, we say that an irreducible component $E \subseteq \ti X_0$ is non-destabilizing if either (1) it has genus $g \geq 2$, (2) it has genus $g = 1$ and intersects at least another curve in $\ti X_0$, or (3) it has genus $g = 0$ and intersects at least three curves in $\ti X_0$.
    \begin{theorem}\label{equivalence-of-principal-under-stability}
        Suppose that $\mu_n \curvearrowright C \subseteq X \to D$ is an equivariant action, where $C$ is a stable, possibly disconnected curve. Then the set of non-destabilizing curves in $\ti X_{0}$ is precisely the set of principal components together with the curves intersecting three components in resolutions over type 4 orbits (c.f. \Cref{point-types,swap-local-glueing-algorithm}).
    \end{theorem}

    \begin{proof}
        The only non-destabilizing curves extracted in the resolution $\ti X \to X^{tw}$ must be the one intersecting three components in the resolution of type 4 orbits. Thus it suffices to show that every principal component is non-destabilizing. We know that $\ti C_i \subseteq \ti X_0$ is non-destabilizing if and only if
        \[\chi(\ti C_i) - (\ti X_{0,red}-\ti C_i)\cdot \ti C_i < 0.\]
        Let $D = \bigcup f^{j}(C_i)$ be the full orbit of $C_i$, so that $D \to \ov C_i$ is Galois of order $m_i$ dividing $n$. Then the above condition is equivalent to
        \[\chi(\ov C_i) - k_0 - k_1 - k_4 < 0,\]
        where $k_0$ is the number of points in $\ov C_i$ intersecting another $\ov C_j$ at a point $p \in \ov C_i$ which is not a ramification point of $D \to \ov C_i$; and $k_1$, $k_4$ are the number of type 1 and 4 points in $\ov C_i$ (considering only the quotient $D \to \ov C_i$) respectively.
        By \Cref{nodal-RH}, we have
        \[\chi(\ov C_i) - k_0 - k_1 - k_4 = \frac{1}{m_i}\p{\chi(D) - \sum_{\text{type 1}} h_p - m_ik_0} \leq \frac{1}{m_i}\p{\chi(D) - (C-D)\cdot D} < 0.\]
        The last inequality holds since $D \subseteq C$ is non-destabilizing by stability of $C$.
    \end{proof}

    \begin{remark}
        This theorem says that, when $C$ is stable and $\ti X_0$ is obtained from $C$ by a quotient, then the principal components of $\ti X_0$ together with the degree $3$ components in D-tails are precisely the principal components of $\ti X_0$ in the sense of \cite[Definitions 4.1 and 6.3]{Dok25}.
    \end{remark}

    Another way of interpreting \Cref{existence-cyclic-galois-cover} is that a potential building block that is combinatorially consistent does indeed exist.

    \begin{definition}
        Let $\chi \in \Z$. Define $\Lambda_{\chi,n}$ as the set of equivalence classes of pairs $(\ti X_0,\mathcal{I})$ where $\ti X_0$ is a building block of multiplicity $n$ and $\mathcal{I} \subseteq \ti C$ is a subset of points, such that
        \[\chi(\ti X_0,\mathcal{I}) = \chi.\]
        Define
        \[\Lambda_{\chi} = \bigcup_{\substack{n \geq 1 \\\chr K \nmid n}} \Lambda_{\chi,n}, \qquad \text{ and } \qquad \Lambda = \bigcup_{\chi \in \Z} \Lambda_{\chi}.\]
    \end{definition}

    \begin{proposition}
        Let $n$ coprime to $\chr K$. Let $B$ be a smooth connected curve. Let $\{x_1,\ldots,x_k\} \subseteq B$ and $\mathcal{I} \subseteq B$ be sets of  points. Let $n_ih_i = n$ be a choice of factorization of $n$ for $i=1,\ldots,k$ and let $0 \leq q_i < n_i$ be coprime to $n_i$ such that
        \[\sum h_iq_i \equiv 0 \mod n.\]
        Then there is a pair $(\ti X_0, \mathcal{I}) \in \Lambda$ where $\ti X_0$ is a building block of multiplicity $n$, $B \cong \ti C \subseteq \ti X_0$, and such that the exceptional chains of $\ti X_0$ are given by the fractions $\frac{n_i}{q_i}$. Conversely, any pair $(\ti X_0,\mathcal{I}) \in \Lambda$ is of this form.
    \end{proposition}

    We say two pairs are combinatorially equivalent if they have the same dual graph including the labeling of multiplicities, genera and the markings $\mathcal{I}$.

    \begin{theorem}
        For $\chi < 0$, there are finitely many combinatorially equivalence classes in $\Lambda_{\chi}$, and they are effectively computable.
    \end{theorem}

    \begin{proof}
        Building blocks are classified by pairs $(C,f)$ where $C$ is a smooth curve and $f$ is an automorphism. Moreover, the condition $\chi(\ti X_0,\mathcal{I}) < 0$ corresponds to $(C,\pi^{-1}(\mathcal{I}))$ being stable. Thus finiteness is a formal consequence of the stack of divisorially marked curves $\M_{g,n}^D = [\M_{g,n}/S_n]$ being a stack of finite type with finite inertia, where the combinatorial types in the inertia stack $\mathcal{I}_{\M_{g,n}^D}$ form a locally closed stratification. Another proof together with an algorithm is found in \cite[Theorems 7.5(6) and 8.1(3)]{Dok25}, which we summarize here. A combinatorial analysis shows that, if $\ti X_0$ has multiplicity $n$ and $\chi(\ti X_0,\mathcal{I}) = \chi < 0$, then
        \[1 \leq n \leq 6-2\chi.\]
        Having $n$ fixed, we can bound $\chi(\ti C)$. It has to be an even integer in the range
        \[\frac{\chi}{n} \leq \chi(\ti C) \leq 2.\]
        Further fixing $\chi(\ti C)$, the number $\ell$ counting the points $x \in \ti C$ such that $n_x > 1$ lies in the range
        \[0 \leq \ell \leq \frac{-2\chi}{n} + 2\chi(\ti C).\]
        Having also fixed $\ell$, there are finitely many (and effectively computable) pairs of tuples $(n_1,\ldots,n_\ell)$ and $(b_1,\ldots,b_\ell)$ with $n_i > 1$ dividing $n$ and $1 \leq b_i < n_i$, such that $n$ divides $\sum b_i$. We can recover $h_i = \gcd(b_i,n)$ and $q_i = \frac{b_i}{h_i}$. Next, for every sub-tuple $\mathcal{I}_{>1}$ of $(n_1,\ldots,n_\ell)$ of size $r \leq \ell$, we determine if the quantity
        \[s = \frac{-\chi}{n} + \chi(\ti C) - r - \sum_{n_j \in \mathcal{I}_{>1}}\p{1 - \frac{1}{n_j}}\]
        is an integer. If so, defining
        \[\mathcal{I} = \mathcal{I}_{> 1} \cup \{y_1,\ldots,y_s\}\]
        for some choice of points $y_j$ with $n_{y_j} = 1$  determines an equivalence class of pairs $(\ti X_0,\mathcal{I})$ in $\Lambda_\chi$.
    \end{proof}

    As promised in \Cref{the-promise}, we will now prove the existence of arbitrary glueing of building blocks exists. With this, we have a justification for the correctness of \cite[Algorithm 10.10]{Dok25} independently of Winters' result.

    \begin{theorem} \label{arbitrary-glueing-exists}
        Let $\ti Y_0 = \bigcup \ti Y_{0,i}$ be a finite collection of building blocks with principal components $\ti C = \bigcup \ti C_i$. Let $\mathcal{J}$ be a finite indexing set and $\{(p_\alpha,p_\alpha')\}_{\alpha \in \mathcal{J}}$ be pairs of points in $\ti C$ such that
        \begin{itemize}
            \item For every $\alpha \in \mathcal{J}$, $h_{p_\alpha} = h_{p_\alpha'}$.
            \item $p_\alpha \neq p_\beta$ and $p_\alpha \neq p_\beta'$ if $\alpha \neq \beta$.
            \item If $p_\alpha = p_\alpha'$, then $h_{p_\alpha}$ is even.
        \end{itemize}
        For $\alpha \in \mathcal{J}$, let $i_\alpha$ be the index such that $p_\alpha \in \ti Y_{0,i_\alpha}$, and similarly for $i_\alpha'$. Let $\{j_\alpha\}_{\alpha \in \mathcal{J}}$ be integers such that
        \begin{itemize}
            \item $j_\alpha \geq -1$ if $p_\alpha \neq p_\alpha'$ and $\frac{a_{p_\alpha}}{n_{i_\alpha,p_\alpha}} + \frac{a_{p_\alpha'}}{n_{i_\alpha',p_\alpha'}} > 1$ in the notation of \Cref{non-swap-local-glueing-algorithm}(3) and (4), or
            \item $j_\alpha \geq 0$ otherwise.
        \end{itemize}
        Then there exists a tame singular fiber $\ti X_0$ with $\ti X_0^\nu \cong \ti Y_0$, which is combinatorially obtained from $\ti Y_0$ in the sense of \Cref{global-glueing-algorithm} by glueing the chains over the points $p_\alpha$ and $p_\alpha'$ using the non-swapping \Cref{non-swap-local-glueing-algorithm} if $p_\alpha \neq p_\alpha'$ and using the swapping \Cref{swap-local-glueing-algorithm} with the data $\{j_\alpha\}$ instead, if $p_\alpha = p_\alpha'$.
    \end{theorem}

    \begin{remark}
        The condition $h_{p_{\alpha}} = h_{p_{\alpha'}}$ is equivalent to the last components in the chains over $p_\alpha$ and $p_{\alpha'}$ having the same multiplicity.
    \end{remark}

    \begin{proof}
        Let $C_i^\nu \to \ti C_i^\nu \subseteq \ti Y_{0,i}$ be the quotient map from the smooth, possibly disconnected curve $C_i^\nu$ to the principal component of $\ti Y_{0,i}$, and let $f_i \in \Aut(C_i^\nu)$ be the automorphism used to obtain $\ti Y_{0,i}$. We construct a curve $C$ with an automorphism $f \in \Aut(C)$ by glueing points in $C^\nu = \bigcup C_i^\nu$ along the pre-images of $p_\alpha$ and $p_\alpha'$:
        \begin{itemize}
            \item If $p_\alpha \neq p_\alpha'$, by hypothesis there exist $h_{p_\alpha}$ points over $p_\alpha$ and the same number over $p_\alpha'$. We glue them in any way which is compatible with the $f_i$'s, which is always possible.
            \item If $p_\alpha = p_\alpha'$, we glue a point $y$ over $p_\alpha$ to the point $f^{h_{p_\alpha}/2}(y)$, which is possible since $h_{p_\alpha}$ is even. This glueing is also compatible with the $f_i$'s. 
        \end{itemize}
        As this glueing is compatible with the $f_i$, it determines $f \in \Aut(C)$ by glueing the automorphisms $f_i$.

        Using the integers $j_\alpha$, we may define the integers
        \[k_\alpha = d_{i_\alpha}a_{p_\alpha} + d_{i_\alpha'}a_{p_\alpha'} + j_\alpha n_{p_\alpha} \geq 1, \quad \text{ if } p_\alpha \neq p_\alpha', \text{ or }\]
        \[k_\alpha = d_{i_\alpha}\ov{2a_{p_\alpha}} + j_\alpha n_{p_\alpha} \geq 1, \quad \text{ if } p_\alpha = p_{\alpha}',\]
        where $\ov{2a_{p_\alpha}}$ is the unique residue modulo $n_{i_\alpha,p_\alpha}$ of $2a_{p_\alpha}$ in the range $[1,n_{i_\alpha,p_\alpha}]$. These $k_\alpha$ will play the role of the singularity type of the nodes in the ambient surface $X \to D$. So using the data $(C,f,\{k_\alpha\})$, we construct a surface $X \to D$ using \Cref{weak-realization-theorem}. By \Cref{non-swap-local-glueing-algorithm,swap-local-glueing-algorithm,global-glueing-algorithm}, we see that the combinatorics of the resulting singular fiber $\ti X_0$ is obtained from $\ti Y_0$ as in the statement.
    \end{proof}

    This completes the parallel between our method and the one worked out in \cite{Dok25}. Knowing that singular fibers arise as quotients of stable curves, we can make use of the combinatorics done in \cite{Dok25} to prove a strengthening of \cite[Corollary 4.3]{Winters1974} which answers \cite[Question 3.8]{Dok25} in characteristic 0, or in positive characteristic when the fiber is ``tame''.

    \begin{definition} \label{combinatorial-singular-curve-definition}
        A combinatorial singular fiber is a pair $(\Gamma, \{m_i\})$ where $\Gamma = \bigcup_{i\in I} \Gamma_i$ is a connected, reduced normal crossings curve, and $m_i \geq 1$ are integers such that for every $i \in I$,
        \[\Gamma_i^2 := -\frac{1}{m_i}\sum_{j \neq i} m_j \cdot \#(\Gamma_i \cap \Gamma_j)\]
        is an integer. We define the Euler characteristic of a component $\Gamma_i$ as
        \[\chi_\Gamma(\Gamma_i) = -2m_i(g(\Gamma_i)-1) - \sum_{j \neq i} m_j\cdot\#(\Gamma_i \cap \Gamma_j),\]
        and the Euler characteristic of the full fiber as its sum
        \[\chi = \chi(\Gamma,\{m_i\}) = \sum_i \chi_\Gamma(\Gamma_i) = -\sum_i \p{2m_i(g(\Gamma_i)-1) + \sum_{j \neq i} m_j\cdot\#\p{\Gamma_i \cap \Gamma_j}}.\]
        We say that $\Gamma_i$ is a $(-1)$-curve if $\Gamma_i \cong \P^1$ and $\Gamma_i^2 = -1$. 
    \end{definition}
    
    \begin{theorem} \label{existence-of-geometrically-connected-fiber}
        Let $d \geq 1$ be an integer and let $(\Gamma,\{m_i\})$ be a combinatorial singular fiber of Euler characteristic $2d(1-g)$, for $g \geq 2$. Assume that any curve $\Gamma_i$ with $m_i$ divisible by $\chr K$ is isomorphic to $\P^1$ and intersects the rest of $\Gamma$ in at most two points. Then the following are equivalent.
        \begin{enumerate}
            \item There exists a proper morphism $\ti \pi \colon \ti X \to B$ with smooth geometric generic fibers, where $\ti X$ is a smooth surface, $0 \in B$ is a smooth point in a curve, and $\ti X_{0,red} \cong \Gamma$ where the multiplicity of $\Gamma_i$ in $\ti X_{0}$ is $m_i$, and $\dim H^0(\ti X_0,\O_{\ti X_0}) = d$.
            \item At least one of the following hold
            \begin{itemize}
                \item $g(\Gamma) \geq 1$ and $d$ divides $\gcd(\{m_i\})$.
                \item $\gcd(\{m_i\}) = d$.
            \end{itemize}
        \end{enumerate}
    \end{theorem}

    First let us translate the statement $\dim H^0(\ti X_0,\O_{\ti X_0}) = d$ to the statement that the generic fiber of $\ti X$ having $d$ components.

    \begin{lemma}\label{quotient-lemma}
        Let $X \to D$ be a map between quasi-projective schemes over $K$, and let $G$ be a finite linearly reductive group acting equivariantly on $X \to D$. Let $B \to D/G$ be any map and define $B' = D \times_{D/G} B$ and $X' = X \times_D B'$. Then the diagram
        \[\begin{tikzcd}
            X'/G \ar[r] \ar[d] & X/G \ar[d]\\
            B \ar[r] & D/G
        \end{tikzcd}\]
        is cartesian.
    \end{lemma}
    \begin{proof}
        This reduces to the algebraic statement that: Given a $K$-algebra $R$ with a $G$ action, an $R$-algebra $S$ with a $G$-equivariant action, and a $R^G$-algebra $T$, then $\p{S \otimes_{R^G} T}^G \cong S^G \otimes_{R^G} T$. This holds because $G$ is linearly reductive, so $S^G$ is a $G$-equivariant direct summand in $S$. Indeed, writing $S = S^G \oplus S'$ where $(S')^G = 0$, then taking a free resolution $(R^G)^{\oplus J} \to T \to 0$ we obtain a surjection $(S')^{\oplus J} \to S' \otimes T \to 0$ where $G$ acts diagonally on $(S')^{\oplus J}$. As $G$ is linearly reductive, taking invariants is an exact functor, and so, since $(S')^{\oplus J, G} = (S'{}^G)^{\oplus J} = 0$, we obtain $\p{S' \otimes T}^G = 0$.
    \end{proof}
    \begin{theorem}\label{geometrically-connected-theorem}
        Let $n$ be coprime to $\chr K$ and let $\mu_n \curvearrowright C \subseteq X \xrightarrow{\pi} D$ be an equivariant action on a generically smooth fibration with nodal central fiber, and let $r = \rank \pi_*\O_X$ be the number of connected components of a geometric generic fiber. If $\pi^{tw} \colon X^{tw} = X/\mu_n \to D$ is the quotient and $X^{tw}_0$ is the central fiber, then $\rank \pi^{tw}_*\O_{X^{tw}} = r$ and
        \[\dim H^0(X^{tw}_0,\O_{X^{tw}_0}) = r.\]
    \end{theorem}
    \begin{proof}
        First, it is clear that $\rank \pi^{tw}_*\O_{X^{tw}}$ is $r$ as $\mu_n$ acts generically freely on $D$.
        
        We know that $X \to D$ is nodal, so in particular every component has multiplicity $1$. This implies that the Stein factorization $X \to D' \to D$ is unramified, and thus, $X \to D$ satisfies cohomology and base change in the sense of \cite[Proposition 2.1]{Con11}. Now, the fiber over $0$ under the quotient $\tau \colon D \to D/\mu_n$ is $\tau^{-1}(0) = \Spec A$, where $A = K[t]/(t^n)$, and $\Spec A/\mu_n = \Spec K$. By cohomology and base change,
        \begin{equation}\label{global-sections-of-thickened-fiber}
            H^0(X_{\tau^{-1}(0)},\O_{X_{\tau^{-1}(0)}}) = (f_*\O_X)|_{\tau^{-1}(0)} \cong A^r,
        \end{equation}
        where this last equality holds as $K$-vector spaces, not necessarily as $\mu_n$-representations. By \Cref{quotient-lemma}, the thickened fiber $X_{\tau^{-1}(0)}$ satisfies $X_{\tau^{-1}(0)}/\mu_n = X^{tw}$, and so
        \[H^0(X_{\tau^{-1}(0)},\O_{X_{\tau^{-1}(0)}})^{\mu_n} = H^0(X^{tw}_0,\O_{X_0^{tw}}).\]
        Thus, we are reduced to verifying that the action of $\mu_n$ on $A^r$ induced from \Cref{global-sections-of-thickened-fiber} is such that $\dim_K (A^r)^{\mu_n} = r$. Let $F_1,\ldots,F_r$ be the connected components of $X_0$, so that
        \[H^0(X_0,\O_{X_0}) = \bigoplus_{i=1}^r H^0(F_i,\O_{F_i}) \cong \bigoplus_{i=1}^r K =: V,\]
        where each summand corresponds to constant functions on $F_i$. This determines a $\mu_n$-equivariant section
        \[V \to H^0(X_{\tau^{-1}(0)},\O_{X_{\tau^{-1}(0)}}),\]
        which we use to construct an injective, and thus also surjective, morphism of $\mu_n$-representations
        \[A \otimes_K V \to H^0(X_{\tau^{-1}(0)},\O_{X_{\tau^{-1}(0)}}).\]
        Now, let us decompose $A$ and $V$
        \[A = \bigoplus_{i=0}^{n-1} A_i, \qquad V = \bigoplus_{i=0}^{n-1} V_i\]
        into their sub-representations $A_i$, $V_i$, where $\zeta_n$ acts as multiplication by $\zeta_n^i$. Note that $A_i = Kt^i \subseteq K[t]/(t^n) \cong K$ is one-dimensional. We have
        \[(A \otimes_K V)^{\mu_n} = \p{\bigoplus_{i,j=0}^{n-1} A_i \otimes_K V_j}^{\mu_n} = \bigoplus_{i+j \equiv 0 \!\!\mod n} A_{i} \otimes_K V_{j} \cong \bigoplus_{j=0}^{n-1} V_j \cong V,\]
        which proves that
        \begin{equation*}
            \dim H^0(X_0^{tw},\O_{X_0^{tw}}) = r.\qedhere
        \end{equation*}
    \end{proof}

    The following is a partial strengthening of the main result of \cite{Bresciani2023}. This says that even if a Deligne-Mumford stack $\mathcal{X}$ is not tame, if a rational map $\Spec R \dasharrow \mathcal{X}$ satisfies a tame valuative criterion, then this extension can still be taken as a root stack.

    \begin{lemma}\label{semi-tame-valuative-criterion}
        Let $f \colon \mathcal{X} \to \mathcal{Y}$ be a proper DM morphism of algebraic stacks, and let $R$ be a DVR with quotient field $L$. Suppose that we have a 2-commutative square
        \[\begin{tikzcd}
            \Spec L \ar[r] \ar[d,hook] & \mathcal{X} \ar[d,"f"]\\
            \Spec R \ar[r] & \mathcal{Y}
        \end{tikzcd}\]
        Moreover, suppose that there exists tamely ramified quasi-finite extension of DVRs $R \hookrightarrow R'$ such that a lift
        \[\begin{tikzcd}
            && \mathcal{X} \ar[d,"f"]\\
            \Spec R' \ar[r] \ar[rru] & \Spec R \ar[r] & \mathcal{Y}
        \end{tikzcd}\]
        exists. Then there exists a unique positive integer $n$ and a representable lifting $\sqrt[n]{\Spec R} \to \mathcal{X}$ making the diagram 
        \[\begin{tikzcd}
            &\Spec L \ar[d,hook] \ar[dl,hook] \ar[r] & \mathcal{X} \ar[d,"f"]\\
            \sqrt[n]{\Spec R} \ar[r] \ar[rru] & \Spec R \ar[r] \ar[from=u,hook,crossing over] & \mathcal{Y}
        \end{tikzcd}\]
        2-commute. 
    \end{lemma}

    \begin{proof}
        Following \cite[Proof of Theorem 3.1]{Bresciani2023}, we may assume that $\mathcal{Y} = \Spec R$ so $\mathcal{X}$ is Deligne-Mumford, that $\Spec L$ is a dense open in $\mathcal{X}$, and that $R$ and $R'$ are Nagata. By passing to the relative (and thus absolute) normalization of $\mathcal{X}$ and since $\Spec R'$ is also normal, by \cite[Lemma A.5 (3)]{Ascher2020}, we may further assume that $\mathcal{X}$ is normal. By \cite[Lemma 3.11]{Bresciani2023}, which does not actually use the tameness hypothesis, we know that the property of being a root stack is étale local on $\Spec R$, so we may further assume that $R$ is strictly henselian. Let $p\geq 0$ be the characteristic of the residue field of $R$. Since $\pi_1^t(\Spec R \setminus 0) \cong \prod_{\ell \neq p} \Z_\ell$, we know that $R/R'$ is Galois with cyclic Galois group $G$ of order coprime to $p$.

        Now let $U \to [U/H] \to \mathcal{X}$ be a connected, étale local representation of $\mathcal{X}$ as a quotient stack, where $H$ is the stabilizer of the point above $0 \in \Spec R$. Since $\Spec R$ is strictly henselian and $U$ is normal, we have that $[U/H] = \mathcal{X}$ and $U/H = \Spec R$, and it follows that $H$ acts generically freely on $U$. This means that the sequence of Galois covers
        \[\Spec R' \to U \to \Spec R\]
        corresponds to the chain of inclusions
        \[G \hookleftarrow H \hookleftarrow 0,\]
        so $H$ must also have order coprime to $p$, thus $\mathcal{X}$ is tame. By \cite[Proposition 3.12]{Bresciani2023}, $\mathcal{X} = \sqrt[n]{\Spec R}$, and representability and uniqueness follow.
    \end{proof}

    \begin{lemma}\label{gcd-equals-hp}
        Let $\ti X_0$ and be singular fiber, which is obtained by an action $f \in \Aut(C)$, and let $m_i$ and $h_p$ be as in \Cref{numbers}. Let $\alpha_j$ be the multiplicities of exceptional curves in $\ti X_0 \to X_0^{tw}$. Then
        \[\gcd(\{m_i\}_i \cup\{\alpha_j\}_j) = \gcd(\{h_p\}_{p \in C}).\]
        Here, $p$ ranges over all points in $C$, so in particular considers general points where $h_p = m_i$.
    \end{lemma}
    \begin{proof}
        We first show that this holds when $\ti X_0$ is a disjoint union of building blocks. Let $C_i \subseteq C$ be a connected component. Looking at the building block $\ti X_{0,i}$, let $E_1,\ldots,E_\ell$ be a chain over a point $p$. Then the multiplicity of $E_j$ in $\ti X_0$ is $h_p\alpha_j$, where $\alpha_j$ is obtained as in \Cref{a_i}. Note that in the sequence $n_{i,p} = a_0,a_1,\ldots,a_{\ell_p}$, each number is coprime to the next one, which implies that the greatest common divisor of the multiplicity of $\ti C_i$ and those of $E_j$ is $h_p$. This shows that $\gcd(m_i,\{\alpha_j\}) = \gcd(\{h_p\})$ where the $j$ ranges over chains in $\ti X_{0,i}$ and $p$ ranges in $C_i$. This still makes sense even when $\ti X_{0,i}$ has no exceptional curves, as in that case, $h_p = m_i$ for every $p$.

        For the general case, we use the fact that $\ti X_0$ is obtained, combinatorially, by glueing chains in $\ti X_0^\nu$. Note that applying \Cref{non-swap-local-glueing-algorithm} to any pair of points in $\ti C \subseteq \ti X_0^\nu$ preserves the $\gcd$ of the multiplicities of the chains and the principal components. When applying \Cref{swap-local-glueing-algorithm} to a point $p$ in $\ti C^\nu$, the set of multiplicities can only be changed by possibly adding two extra values corresponding to the $\P^1$'s at the end of the D-tail. These $\P^1$'s have multiplicity $h_p = \frac{1}{2}h_p^\nu$, so the result still follows from the first case.
    \end{proof}

    \begin{lemma}\label{components-equals-gcd}
        Let $C$ be a possibly disconnected nodal curve and $f \in \Aut(C)$ be an automorphism of order coprime to $\chr K$. If $C/f$ is connected and has genus $0$, then $C$ has $\gcd(\{h_{p}\})$ connected components.
    \end{lemma}

    \begin{proof}
        We prove this by induction on the number of components of $\ov C = C/f$. When $\ov C$ is irreducible, then $\ov C \cong \P^1$. If $C$ is smooth, the result follows from \Cref{existence-cyclic-galois-cover}. Otherwise, there are points $p_1,\ldots,p_k \in \ov C$ such that over $p_i$, $C$ is nodal and $f$ must swap the branches at each node. If $q_{i,j}$ are the nodes over $p_i$, let $q_{i,j,1}$ and $q_{i,j,2}$ be their normalizations on $C^\nu$. Then $q_{i,j,1}$ and $q_{i,j,2}$ belong to the same component of $C^\nu$ if $f^{h_{p_i}} = f^{h_{p_i}^\nu/2}$ does not permute the components of $C$, or in other words, if $\gcd(\{h_p^\nu\})$ divides $h_{p_i}$. In case it does permute them, it follows that $f^{\gcd(\{h_{p}^\nu\})/2}$ is the first power of $f$ which maps the component containing $q_{i,j,1}$ to the one containing $q_{i,j,2}$ and vice-versa. This power of $f$ is independent of $p_i$, so we have that either
        \begin{itemize}
            \item the number of connected components of $C$ is $\gcd(\{h_p^\nu\})$ if $\gcd(\{h_p^\nu\})$ divides every $h_{p_i}$. In this case, $\gcd(\{h_p^\nu\}) = \gcd(\{h_p\})$.
            \item otherwise, as long as there is at least one $p_i$ for which $h_{p_i}$ is not divisible by $\gcd(\{h_p^\nu\})$, then the number of connected components of $C$ is $\frac{1}{2}\gcd(\{h_p^\nu\})$. In this case, $\gcd(\{h_p\}) = \frac{1}{2}\gcd(\{h_p^\nu\})$.
        \end{itemize}
        With this we can deduce that the number of connected components of $C$ is always $\gcd(\{h_p\})$.

        Now suppose that $\ov C = \ov C_1 \cup \ov C_2$, where $\ov C_1$ intersects $\ov C_2$ at a single point $p_0$, and write $C = C_1 \cup C_2$. Then by induction hypothesis the number of connected components of $\ov C_1$ is $\gcd(\{h_{p_1}\})$ where $p_1$ ranges in $\ov C_1$, and similarly, the number of connected components of $C_2$ is $\gcd(\{h_{p_2}\})$ where $p_2$ ranges in $\ov C_2$. As $C_1$ and $C_2$ are glued along the $h_{p_0}$ points over $p_0$, which are permuted cyclically, it follows that the number of connected components of $C$ is the $\gcd$ of the number of connected components of $C_1$ and $C_2$.

        A way to visualize this last fact is by reducing to a graph theory problem. If $g_1 = \gcd(\{h_{p_1}\})$ and $g_2 = \gcd(\{h_{p_2}\})$, then let us consider two collections of nodes $v_1,\ldots,v_{g_1}$ and $w_1,\ldots,w_{g_2}$, and we apply the following procedure. We connect $v_1$ with $w_1$, then $v_2$ with $w_2$ and so on. Whenever we reach $v_{g_1}$ or $w_{g_2}$, we continue the next step with $v_{g_1}$ or $w_{g_2}$ respectively. The procedure ends once we arrive back to $v_1$ and $w_1$. Then at the end there will be $\gcd(g_1,g_2)$ connected components. Indeed $v_i$ and $w_j$ end up connected if and only if $i \equiv j \mod \gcd(g_1,g_2)$.
    \end{proof}

    Finally, we summarize most of the classification result of \cite{Dok25} for use with our notation.
    \begin{definition}\label{Dokchitser-main-result-definitions}
        Let $(\Gamma,\{m_i\})$ be a combinatorial singular fiber of strictly negative Euler characteristic. Assume that every $(-1)$-curve in $(\Gamma,\{m_i\})$ intersects at least three other components. Using \Cref{combinatorial-singular-curve-definition}, we say a component $\Gamma_i \subseteq \Gamma$ is called
        \begin{itemize}
            \item a principal component if $\chi_\Gamma(\Gamma_i) < 0$,
            \item the D-tail core if $\chi_\Gamma(\Gamma_i) = 0$, or
            \item a link if $\chi_\Gamma(\Gamma_i) > 0$.
        \end{itemize}
        For every intersection point $p \in \Gamma$, if $p$ belongs to $\Gamma_i$ and $\Gamma_j$ (where we allow $i=j$ in case $p$ is a node of the component $\Gamma_i$), define $\mathfrak{o}_{i,p}$ to be the residue of $m_j$ modulo $m_i$ in the range $[1,m_i]$. Let $h_{i,p} = \gcd(\mathfrak{o}_{i,v},m_i)$ and $q_{i,p} = \frac{\mathfrak{o}_{i,p}}{h_{i,p}}$. We define $a_{i,p}$ to be the inverse of $q_{i,p}$ modulo $\frac{m_i}{h_{i,p}}$ in the range $[1,\frac{m_i}{h_{i,p}}]$.

        We say that a maximal linear configuration of links in $\Gamma$ is a chain. A chain is called an outer chain (resp. inner chain) if it is connected to the rest of $\Gamma$ at one point (resp. at two points).
    \end{definition}
    \begin{remark}
        Again, this definition differs slightly from \cite[Definition 4.1]{Dok25}. Assuming $\chi(\Gamma,\{m_i\})< 0$, it follows from \cite[Theorem 8.1]{Dok25} that the only difference is that we do not consider the core of a D-tail to be a principal component, whereas they are considered to be principal components in \cite{Dok25}.
    \end{remark}
    \begin{theorem}\label{Dokchitser-main-result}
        Let $(\Gamma,\{m_i\})$ be as above. Then $\Gamma$ is the union of its principal components, D-tail cores, and inner and outer chains. Specifically,
        \begin{enumerate}
            \item Links are precisely those components isomorphic to $\P^1$, which intersect the rest of $\Gamma$ at most twice.
            \item For every $\Gamma_i$, $m_i$ divides
            \[\sum_{p \in \Gamma_i} \mathfrak{o}_{i,p}.\]
            \item For every node $p \in \Gamma_i \cap \Gamma_j$,
            \[h_{i,p} = h_{j,p},\]
            and if $\Gamma_i$ is a link that intersects the rest of $\Gamma$ at two points $p,p'$, then $h_{i,p} = h_{i,p'}$.
            \item Principal components $\Gamma_1$ and $\Gamma_2$ are connected to each other by (possibly empty) inner chains through $p_1 \in \Gamma_1$ and $p_2 \in \Gamma_2$, whose self intersections and multiplicities (up to dividing by $\gcd(m_1,m_2)$) corresponds to those of the resolution of the singularity
            \[\frac{1}{m_1m_2}(m_2a_1,m_1a_2,1;j),\]
            where $a_i = a_{i,p}$, and $j+1$ is the number of links in the chain with multiplicity $\gcd(m_1,m_2)$. Here, $j \geq -1$ if $\frac{a_1}{m_1} + \frac{a_2}{m_2} > 1$, and $j \geq 0$ otherwise.
            \item An outer chain intersecting a principal component $\Gamma_1$ at $p$ has self intersections and multiplicities are given by the resolution of the singularity $\frac{1}{m_1}(1,q_{1,p})$.
            \item A D-tail core $\Gamma_\delta$ is isomorphic to $\P^1$ and is connected via a (possibly empty) inner chain to a principal component $\Gamma_1$ at some point $p$, and off it sprout two outer chains consisting of one link each, self intersection $-2$ and multiplicity half that of $\Gamma_\delta$. The self intersections and multiplicities of the inner chain together with the D-tail core are given by the resolution of the singularity
            \[\frac{1}{m_1/\gcd(\mathfrak{o}_{1,p},m_1)}(2a_{1,p},1;j),\]
            where $j+1$ is the number of components with multiplicity $\gcd(\mathfrak{o}_{1,p},m_1)$ in this extended chain.
        \end{enumerate}
    \end{theorem}
    \begin{definition}
        We define a D-tail as the union of a D-tail core, the two $(-2)$ outer components intersecting it, and the (possibly empty) inner chain connecting it to some principal component.
    \end{definition}
    \begin{proof}[Proof of \Cref{Dokchitser-main-result}]\ 
        \begin{enumerate}
            \item Follows from \cite[Theorem 7.5 (1) and (2)]{Dok25}.
        \end{enumerate}
        With this, the statement that $\Gamma$ is the union of principal components, D-tail cores and chains is proved using basic graph theory and \cite[Theorem 8.1]{Dok25} (where we may discard the Kodaira fibers since $\chi(\Gamma,\{m_i\}) < 0$).
        \begin{enumerate}[resume]
            \item Follows classically from the theory of Hirzebruch-Jung continued fractions, \cite[Theorem B (i)]{Dok25}.
            \item This is essentially the divisibility condition of \Cref{Dokchitser-main-result-definitions}.
            \item Follows by definition and from \cite[Lemma 5.4]{Dok25}.
            \item Follows from the statement about $\beta$ in \cite[Theorem 5.14 (1)]{Dok25}, as the multiplicities $\alpha_k$ of the exceptionals $E_k$ in the resolution of the singularity $\frac{1}{m_1m_2}(m_2a_1,m_1a_2,1;j)$ satisfy
            \[\frac{\alpha_{k-1}+\alpha_{k+1}}{\alpha_k} = -\Gamma_k^2 \in \Z_{\geq 2}.\]
            \item  The fact that each D-tail core intersects two outer and one inner chain connected to a principal component follows from \cite[Theorem 8.1 (b)]{Dok25}, where we use $\chi < 0$ to discard the Kodaira case. The statement about self intersections and multiplicities follows from \cite[Algorithm 5.16]{Dok25}. Indeed, since the multiplicity of $\Gamma_\delta$ equals that of the first link it intersects in the inner chain, $m_\delta$ must divide $m_1$. In \cite{Dok25}'s notation, the multiplicities on the chain connecting $\Gamma_\delta$ and $\Gamma_1$ are obtained, in that order, from the inner sequence associated to $m_\delta\frac{m_\delta-q_1m_\delta}{n}m_1$. In our notation, this corresponds to the singularity $\frac{1}{m_1/m_\delta}(m_1/m_\delta, a_1,1;j)$, which by \Cref{non-swap-quotient-theorem} equals
            \[\frac{1}{a_1 + \frac{m_1}{m_\delta}(j+1)}\p{1,a_1 + \frac{m_1}{m_\delta}j}.\]
            By \Cref{-2-contraction} its continued fraction expansion is
            \[[\underbrace{2,\ldots,2}_{j},e_\ell+1,\ldots,e_1],\]
            where
            \[\frac{m_1/m_\delta}{q_1} = [e_1,\ldots,e_\ell].\]
            The rest follows from \Cref{swap-quotient-theorem}.

            Some remarks on this part of the proof. In most of the cases, the inner sequence associated to a type $m\frac{d-d'}{n}m'$ includes the multiplicities of the curves through the singularity, that is, $L_1$ and $L_2$ from \cref{non-swap-local-glueing-algorithm}, so the coefficients in the continued fraction expansion should correspond to the the non-extremal terms in the inner sequence. However, since the multiplicity of $\Gamma_\delta$ equals the multiplicity of the link it intersects, we are in the case $m = c$ from \cite[Algorithm 5.16 (1)]{Dok25}, so by \cite[Algorithm 5.16 (3)]{Dok25}, the first term in the sequence does not actually correspond to $L_1$ but to the first component in the resolution.

            For similar reasons, the case $m_1 = q_1 = m_\delta$ would require special treatment as it does not satisfy the hypothesis of \Cref{-2-contraction} and the condition $m' = c$ would also apply on the other end, but the end result is the same. 
        \end{enumerate}
    \end{proof}

    \begin{proof}[Proof of \Cref{existence-of-geometrically-connected-fiber}]
        We start with the case $d = 1$. We may assume that $\Gamma$ has no $(-1)$-curves intersecting less than 3 other curves in $\Gamma$, as contracting those does not change the $\gcd$, and a smooth such surface $\ti X \to B$ exists if and only if there exists $\ti X' \to B$ having the contraction of $\Gamma$ as central fiber.

        $(1) \Rightarrow (2)$ As $\ti X \to B$ has no $(-1)$-curves in the central fiber, this is a minimal regular normal crossings (m.r.n.c.) model. Since $X \to B$ has connected geometric fibers, $\ti\pi$ corresponds to a rational classifying map $B \dasharrow \ov\M_g$. By \cite[Theorem 10.4.47]{Liu}, near $0$ there is a tamely ramified cover $B' \to B$ with a lift $B' \to \ov\M_g$ corresponding to a stable fibration $X \to B'$ with central fiber $C$. By \Cref{semi-tame-valuative-criterion}, we may assume that $X \to B'$ has an equivariant $\mu_n$-action, for $n$ coprime to the characteristic of the base field, which fixes $0 \in B'$ and $B'/\mu_n = B$. By \Cref{equivalence-of-principal-under-stability}, the minimal resolution of $X/\mu_n$ is m.r.n.c, and thus equals $\ti X$. By \Cref{gcd-equals-hp,components-equals-gcd}, if $g(\Gamma) = 0$, then $\gcd(\{m_i\}) = 1$, as $X$ has connected fibers.

        $(2) \Rightarrow (1)$ By \Cref{Dokchitser-main-result}, there is a decomposition of $\Gamma$ into principal components, inner and outer chains connecting them, and D-tails. If an irreducible component of $\Gamma$ has a node, we interpret it as a length zero inner chain.
        
        We claim that there exists a connected stable curve $C$ of genus $g$, an automorphism $f \in \Aut(C)$ of order coprime to $\chr K$ and integers $j_p \geq 1$ indexed by orbits of nodes $p$ of $C$ such that the associated central fiber $\ti X_0$ constructed using \Cref{weak-realization-theorem} has combinatorial data $(\Gamma,\{m_i\})$. We can prove this directly in characteristic zero using Winters' theorem, but we can do it combinatorially in arbitrary characteristic independently of that result.

        Let $\ov C^\nu$ consist of the normalized principal components in $\Gamma$ (so we do not include D-tails). We let $\ov C$ be obtained from $\Gamma$ by deleting outer chains and D-tails, and contracting inner chains. Focusing on a component $\ov C_i^\nu \subseteq \ov C^\nu$, we have the data of its multiplicity $m_i$, and for each point $p_{i,v} \in \ov C_i^\nu$ whose image in $\ov C$ is a node, then its image in $\Gamma$ is also a node, so we obtain the numbers $\mathfrak{o}_{i,v}$ and $a_{i,v}$ as in \Cref{Dokchitser-main-result-definitions} and $j_{i,v}$ as in \Cref{Dokchitser-main-result}. Since $m_i$ is coprime to the characteristic of $\chr K$ by hypothesis and $\sum \mathfrak{o}_{i,v}$, \Cref{existence-cyclic-galois-cover} tells us that there is a possibly disconnected curve $C_i^\nu$ and an order $m_i$ automorphism $f_i \in \Aut(C_i^\nu)$ such that $\ov C_i^\nu = C_i^\nu/f_i$, where the orbit of $p_{i,v}$ has size $h_{i,v}^\nu = \frac{m_i}{\gcd(\mathfrak{o}_{i,v},m_i)}$, and the actions on the tangent spaces over the point $p_{i,v}$ is given by the inverse of $a_{i,v}$ modulo $\frac{m_i}{h_{i,v}}$. If $\ov C^\nu \to \ov C$ glues $p_{i,v}$ to some other $p_{i',v'}$ of a different orbit, then by using \Cref{Dokchitser-main-result} (3) inductively along the chain we have $h_{i,v}^\nu = h_{i',v'}^\nu$, so we may glue the fibers over $p_{i,v}$ and $p_{i',v'}$ in $C_i$ and $C_{i'}$ respectively, in any way compatible with the actions $f_i$. If $p_{i',v'}$ belongs to the same orbit as $p_{i,v}$, we may also glue a point $p$ in the orbit to $f^{h_{i,v}^\nu/2}(p)$, as $h_{i,v}$ needs to be even in this case. We obtain a curve $C$ and an automorphism $f \in \Aut(C)$, and integers $j_p = j_{i,v}$ for each node $p$ in $\ov C$.

        If $g(\Gamma) = 0$ and $\gcd(\{m_i\}) = 1$, it follows from \Cref{components-equals-gcd} that $C$ is connected. If $g(\Gamma) \geq 1$, we'll show how to modify $(C,f)$ into $(C',f')$ which is connected and has an associated singular fiber with the same combinatorics.

        If there is some component in $\ov C_i$ with geometric genus $p_g(\ov C_i) \geq 1$, then by \Cref{existence-cyclic-galois-cover} we may replace $C_i$ with a connected curve. Once one fiber over $\ov C_i$ is connected, all of $C$ will be connected as well.

        If all components of $\ov C$ are possibly nodal but rational, then there must exist a loop $\ov C_1,\ldots,\ov C_k$ with $k \geq 1$, such that $\ov C_i$ intersects $\ov C_{i+1}$ but also $\ov C_k$ intersects $\ov C_1$ at some point $p$. Here, if $k = 1$, we mean that $\ov C$ is a rational curve and $p$ is a node in it. Let $D$ be obtained from $C$ normalizing the nodes above $p$, and let $f_D$ the induced automorphism. Write $D = \bigsqcup D_j$ where the $D_j$ are the $\ell$ isomorphic connected components of $D$, and order them such that $f_D(D_j) = D_{j+1}$ for $j=1,\ldots,\ell-1$. Note that $\ell$ necessarily divides $h_p$. Over $p$, let $q_{1,1},\ldots,q_{1,h_p}$ and $q_{2,1},\ldots,q_{2,h_p}$ be the orbits of points in $D$ chosen in some order induced by $f$, where $q_{1,j}$ belongs to $D_{j \!\mod \ell}$. Then we obtain $C'$ by glueing $q_{1,j}$ to $q_{2,j+1}$ for $j = 1,\ldots,\ell-1$ and $q_{1,\ell}$ to $q_{2,1}$. This glueing is $f_D$-equivariant, so we can glue it to an automorphism $f' \in \Aut(C')$. Clearly, $C'$ is connected, and $f'$ acts on tangent spaces in the same way as $f$ did. We may now replace $C$ with $C'$.

        Now that we know that $C$ is connected, we can use \Cref{weak-realization-theorem} with the data $(C,f,\{j_p\})$ to obtain a $\mu_n$-equivariant family of stable curves $X \to B'$, from which we quotient and resolve singularities to obtain a fibration $\ti X \to B$. It is easy to see using \Cref{Dokchitser-main-result} (4) - (6) that the central fiber $\ti X_0$ is combinatorially identical to $(\Gamma,\{m_i\})$.

        Finally, for an arbitrary $d$, we note that if $\ti X \to B$ has $\dim H^0(\ti X_0,\O_{\ti X_0}) = d$, then the Stein factorization $\ti X \to B'$ of $\ti X \to B$ has connected fibers, so $\dim H^0(\ti X_{0'},\O_{\ti X_{0'}}) = 1$. This implies that $B' \to B$ is totally ramified to order $d$, so the multiplicities of the curves in $\ti X_{0'}$ are obtained from the ones in $\ti X_0$ by dividing by $d$.

        Conversely, if $g(\Gamma) \geq 1$ and $d$ divides $\gcd(\{m_i\})$ or if $\gcd(\{m_i\}) = d$, we consider the fibration $\ti X \to B'$ obtained by applying the above construction to the combinatorial singular fiber $(\Gamma,\{\frac{m_i}{d}\})$ with $d = 1$, and composing the resulting fibration $\ti X \to B'$ with a map $B' \to B$ which is generically $d:1$ and completely ramified at $0'$.
    \end{proof}

    \printbibliography
    
\end{document}